\documentclass[11pt]{amsart}

\usepackage[colorlinks=true,linkcolor=blue,urlcolor=blue,citecolor=blue]{hyperref}
\usepackage[nameinlink]{cleveref}

\usepackage{amsmath}
\usepackage{amssymb}
\usepackage{amsthm}
\usepackage{latexsym}
\usepackage{graphicx}
\usepackage{enumerate}
\usepackage[all]{xy}
\usepackage{chngcntr}
\usepackage{relsize}

\newtheorem{thm}{Theorem}[section]
\newtheorem{cor}[thm]{Corollary}
\newtheorem{lem}[thm]{Lemma}
\newtheorem{prop}[thm]{Proposition}

\newtheorem{thmintro}{Theorem}

\newtheorem{ass}[thm]{Hypothesis}

\newcommand{\N}{\mathbb N}
\newcommand{\Z}{\mathbb Z}
\newcommand{\Q}{{\mathbb Q}}
\newcommand{\R}{\mathbb R}
\newcommand{\C}{\mathbb C}
\newcommand{\F}{\mathbb F}
\newcommand{\mf}{\mathfrak}
\newcommand{\mc}{\mathcal}
\newcommand{\mb}{\mathbf}
\newcommand{\mh}{\mathbb}
\newcommand{\mr}{\mathrm}
\newcommand{\ind}{\mathrm{ind}}

\newcommand{\enuma}[1]{\begin{enumerate}[\textup{(}a\textup{)}] {#1} \end{enumerate}}
\newcommand{\Fr}{\mathrm{Frob}}
\newcommand{\Sc}{\mathrm{Sc}}

\newcommand{\cusp}{\mathrm{cusp}}
\newcommand{\nr}{\mathrm{nr}}
\newcommand{\unr}{\mathrm{unr}}

\newcommand{\cpt}{\mathrm{cpt}}
\newcommand{\der}{\mathrm{der}}
\newcommand{\Irr}{{\mathrm{Irr}}}
\newcommand{\Rep}{\mathrm{Rep}}
\newcommand{\Res}{\mathrm{Res}}

\newcommand{\enh}{\mathrm{e}}
\newcommand{\mini}{\mathrm{min}}
\newcommand{\Jord}{\mathrm{Jord}}

\newcommand{\red}{\mathrm{red}}
\newcommand{\matje}[4]{\left(\begin{smallmatrix} #1 & #2 \\ 
#3 & #4 \end{smallmatrix}\right)}
\newcommand{\Mod}{\mathrm{Mod}}

\newcommand{\Hom}{\mathrm{Hom}}
\newcommand{\End}{\mathrm{End}}

\newcommand{\op}{\mathrm{op}}

\newcommand{\GL}{\mathrm{GL}}
\newcommand{\rO}{\mathrm{O}}
\newcommand{\rS}{\mathrm{S}}
\newcommand{\SL}{\mathrm{SL}}
\newcommand{\SO}{\mathrm{SO}}
\newcommand{\GSO}{\mathrm{GSO}}
\newcommand{\GO}{\mathrm{GO}}
\newcommand{\GSpin}{\mathrm{GSpin}}
\newcommand{\GPin}{\mathrm{GPin}}
\newcommand{\rN}{\mathrm{N}}
\newcommand{\Sp}{\mathrm{Sp}}
\newcommand{\GSp}{\mathrm{GSp}}
\newcommand{\rU}{\mathrm{U}}
\newcommand{\rZ}{\mathrm{Z}}

\newcommand{\isom}{\xrightarrow{\;\sim\;}}

\begin{document}

\title{Affine Hecke algebras for classical $p$-adic groups}
\author[A.-M. Aubert]{Anne-Marie Aubert}
\address{Sorbonne Universit\'e and Universit\'e Paris Cit\'e, CNRS,
IMJ-PRG, F-75005 Paris, France}
\email{anne-marie.aubert@imj-prg.fr}
\author[A. Moussaoui]{Ahmed Moussaoui}
\address{Laboratoire de Math\'ematiques et Applications, Universit\'e de Poitiers, 
11 Boulevard Marie et Pierre Curie,
B\^{a}timent H3 - TSA 61125, 86073 Poitiers Cedex 9, France}
\email{ahmed.moussaoui@math.univ-poitiers.fr}
\author[M. Solleveld]{Maarten Solleveld}
\address{IMAPP, Radboud Universiteit Nijmegen, Heyendaalseweg 135,
6525AJ Nijmegen, the Netherlands}
\email{m.solleveld@science.ru.nl}
\date{\today}

%\enlargethispage{12mm}

\maketitle

%\vspace{-8mm}

\begin{abstract}
We consider four classes of classical groups over a non-archimedean local field $F$:
symplectic, (special) orthogonal, general (s)pin and unitary. These groups need not be
quasi-split over $F$. The main goal of the paper is to obtain a local Langlands correspondence 
for any group $G$ of this kind, via Hecke algebras.

To each Bernstein block Rep$(G)^{\mf s}$ in the category Rep$(G)$ of smooth com\-plex $G$-representations,
an (extended) affine Hecke algebra $\mc H (\mf s)$ can be associated with the method of Heiermann.
On the other hand, to each Bernstein component $\Phi_\enh (G)^{\mf s^\vee}$ of the
space $\Phi_\enh (G)$ of enhanced $L$-parameters for $G$, one can also associate an (extended) affine 
Hecke algebra, say $\mc H (\mf s^\vee)$. For the supercuspidal representations underlying 
Rep$(G)^{\mf s}$, a local Langlands correspondence is available via endoscopy, due to M\oe glin 
and Arthur. Using that we assign to each Rep$(G)^{\mf s}$ a unique $\Phi_\enh (G)^{\mf s^\vee}$. 

Our main new result is an algebra isomorphism $\mc H (\mf s)^{\op} \cong \mc H (\mf s^\vee)$, 
canonical up to inner automorphisms. In combination with earlier work, that provides an injective
local Langlands correspondence Irr$(G) \to \Phi_\enh (G)$ which satisfies Borel's desiderata. When
$F$ has characteristic zero, this parametrization map is in fact bijective. When $F$ has 
positive characteristic it is probably bijective as well, but we could not show that in all cases.

Our framework is suitable to (re)prove many results about smooth $G$-represen\-tations (not 
necessarily irreducible), and to relate them to the geometry of a space of $L$-parameters.
In particular our Langlands parametrization yields an independent way to classify discrete series
$G$-representations in terms of Jordan blocks and supercuspidal representations of Levi subgroups.
We show that it coincides with the classification of the discrete series obtained twenty years
ago by M\oe glin and Tadi\'c.
\end{abstract}

\tableofcontents

\section*{Introduction}

In the theory of linear algebraic groups, the classical groups play a special role.
As the stabilizer groups of bilinear/hermitian forms, they can arise from many 
directions and have various applications. Within the representation theory of reductive $p$-adic
groups, the main advantage of classical groups is their explicit structure. It enables
precise, combinatorial methods to study representations, on a level which is hard to reach
for other reductive groups. Such methods have been pursued by many mathematicians, see for 
instance \cite{Art,GGP,Hei4,KiMa,MoTa}.

In this paper we translate the (smooth, complex) representation theory of classical $p$-adic 
groups to affine Hecke algebras arising from Langlands parameters. This is part of a long-term
program \cite{AMS1,AMS2,AMS3} that applies to all reductive $p$-adic groups and aims to
establish instances of Langlands correspondences via Hecke algebras. The method has already
proven successful for principal series representations of split groups \cite{ABPS3} and for
unipotent representations \cite{SolUnip}. For classical groups, our Hecke algebra methods
provide alternative proofs of many earlier results (e.g. the classification of discrete
series representations) and install a framework in which one can easily establish many
new results that involve categories of smooth representations. \\

Let $F$ be any non-archimedean local field ($p$-adic or a local function field). We will
consider classical $F$-groups in a broad sense, namely
\begin{itemize}
\item symplectic groups;
\item (special) orthogonal groups associated to symmetric bilinear forms on a finite dimensional
$F$-vector space $V$; 
\item general (s)pin groups associated to such bilinear forms;
\item unitary groups associated to hermitian forms on vector spaces over a separable 
quadratic extension of $F$.
\end{itemize}
We stress that these groups do not have to be quasi-split, we allow pure inner forms.
For $G = \SO (V)$ and $G = \GSpin (V)$ we write, respectively, $G^+ = \rO(V)$ and 
$G^+ = \GPin (V)$, otherwise $G^+ = G$. An advantage of including general spin groups is 
that they provide information about all representations of spin groups, including those 
that do not factor through special orthogonal groups.

General linear groups could also figure in the list, they are very classical (but note that they
do not come from a nondegenerate bilinear form). We excluded them because for $\GL_n (F)$ everything
we will discuss has been known for a long time already, see \cite{HaTa,Hen, Sch,LRS} for the local Langlands correspondence and 
\cite{BuKu,AMS1} for the Hecke algebras.

The common feature of all the above groups $G$ is that their Levi subgroups are isomorphic to
$G' \times \GL_{n_1}(F') \times \cdots \times \GL_{n_k}(F')$, where $G'$ is a group in the same 
family as $G$ but of smaller rank, and $F' = F$ unless $G$ is a unitary group, then $[F' : F] = 2$.
It is this structure which enables the aforementioned ``combinatorial" approach to representations
of classical groups. In a sense that approach is recursive, relating $G$-representations to
similar groups of smaller rank and to representations of $\GL_n (F)$, which are understood well already.
However, such a reduction strategy does not say much about supercuspidal $G$-representations.
The crucial technique to analyse those is endoscopy, as in \cite{Art,Mok,KMSW,MoRe}. From the work
of Arthur and M\oe glin, the following version of a local 
Langlands correspondence (for the discrete objects) can be distilled.

\begin{thmintro} \label{thm:A}
\textup{(Arthur, M{\oe}glin, see Theorem\autoref {thm:1.1})} \ \\
Let $F$ be a $p$-adic field and let $G$ be one of the connected classical groups listed above.
\enuma{
\item Let $\pi$ be a discrete series representation of $G^+$. Then the $L$-parameter of $\pi$ can
be obtained from the set of Jordan blocks of $\pi$, by taking the $L$-parameters of all
$\GL_n (F)$-representations in $\mr{Jord}(\pi)$ and combining those via block-diagonal matrices.
\item Pick a Whittaker datum for the quasi-split inner form of $G$. This determines an extension
of part (a) to an injection from the discrete series of $G^+$ to the set of enhanced bounded 
discrete $L$-parameters for $G$ (where the component groups of $L$-parameters are computed 
in the possibly disconnected group $G^{\vee+}$).
\item When $G^+ \neq G$, it can be described explicitly in terms of $\mr{Jord}(\pi)$ whether or
not $\mr{Res}^{G^+}_G (\pi)$ is irreducible.
}
\end{thmintro}

M\oe glin characterized also supercuspidality in the context of Theorem~\ref{thm:A}, both for
$G$-representations and for enhanced $L$-parameters. We refer to Section~\ref{sec:Moe} for the
notations and more background. For now, we 
make a couple of remarks to aid the correct interpretation of Theorem~\ref{thm:A}. Firstly, note that 
in part (b) no bijectivity is claimed, although that is known for many of these groups. Secondly, we 
have to warn that not all details of the proof of Theorem~\ref{thm:A} have been worked out (we ourselves 
did not try, we only provide the relevant references). Further, Theorem~\ref{thm:A} relies heavily on 
endoscopy, that is the reason why $F$ needs to have characteristic zero.

Nevertheless, Theorem~\ref{thm:A} should also hold for classical groups over local function fields,
see \cite{GaVa,GaLo} for some instances. In Paragraph~\ref{par:close} we attempt to derive that with the 
method of close local fields. We managed to prove that in Proposition\autoref {prop:1.9}, under Hypothesis 
\ref{as:1.8} on depths of representations in Jordan blocks (the hypothesis most probably holds always).
Unfortunately our arguments do not suffice to prove surjectivity in Theorem~\ref{thm:A}.b in complete
generality for classical groups over local function fields, even if we would know such surjectivity 
for the analogous groups over $p$-adic fields.

For the purposes of this paper, we only need to know Theorem~\ref{thm:A} for supercuspidal 
$G^+$-representations. Indeed, the remainder of Theorem~\ref{thm:A} follows from those cases with
either \cite{Moe0,MoTa} or with our results discussed below and the detailed know\-ledge of the 
discrete series of the Hecke algebras from \cite{AMS2,AMS3}. Consequently all results in paper 
hold for $G$ and $G^+$ as soon as we know Theorem~\ref{thm:A} for supercuspidal representations 
of $G^+$ and of the groups of smaller rank in the same family.\\

Next we discuss our new results. An important aspect is their canonicity, by which
we mean that they do not depend on arbitrary choices. For results that are not entirely
canonical, we indicate the freedom in the choices.

Recall that the category of smooth complex representations of any classical 
$F$-group $G$ admits the Bernstein decomposition
\begin{equation}\label{eq:4}
\Rep (G) = \prod\nolimits_{\mf s} \Rep (G)^{\mf s} ,
\end{equation}
indexed by inertial equivalence classes $\mf s$. Each such $\mf s$ is a $G$-conjugacy class of
pairs $(L,X_\nr(L)\cdot\sigma)$, where $\sigma$ is an irreducible supercuspidal representation 
of a Levi subgroup $L$ of $G$ and $X_\nr(L)$ is the group of unramified characters of $L$.
Every Bernstein block $\Rep (G)^{\mf s}$ is equivalent with the category of right modules of some
finitely generated algebra $\mc H (\mf s)$, often an affine Hecke algebra. Usually these Hecke algebras
arise via types (in the sense of Bushnell--Kutzko). For classical groups such types are indeed
available \cite{MiSt}, but it has turned out to be difficult to analyse the Hecke algebras via those 
types. Instead we follow the approach of Heiermann \cite{Hei2,Hei3,Hei4}, who constructed $\mc H (\mf s)$
as the $G$-endomorphism algebra of a canonical progenerator $\Pi_{\mf s}$ of $\Rep (G)^{\mf s}$. 
(For $\GSpin (V)$ we use the more general results from \cite{SolEnd,SolParam}.) We recall that 
by design $\Hom_G (\Pi_{\mf s},?)$ provides an equi\-va\-lence of categories
\begin{equation}\label{eq:1}
\Rep (G)^{\mf s} \cong \Mod(\mc H (\mf s)^{\op}) = \Mod(\End_G (\Pi_{\mf s})^{\op}).
\end{equation}
Here $A^{\op}$ denotes the opposite algebra of $A$, so that $\Mod (A^{\op})$ is the category of
right $A$-modules.
These algebras $\mc H (\mf s)$ have been described explicitly in terms of the Jordan blocks of the 
underlying supercuspidal representations (of a Levi subgroup of $G$). They are extended affine
Hecke algebras \cite{Hei3,SolEnd}. That links them to Theorem
\ref{thm:A} and hence to Langlands parameters. These links were investigated in \cite{Hei4}, where it
was shown that each $\Rep (G)^{\mf s}$ equivalent to a Bernstein block of unipotent representations
in another group. Unfortunately these equi\-va\-lences are far from canonical, as most of the comparison
steps in \cite{Hei4} involve arbitrary choices. Also, the back-and-forth between various Langlands
parameters and Hecke algebras entails that in \cite{Hei4} there is no construction of representations in
$\Rep (G)^{\mf s}$ with an unambiguous relation to Langlands parameters for $G$.

Fortunately, objects of the above kinds are also available directly for $L$-parameters. Indeed, 
in \cite[\S 8]{AMS1} the space of enhanced $L$-parameters (of any connected reductive $p$-adic group $G$) 
is partitioned into Bernstein components:
\begin{equation}\label{eq:5}
\Phi_\enh (G) = \bigsqcup\nolimits_{\mf s^\vee} \Phi_\enh (G)^{\mf s^\vee} ,
\end{equation}
indexed by inertial equivalence classes $\mf s^\vee$ for $\Phi_e (G)$. By definition each
$\mf s^\vee$ is a $G^\vee$-conjugacy class of triples $({}^L M,X_\nr({}^L M)\cdot \phi, 
\epsilon)$, where ${}^L M$ is a $L$-Levi subgroup of ${}^L G$, $X_\nr({}^L M)$ denotes the 
analogue of unramified characters for ${}^L M$ and $(\phi,\epsilon) \in \Phi_\cusp (M)$ 
is a cuspidal enhanced $L$-parameter, see \cite[7.1]{AMS1}.

To each such Bernstein component, one can associate a twisted affine Hecke algebra $\mc H (\mf s^\vee,
\mb z)$ \cite{AMS3}. Here $\mb z$ is an invertible indeterminate, analogous to $\sqrt{\mb q}$ for
Iwahori--Hecke algebras. Two important features of $\mc H (\mf s^\vee, \mb z)$ were established in
\cite{AMS2,AMS3}: a construction of (irreducible) representations in terms of the geometry of a space
of Langlands parameters and for each $z \in \R_{>0}$ a canonical bijection 
\begin{equation}\label{eq:2}
\Phi_\enh (G)^{\mf s^\vee} \longleftrightarrow \Irr (\mc H (\mf s^\vee,z)) ,
\end{equation}
where $\mc H (\mf s^\vee,z)$ denotes the specialization of $\mc H (\mf s^\vee,\mb z)$ at $\mb z = z$.
Moreover, for $z > 1$ the bijection \eqref{eq:2} sends bounded 
parameters and discrete parameters to the expected kind of representations (respectively tempered 
and essentially discrete series). Later we will specialize $\mb z$ to $q_F^{1/2}$, where $q_F$ denotes 
the cardinality of the residue field of $F$. The algebras $\mc H (\mf s^\vee, q_F^{1/2})$ are crucial,
without them it is hardly possible to make the relations between $\Phi_\enh (G)$ and $\Rep (G)$ canonical.

In Paragraph~\ref{par:HAL} we make the affine Hecke algebras $\mc H (\mf s^\vee, \mb z)$ completely 
explicit, for any Bernstein component of $\Phi_\enh(G)$ with $G$ a classical $F$-group. This involves
a description of the underlying root datum and of the labels (equivalently: the $q$-parameters) in 
terms of the relevant Jordan blocks. We refer to Table~\ref{tab:2} for an overview.

Theorem~\ref{thm:A} enables us to associate to each Bernstein block $\Rep (G)^{\mf s}$ a unique
Bernstein component of $\Phi_\enh (G)$ which we call $\Phi_\enh (G)^{\mf s^\vee}$, see Theorem\autoref {thm:1.5}.
When $G^+ \neq G$ (so for special orthogonal groups and general spin groups), $\mf s^\vee$ is only
canonical up to the action of the two-element group $\mr{Out}(\mc G)$. Our most important result is 
a comparison of Hecke algebras on the two sides of the local Langlands correspondence:

\begin{thmintro}\label{thm:B}
\textup{(see Theorem\autoref{thm:1.7}, Proposition\autoref{prop:3.2} and Proposition\autoref {prop:4.4})} \\
Let $G$ be a connected classical group over $F$ and fix a Whittaker datum for the quasi-split
inner form of $G$. Suppose that the inertial equivalence classes $\mf s$ and $\mf s^\vee$ are 
matched as in Theorem\autoref{thm:1.5}. There exists an algebra isomorphism 
\[
\mc H (\mf s ) \cong \mc H (\mf s^\vee, q_F^{1/2})
\]
with the following properties.
\begin{itemize}
\item On the standard maximal commutative subalgebras $\mc O (T_{\mf s}) \subset \mc H (\mf s)$ and\\
$\mc O (T_{\mf s^\vee}) \subset \mc H (\mf s^\vee, q_F^{1/2})$, the isomorphism is prescribed by the
$L$-parameters of supercuspidal representations from Theorem~\ref{thm:A} and from the local Langlands correspondence for general
linear groups.
\item There is a canonical bijection between the root system associated to $\mf s$ and the
root system associated to $\mf s^\vee$.
\item The isomorphism is canonical up to conjugation by elements of $\mc O (T_{\mf s})^\times$ and
(in the cases with $G^+ \neq G$) up to the action of $\mr{Out}(\mc G)$.
\end{itemize}
There exists an analogous isomorphism of Hecke algebras for $G^+$, which is canonical up to 
conjugation by elements of $\mc O (T_{\mf s})^\times$.

Further, the algebra $\mc H (\mf s^\vee, q_F^{1/2})$ is canonically isomorphic to its own opposite.
\end{thmintro}

We remind the reader that in the case of classical groups over local function fields we need the
mild Hypothesis~\ref{as:1.8} for Theorem~\ref{thm:B} (and hence for most subsequent results).
The Whittaker datum is needed for the canonicity of the above Hecke algebra isomorphism. That is 
a subtle affair, it relies on a normalization of certain intertwining operators in 
Paragraph\autoref {par:intertwining}, which in the end boils down to \cite{Art}.

As a direct consequence of Theorem~\ref{thm:B} and \eqref{eq:1} we find an equivalence of categories
\begin{equation}\label{eq:3}
\Rep (G)^{\mf s} \cong 
\Mod \big( \mc H (\mf s^\vee, q_F^{1/2}) \big),
\end{equation}
which is canonical up to inner automorphisms of $\mc H (\mf s^\vee, q_F^{1/2})$ (and up to the 
action of $\mr{Out}(\mc G)$ when $G^+ \neq G$). The analogous equivalence of categories for $G^+$ 
is canonical up to inner automorphisms of the involved Hecke algebra.

With the Bernstein decompositions \eqref{eq:4} and \eqref{eq:5}, \eqref{eq:3} makes $\Rep (G)$
equivalent to the module category of a direct sum of Hecke algebras.  
In combination with \eqref{eq:2} we obtain:

\begin{thmintro}\label{thm:C}
\textup{(see Theorem\autoref {thm:3.4} and Corollary\autoref {cor:3.5})} \\
Theorem~\ref{thm:A} (for supercuspidal representations) and Theorem~\ref{thm:B} 
induce an injective local Langlands correspondence
\[
\Irr (G) \hookrightarrow \Phi_\enh (G) .
\]
It is canonical (up to the action of $\mr{Out}(\mc G)$ when $G^+ \neq G$) and it sends supercuspidal/
essentially square-integrable/tempered representations to cuspidal/discrete/bounded enhanced
$L$-parameters.

There exists an analogous parametrization of $\Irr (G^+)$, which uses component groups of
$L$-parameters computed in $G^{\vee +}$ and is entirely canonical.
\end{thmintro}

We note that \eqref{eq:3} is much stronger than any results about the parametrization of $\Irr (G)$,
in the sense that it deals with an entire category of representations. Indeed, earlier results about 
Hecke algebras entail that \eqref{eq:3} has various consequences that 
involve reducible representations, see Paragraph~\ref{par:LLC}. Furthermore the equivalence of 
categories \eqref{eq:3} makes it possible to relate $\Rep (G)$ to the complex geometry of the 
space/stack of $L$-parameters, as in \cite{SolKL}.

In Theorem~\ref{thm:C} we do not claim surjectivity of the parametrization map, because for that
we would need surjectivity in Theorem~\ref{thm:A}.b, which we do not know when $F$ is a local
function field. That is in fact the only obstruction: the image of the map in Theorem~\ref{thm:C} is 
the union of all Bernstein components of $\Phi_\enh (G)$ whose underlying cuspidal $L$-parameters can be 
reached via Theorem~\ref{thm:A}. So in all the cases where the surjectivity of Theorem~\ref{thm:A}.b 
has been proven, we also get surjectivity in Theorem~\ref{thm:C}.

Theorem~\ref{thm:C} yields in particular a classification of the discrete series of $G^+$, in terms
of the bounded discrete enhanced $L$-parameters in the image of the parametrization map.
On the other hand, Theorem~\ref{thm:A} also classifies discrete series representations of $G^+$.
For supercuspidal representations these two methods agree, that is a starting point of our setup.
We obtain two independent ways to classify the discrete series in terms of supercuspidal
representations of Levi subgroups: with Hecke algebras via Theorem~\ref{thm:C} and with Jordan
blocks as in \cite{Moe0,MoTa,KiMa}. 

Moreover both methods can be pushed further, to classify all irreducible smooth $G^+$-representations.
Indeed, in Theorem~\ref{thm:C} that comes at the same time as the discrete series (in the
underlying proofs from \cite{AMS2} the discreteness of representations is actually analysed last).
Irreducible tempered $G^+$-representations are classified with endoscopy and Jordan blocks in
\cite{MoTa,MoRe}. The step from tempered representations to all irreducible smooth representations
via the Langlands classification is well-known for $G$ and we show how this technique can be made
to work for $G^+$. With that extension the papers \cite{MoTa,MoRe} also classify $\Irr (G^+)$.

\begin{thmintro}\label{thm:D}
\textup{(see Theorem\autoref {thm:4.12})} \\
Let $G$ be a connected classical group over $F$ and fix a Whittaker datum for the quasi-split
inner form of $G$.
The following two ways to parametrize $\Irr (G^+)$ with enhanced $L$-parameters coincide:
\begin{itemize}
\item with Hecke algebras as in Theorem~\ref{thm:C},
\item with endoscopy and Jordan blocks (as in Theorem~\ref{thm:A}) 
and the Langlands classification.
\end{itemize}
\end{thmintro}

Because the two strategies are very different, it is rather cumbersome to check that they agree.
We do this step by step in Section~\ref{sec:Langlands}, in the following order: completely positive
discrete series, all discrete series, irreducible tempered, all irreducible representations.
The most difficult part concerns the enhancements of $L$-parameters for discrete series
representations. To match those for the two methods in Theorem~\ref{thm:D}, we have to
investigate normalizations of intertwining operators. That also informs us how to make 
Theorem~\ref{thm:B} canonical.\\

A few words about the setup of the paper are in order. As we mentioned at the start of the
introduction, we consider four classes of classical groups. For all classes the proofs of
our results are extremely similar, yet not entirely the same. For symplectic and (special)
orthogonal groups, almost everything that we show about Hecke algebras was known already,
from \cite{Hei2,Hei3} for $p$-adic groups and from \cite{Mou,AMS3} for Langlands parameters.

Instead we focus on general (s)pin groups in Sections\ref{sec:GSpin} and \ref{sec:Hecke}. That
is a little bit more involved because the center $\rZ(G)$ of such a group $G$ is not compact.
The proofs for symplectic and (special) orthogonal groups can be recovered by restricting
from $G^\vee$ to its derived group, where $G^\vee$ denotes the complex reductive group with root 
datum dual to that of $G$. Sections \ref{sec:Moe} and \ref{sec:Langlands} are written
so that they apply equally well to symplectic, (special) orthogonal and general (s)pin groups.

For unitary groups, the necessary changes affect the notations so much that we discuss
them in the separate Section~\ref{sec:unitary}. We check carefully which modifications are
needed to make Sections\ref{sec:GSpin}--\ref{sec:Langlands} work for unitary groups. It turns 
out that for unramified unitary groups some calculations in Paragraph~\ref{par:HAL} have different outcomes, which we record.

\renewcommand{\theequation}{\arabic{section}.\arabic{equation}}
\counterwithin*{equation}{section}

\section{General spin groups}
\label{sec:GSpin}

Let $F$ be a non-archimedean local field with absolute Weil group $\mb W_F$. 
Consider a finite dimensional $F$-vector space $V$ endowed with a symmetric bilinear form. 
The associated general pin group, denoted GPin$(V)$, contains the general spin group 
$\GSpin (V)$ with index two. Both are subgroups of the multiplicative group of the Clifford 
algebra of $V$. For the root datum of $\GSpin (V)$ we refer to \cite[\S 2]{AsSh}. 
Simultaneously we consider the groups $\GSpin (V')$, where $\dim (V') = \dim (V)$ and
$\mr{disc}(V) = \mr{disc}(V')$. The equivalence classes of such groups are naturally in 
bijection with:
\begin{itemize}
\item equivalence classes of symmetric bilinear forms of the same dimension and the 
same discriminant as $V$,
\item pure inner twists of $\SO (V)$,
\item rigid inner twists of $\GSpin (V)$ \cite{Kal}, with respect to the central subgroup
\[
\{\pm 1\} = \ker \big( \mr{Spin}(V) \to \SO (V) \big) = 
\ker \big( \GSpin (V) \to \SO(V) \times F^\times \big) .
\]
\end{itemize}
To align with the other classical groups, we refer to the above groups as the pure inner 
twists of $\GSpin (V)$. Let us list all the possibilities:
\begin{itemize}
\item for $\dim = 2n+1$, the split group $\GSpin_{2n+1}(F)$ of $F$-rank $n$ and one
pure inner twist $\GSpin'_{2n+1}(F)$ of $F$-rank $n \!-\! 1$,
\item for $\dim = 2n$, the split group $\GSpin_{2n}(F)$ of $F$-rank $n$ and one
pure inner twist $\GSpin'_{2n}(F)$ of $F$-rank $n \!-\! 1$,
\item for $\dim = 2n$, the quasi-split group $\GSpin^*_{2n}(F)$ of $F$-rank $n-1$ and one
pure inner twist $\GSpin^{*'}_{2n+1}(F)$, which is also quasi-split.
\end{itemize}
For any of these groups $G$, we write
\begin{equation}\label{eq:G+}
G^+ = \left\{
\begin{array}{ll}
\mr{GPin}(V) & \text{if } \dim V \text{ is even} \\
\GSpin (V) & \text{if } \dim V \text{ is odd}
\end{array}\right. .
\end{equation}
All inner twists share the same Langlands dual group, so for that we have precisely
three possibilities:
\begin{itemize}
\item $\GSpin_{2n+1}^\vee = \GSp_{2n}(\C)$, and since one of the $p$-adic groups is split
we may take ${}^L \GSpin_{2n+1} = \GSp_{2n}(\C)$,
\item $\GSpin_{2n}^\vee = \GSO_{2n}(\C)$, and again one of the $p$-adic groups is split
so we take ${}^L \GSpin_{2n} = \GSO_{2n}(\C)$,
\item $\GSpin_{2n}^{*\vee} = \GSO_{2n}(\C)$, and $\mb W_F$ acts on it via passing to a quotient
$\mb W_F / \mb W_E$ of order two and then conjugation by an element of $\mr{O}_{2n}(\C)
\setminus \SO_{2n}(\C)$. We may take ${}^L \GSpin_{2n}^* = \GO_{2n}(\C)$,
where we remember that every Langlands parameter for $\GSpin_{2n}^* (F)$ sends $\mb W_E$ to
$\GSO_{2n}(\C)$ and $\mb W_F \setminus \mb W_E$ to $\GO_{2n}(\C) \setminus \GSO_{2n} (\C)$.
\end{itemize}
We write ${}^L G_n$ or ${}^L G$ for  ${}^L \GSpin_{2n+1}, {}^L \GSpin_{2n}$ or 
${}^L \GSpin_{2n}^*$. We also write
\[
G^{\vee +} = \left\{
\begin{array}{l}
\GO_{2n}(\C) \\
\GSp_{2n}(\C) 
\end{array}\right. ,\quad
{G^\vee}^+_\der = \left\{
\begin{array}{llll}
\rO_{2n}(\C) & \text{if } \dim V = 2n \\
\Sp_{2n}(\C) & \text{if } \dim V = 2n+1
\end{array}\right. .
\]
Langlands parameters for $G$ are group homomorphisms $\phi : \mb W_F \times \SL_2 (\C) \to {}^L G$,
with the usual requirements as for instance in \cite[Definition 6.4]{AMS1}. Considered up to
$G^\vee$-conjugation, these form the set $\Phi (G)$. Langlands parameters for $G^+$ take values in 
${}^L G$ and are considered up to conjugation by $G^{\vee +}$.

\subsection{Properties of Langlands parameters} \
\label{par:properties}

Let us investigate when a Langlands parameter $\phi$ for $G$ is discrete.
The image of $\phi$ is contained in ${}^L G$, so in $\GSp_{2n}(\C)$ or in $\GO_{2n}(\C)$.
In the former case $\phi$ is an $L$-parameter for $\GSpin_{2n+1}(F)$ or $\GSpin'_{2n+1}(F)$,
in the latter case for a general spin group of even size. We can distinguish two subcases:
\begin{itemize}
\item when im$(\phi) \subset \GSO_{2n}(\C)$, where $\phi$ is an $L$-parameter for
$\GSpin_{2n}(F)$ or $\GSpin'_{2n}(F)$,
\item otherwise there is an index two subgroup $\mb W_E \subset \mb W_F$ such that\\
$\phi (\mb W_E \times \SL_2 (\C)) \subset \GSO_{2n}(\C)$. Then $\phi$ is an $L$-parameter
for a group $\GSpin^*_{2n}(F)$ or $\GSpin^{*'}_{2n}(F)$ which splits over $E$.
\end{itemize}
We suppose that the bilinear form $B_J$ on $\C^{2n}$ from which $G^\vee$ is defined is given by 
a (skew-)symmetric matrix $J \in \GL_{2n}(\C)$. Let $\mu_G^\vee \colon{}^L G \to \C^\times$
be the similitude character, that is
\begin{equation}\label{eq:1.1}
B_J (g v_1, g v_2) = \mu_G^\vee (g) B_J (v_1, v_2) \qquad v_1, v_2 \in \C^{2n}, g \in {}^L G.
\end{equation}
Recall that \eqref{eq:1.1} holds for $g = \phi (w)$ with $w \in \mb W_F \times \SL_2 (\C)$. 
Hence the map
\[
\begin{array}{cccc}
\tilde{B_J} : & \C^{2n} & \to & (\C^{2n})^\vee \\
 & v & \mapsto & [v' \mapsto B_J (v',v) ]
\end{array}
\]
provides an isomorphism of $\mb W_F \times \SL_2 (\C)$-representations 
\begin{equation}\label{eq:1.51}
\phi \isom \phi^\vee \otimes \mu_G^\vee \circ \phi \qquad \text{or equivalently} \qquad
\phi \otimes (\mu_G^\vee \circ \phi)^{-1} \isom \phi^\vee .
\end{equation}
Here $\phi^\vee$ denotes the contragredient of $\phi$. The adjoint map
\[
\tilde{B_J}^\vee : \phi \isom \big( \phi \otimes (\mu_G^\vee \circ \phi)^{-1} \big)^\vee
= \phi^\vee \otimes \mu_G^\vee \circ \phi 
\] 
is also an isomorphism of $\mb W_F \times \SL_2 (\C)$-representations. Suppose that $V_1$
is an irreducible subrepresentation of $(\phi,\C^{2n})$, on which $B_J$ is nondegenerate.
By Schur's lemma there exists $c_1 \in \C^\times$ such that $\tilde{B_J}^\vee |_{V_1} = 
c_1 \tilde{B_J} |_{V_1}$. Then
\[
\tilde{B_J} |_{V_1} = \tilde{B_J}^{\vee \vee} |_{V_1} = c_1 \tilde{B_J}^\vee |_{V_1} =
c_1^2 \tilde{B_J} |_{V_1} ,
\]
so $c_1 \in \{1,-1\}$. This says that $(V_1,B_J)$ has a well-defined sign $c_1$.

Since ${}^L G = \C^\times {}^L G_\der$ and ${}^L G_\der = \Sp_{2n}(\C)$ or
${}^L G_\der \subset\rO_{2n}(\C)$, the decomposition of $(\phi,\C^{2n})$ in irreducible
subrepresentations can be carried out just like for orthogonal or symplectic representations.
For those kinds of representations we use the instructive paper \cite{GGP}. Thus we decompose
\begin{equation}\label{eq:1.2}
(\phi,\C^{2n}) = 
\bigoplus\nolimits_{\psi \in \Irr (\mb W_F \times \SL_2 (\C))} N_\psi \otimes V_\psi ,
\end{equation}
where $V_\psi$ is the space of the representation $\psi$ and $N_\psi$ is the multiplicity space 
(with a trivial action). By \cite[Theorem 8.1]{GGP}
the right hand side of \eqref{eq:1.2} determines $\phi$ up to $G^\vee$-conjugacy, apart from
some exceptional cases in which it is up to $\GO_{2n}(\C)$-conjugacy. Further, by \cite[\S 4]{GGP}
$B_J$ induces bilinear a form on each $N_\psi$ and
\begin{align}\label{eq:1.3}
& \rZ_{{G^\vee}_\der}(\phi) := \rZ_{{G^\vee}_\der}(\phi(\mb W_F\times\SL_2(\C))) = \\
\nonumber \rS \big( \prod\nolimits_{\psi \in I^+} \rO(N_\psi) \otimes & \mr{Id}_{V_\psi} \big) 
\times \prod\nolimits_{\psi \in I^-} \Sp (N_\psi) \otimes \mr{Id}_{V_\psi} \times
\prod\nolimits_{\psi \in I^0} \GL (N_\psi) \otimes \mr{Id}_{V_\psi \oplus V_\psi^\vee} ,
\end{align}
where $\rS(H)$ denotes the subgroup of elements in $H$ with determinant equal to $1$.
Here we abbreviated
\[
\begin{array}{lll}
I^\pm & = & \{ \psi \in \Irr (\mb W_F \times \SL_2 (\C)) \colon \psi \cong \psi^\vee \otimes 
\mu_G^\vee \circ \phi, \mr{sgn}(\psi) = \pm \mr{sgn}({G^\vee}_\der) \} , \\
I^0 & = & \{ \psi \in \Irr (\mb W_F \times \SL_2 (\C)) : \psi \not\cong \psi^\vee \otimes 
\mu_G^\vee \circ \phi \} / (\psi \sim \psi^\vee \otimes 
\mu_G^\vee \circ \phi ) .
\end{array}
\]
If $G^\vee = \GSO_{2n}(\C)$, then we may also consider \eqref{eq:1.3} with 
$G^{\vee+} = \GO_{2n}(\C)$ instead of $G^\vee$. That results in omitting the $S$ from 
the second line of \eqref{eq:1.3}.\\
Recall that $\phi$ is discrete if and only if $\rZ_{G^\vee}(\phi) / \rZ(G^\vee)^{\mb W_F}$ is finite,
which is equivalent to: $\rZ_{{G^\vee}_\der}(\phi)$ is finite. From \eqref{eq:1.3} we see that
that is the case if and only if 
\begin{equation}\label{eq:1.4}
N_\tau = 0 \text{ for } \tau \in I^- \cup I^0 \qquad \text{and} \qquad 
\dim (N_\tau) \leq 1 \text{ for } \tau \in I^+.
\end{equation}
From now on we assume that $\phi$ is discrete. Thus each $\tau \otimes P_a$ has multiplicity
at most one in $\phi$.

Recall that $\SL_2 (\C)$ has a unique irreducible representation $(P_a,\C^a)$ of dimension
$a \in \Z_{>0}$, and that it is self-dual with sign $(-1)^{a-1}$. Let Jord$(\phi)$ be the set of
pairs $(\tau,a) \in \Irr (\mb W_F) \times \Z_{>0}$ for which $\tau \otimes P_a$
occurs in $(\phi,\C^{2n})$. The set Jord$(\phi)$ describes the Jordan decomposition of the unipotent 
element $u_\phi = \phi (1,\matje{1}{1}{0}{1})$: for each $(\tau,a) \in \mr{Jord}(\phi)$, $u_\phi$
has $\dim \tau$ Jordan blocks of size $a$. We abbreviate
\[
\mr{Jord}_\tau (\phi) = \{ a \in \Z_{>0} : (\tau,a) \in \mr{Jord}(\phi) \} .
\]
We define
\[
\begin{array}{lll}
\Irr (\mb W_F)^\pm_\phi & = & \{ \tau \in \Irr (\mb W_F) : \tau \cong \tau^\vee \otimes 
\mu_G^\vee \circ \phi, \mr{sgn}(\tau) = \pm \mr{sgn}({G^\vee}_\der) \} , \\
\Irr (\mb W_F)^0_\phi & = & \{ \tau \in \Irr (\mb W_F) : \tau \not\cong \tau^\vee \otimes 
\mu_G^\vee \circ \phi \} / (\tau \sim \tau^\vee \otimes \mu_G^\vee \circ \phi ) .
\end{array}
\]
Then we can express \eqref{eq:1.2} more precisely as
\begin{equation}\label{eq:1.5}
(\phi,\C^{2n}) = \bigoplus_{\tau \in \Irr (\mb W_F)^+_\phi \hspace{-5mm}} \tau \otimes 
\left( \bigoplus_{a \; \mr{odd}: (\tau,a) \in \mr{Jord}(\phi)} P_a \right) \oplus 
\bigoplus_{\tau \in \Irr (\mb W_F)^-_\phi \hspace{-5mm}} \tau \otimes \left( 
\bigoplus_{a \; \mr{even}: (\tau,a) \in \mr{Jord}(\phi)} P_a \right) .
\end{equation}
The setup with rigid inner forms from \cite{Kal} entails that we should consider 
centralizers of $L$-parameters in
\[
(G / \{ \pm 1 \})^\vee = (\SO (V) \times \GL_1 (F))^\vee = \SO(V)^\vee \times \GL_1 (\C) .
\] 
The factor $\GL_1(\C)$ is central, connected and fixed by $\mb W_F$, so does not influence the
component groups and may be omitted. Therefore we may and will compute component groups of 
$L$-parameters for our pure inner twists of $\GSpin (V)$ in ${G^\vee}_\der = \SO (V)^\vee$, 
which equals $\SO_{2n}(\C)$ or $\Sp_{2n}(\C)$. We put
\begin{equation} \label{eqn:Sphi}
\mc S_\phi = \pi_0 (\rZ_{{G^\vee}_\der} (\phi)) \quad \text{and} \quad
\mc S_\phi^+ = \pi_0 (\rZ_{{G^\vee}^+_\der} (\phi)) .
\end{equation}
and we use the irreducible representations of $\mc S_\phi$ as enhancements of $\phi$. 
From \eqref{eq:1.3} we see that every $(\tau,a) \in \mr{Jord}(\phi)$ contributes a 
generator $z_{\tau,a}$ of order two to $\mc S_\phi^+$. Then $\mc S_\phi^+$ is the 
$\F_2$-vector space with these $z_{\tau,a}$ as basis. Here $z_{\tau,a}$ acts as $-1$ on
$\tau \otimes P_a$ and as 1 on the other summands of \eqref{eq:1.5}. The group $\mc S_\phi$
has index at most two in $\mc S_\phi^+$ and consists of all products of the $z_{\tau,a}$ such 
that the determinant is 1. Thus every element of $\mc S_\phi$ involves an even number of 
$z_{\tau,a}$ with $a \dim \tau$ odd.

A character $\epsilon$ of $\mc S_\phi$ is a $G$-relevant enhancement of $\phi$
if and only if $\epsilon$ restricted to $\rZ({G^\vee}_\der)^{\mb W_F}$ encodes $G$ via the 
Kottwitz isomorphism, i.e. it is quadratic if $G$ is a ``prime'' form (with notation as above
\eqref{eq:G+}) and trivial otherwise. Here the image of $\rZ({G^\vee}_\der)^{\mb W_F}$ in 
$\mc S_\phi$ is generated by
\begin{equation}\label{eq:1.40}
\prod\nolimits_{(\tau,a) \in \mr{Jord}(\phi) :\, a \dim \tau \text{ odd}} z_{\tau,a} ,
\end{equation}
which is an element of order at most two. We refine $\Phi (G)$ and $\Phi (G^+)$ to sets of 
enhanced L-parameters by
\begin{equation}\label{eq:1.80}
\begin{array}{lll}
\Phi_e (G) & = & \{ (\phi,\rho) : \phi \in \Phi (G), \rho \in \Irr (\mc S_\phi)
\text{ is $G$-relevant} \}, \\
\Phi_e (G^+) & = & \{ (\phi,\rho) : \phi \in \Phi (G), \rho \in \Irr (\mc S_\phi^+)
\text{ is $G$-relevant} \}.
\end{array}
\end{equation}
We want to make explicit which enhancements of $\phi$ are cuspidal. Like in \eqref{eq:1.5}
\begin{multline}\label{eq:1.6}
\rZ_{{G^\vee}_\der}(\phi (\mb W_F)) = \rS \Big(
\prod_{\tau \in \Irr (\mb W_F)^+_\phi} \mr{Id}_{V_\tau} \otimes 
\rO \big( \bigoplus_{a \; \mr{odd}: (\tau,a) \in \mr{Jord}(\phi)} \C^a \big) \Big) \times \\
\prod_{\tau \in \Irr (\mb W_F)^-_\phi} \mr{Id}_{V_\tau} \otimes 
\Sp \big( \bigoplus_{a \; \mr{even}: (\tau,a) \in \mr{Jord}(\phi)} \C^a \big) .
\end{multline}
This brings us to the setting of \cite{Moe0,MoTa} and \cite[\S 5.3]{AMS3}. In the latter it is 
checked that $(\phi,\epsilon)$ is cuspidal if and only if the following conditions are met:
\begin{itemize}
\item Jord$(\phi)$ does not have holes, that is, if $(\tau,a) \in \mr{Jord}(\phi)$ and $a>2$,
then also $(\tau,a-2) \in \mr{Jord}(\phi)$,
\item $\epsilon$ is alternated, in the sense that for all $(\tau,a), (\tau,a+2), (\tau',2) 
\in \mr{Jord}(\phi)$: 
\begin{equation}\label{eq:1.7}
\epsilon_\pi (z_{\tau,a} z_{\tau,a+2}) = -1 \quad \text{and} \quad
\epsilon_\pi (z_{\tau',2}) = -1 .
\end{equation}
\end{itemize}

\subsection{Hecke algebras for Langlands parameters} \
\label{par:HAL}

We will work out the Hecke algebras associated in \cite[\S 3]{AMS3} to Bernstein components of
enhanced $L$-parameters for $G$. Although in \cite{AMS3} we used an alternative group $\mc S_\phi$
coming from the simply connected cover of ${G^\vee}_\der$, the constructions work equally well
with $\mc S_\phi$ as above. 

For the convenience of the reader, we start by recalling the definition of the precise kind of 
(extended) affine Hecke algebra that will be involved \cite[Proposition 2.2]{AMS3}. 
We use the following data: 
\begin{itemize}
\item a root datum $\mc R = (R,X,R^\vee,Y)$, endowed with a notion of positive and simple roots in $R$,
\item the Weyl group $W(R)$,
\item $W(R)$-invariant parameter functions $\lambda: R \to \Z_{\geq 0}$ and\\
$\lambda^* : \{ \alpha \in R_\red : \alpha^\vee \in 2 Y\} \to \Z_{\geq 0}$,
\item an invertible indeterminate $\mb z$,
\item (optional) a finite group $\Gamma$ acting on $X$ and $R$, stabilizing $\lambda, \lambda^*$ and 
the set of positive roots.
\end{itemize} 
Then $\mc H (\mc R,\lambda,\lambda^*,\mb z) \rtimes \Gamma$ is the vector space
$\C[X] \otimes \C[W(R)] \otimes \C[\Gamma] \otimes \C[\mb z,\mb z^{-1}]$ with the multiplication rules:
\begin{itemize}
\item $\C[X], \C[\Gamma]$ and $\C[\mb z,\mb z^{-1}]$ are embedded as subalgebras.
\item $\C[\mb z,\mb z^{-1}]$ is central.
\item the $\C[\mb z,\mb z^{-1}]$-span of $W (R)$ is the Iwahori--Hecke algebra $\mc H (W(R),
\mb z^{2 \lambda})$ of $W (R)$ with parameters $\mb z^{2 \lambda (\alpha)}$.
That is, it has a $\C[\mb z,\mb z^{-1}]$-basis $\{ N_w : w \in W(R) \}$ such that
\[
\begin{array}{cccl}
N_w N_v & = & N_{wv} & \text{if } \ell (w) + \ell (v) = \ell (wv), \\
(N_{s_\alpha} + \mb{z}_j^{-\lambda (\alpha)}) (N_{s_\alpha} - \mb{z}_j^{\lambda (\alpha)}) &
= & 0 & \text{if } \alpha \in R_\red \text{ is a simple root.}
\end{array}
\]
\item For $w \in W(R), x \in X$ and $\gamma \in \Gamma$, corresponding to 
$N_\gamma \in \C[\Gamma]$:
\[
N_\gamma N_w \theta_x N_\gamma^{-1} = N_{\gamma w \gamma^{-1}} \theta_{\gamma (x)} ,
\]
\item For a simple root $\alpha \in R_\red$ and $x \in X$ corresponding to $\theta_x \in \C[X]$:
\end{itemize}
\begin{multline*}
\theta_x N_{s_\alpha} - N_{s_\alpha} \theta_{s_\alpha (x)} = \\
\left\{ \begin{array}{ll}
\big( \mb{z}^{\lambda (\alpha)} - \mb{z}^{-\lambda (\alpha)} \big) (\theta_x -
\theta_{s_\alpha (x)}) / (\theta_0 - \theta_{-\alpha}) & \alpha^\vee \notin 2 Y \\
\big( \mb{z}^{\lambda (\alpha)}  - \mb{z}^{-\lambda (\alpha)}  + \theta_{-\alpha}
(\mb{z}^{\lambda^* (\alpha)}  - \mb{z}^{-\lambda^* (\alpha)}) \big) (\theta_x -
\theta_{s_\alpha (x)}) / (\theta_0 - \theta_{\! -2\alpha}) & \alpha^\vee \in 2 Y
\end{array}\right. . 
\end{multline*}
We can specialize $\mb z$ to any $z \in \C^\times$, that gives another extended
affine Hecke algebra
\[
\mc H (\mc R,\lambda,\lambda^*,z) \rtimes \Gamma = 
\big( \mc H (\mc R,\lambda,\lambda^*,\mb z) \rtimes \Gamma \big) / (\mb z - z) .
\]
When $\Gamma$ is trivial or absent, we have the (non-extended) affine Hecke algebras 
$\mc H (\mc R,\lambda,\lambda^*,\mb z)$ and $\mc H (\mc R,\lambda,\lambda^*,z)$.\\

We fix a minimal $F$-Levi subgroup $L_{\mini}$ of $G = \GSpin (V)$ as in \cite[\S 2]{AsSh}.
We call a Levi subgroup $L$ standard if it contains $L_{\mini}$. 
It is shown in \cite[\S 2]{AsSh} that every such $L$ has the form
\begin{equation}\label{eq:1.60}
L = G_{n_-} \times \GL_{n_1}(F) \times \cdots \times \GL_{n_k}(F) ,
\end{equation}
where $n_- \in \N$, $G_{n_-} = \mc G_{n_-}(F) = \GSpin (V_-)$ with disc$(V_-) = \mr{disc}(V)$ and
\[
\dim (V) - \dim (V_-) = 2 (n_1 + \cdots + n_k) .
\]
The embedding $L \to G$ will be described in \eqref{eq:3.59}.

The group ${}^L G_{n_-}$ has the same type as ${}^L G$ (but smaller rank) and
\begin{equation}\label{eq:1.12}
{}^L  L = {}^L G_{n_-} \times \GL_{n_1}(\C) \times \cdots \times \GL_{n_k}(\C) .
\end{equation}
Assume that the (skew-)symmetric matrix $J \in \GL_{2n}(\C)$ defining the bilinear form has the 
following simple shape: the isotropic part is built from matrices $\matje{0}{1}{1}{0}$ or
$\matje{0}{1}{-1}{0}$ placed in rows and columns $j, 2n + 1 - j$.
Then the embedding ${}^L L \to {}^L G$ is given by 
\begin{equation}\label{eq:1.13}
(h_-,h_1,\ldots,h_k) \mapsto \big( h_1,\ldots,h_k, h_-, \mu_G^\vee (h_-) 
J h_k^{-T} J^{-1}, \ldots, \mu_G^\vee (h_-) J h_1^{-T} J^{-1} \big) ,\hspace{-4mm}
\end{equation}
where for an invertible matrix $m$ we denote the inverse-transpose by $m^{-T}={(m^{-1})}^{T}$. 

Consider a Langlands parameter $\phi \colon \mb W_F \times \SL_2 (\C) \to {}^L  L$.
With \eqref{eq:1.12} and \eqref{eq:1.13} we can write
\begin{equation}\label{eq:1.14}
\phi = \bigoplus\nolimits_j \phi_j \oplus \phi_- \oplus 
\bigoplus\nolimits_j \phi_j^\vee \otimes \mu_G^\vee \circ \phi \cong \phi_- \oplus 
\bigoplus\nolimits_j \phi_j \oplus (\phi_j^\vee \otimes \mu_G^\vee \circ \phi ) ,
\end{equation}
where $\phi_- \colon \mb W_F \times \SL_2 (\C) \to {}^L G_{n-}$ and 
$\phi_j \colon \mb W_F \times \SL_2 (\C) \to \GL_{n_j} (\C)$. Clearly $\phi$ is discrete if and only
if $\phi_-$ and all the $\phi_j$ are discrete. Notice that $\mc S_\phi = \mc S_{\phi_-}$
because $\rZ_{\GL_{n_j}(\C)}(\phi_j)$ is connected. An enhancement
\[
\epsilon \in \Irr (\mc S_\phi) = \Irr (\mc S_{\phi_-})
\]
is cuspidal if and only $(\phi_-,\epsilon)$ and all the $(\phi_j,\mr{triv})$ are cuspidal. 
For $(\phi_-,\epsilon)$ cuspidality was analysed after \eqref{eq:1.6}, while for $(\phi_j,\mr{triv})$
it means that $\phi_j$ is trivial on $\SL_2 (\C)$ and $\phi_j$ is irreducible as representation
of $\mb W_F$ \cite[Example 6.11]{AMS1}. 

Let $\Phi_\cusp (L)$ denote the set of $L^\vee$-conjugacy classes of cuspidal enhanced 
$L$-parameters for $L$. The group $Z(L^\vee)^{\mb I_F}$ acts naturally on $\Phi (L), \Phi_e (L)$
and $\Phi_\cusp (L)$ \cite[(100)]{AMS1}. From now on we assume that 
$(\phi,\epsilon) \in \Phi_\cusp (L)$. Following \cite[\S8]{AMS1} this gives a subset 
\[
\mf s_L^\vee = (Z(L^\vee)^{\mb I_F,\circ} \cdot \phi, \epsilon) \subset \Phi_\cusp (L)
\]
and a Bernstein component $\Phi_\enh (G)^{\mf s^\vee} \subset \Phi_\enh (G)$. 
For $\tau \in \Irr (\mb W_F)$, let $\ell_\tau$ be the multiplicity of $\tau$ in $\phi_-$ 
(regarded as $\mb W_F$-representation via the standard embedding ${}^L G_{n_-} \to \GL_{2 n_-}$)
and let $e_\tau$ be the sum of the multiplicities of $\tau$ in the $\GL_{n_j}(\C)$.
Then \eqref{eq:1.14} and \eqref{eq:1.5} become
\begin{align*}
& \phi = \phi_- \oplus \bigoplus_{\tau \in \Irr (\mb W_F)^\pm_\phi} 2 e_\tau \tau \oplus
\bigoplus_{\tau \in \Irr (\mb W_F)^0_\phi} 
e_\tau (\tau \; \oplus \; \tau^\vee \otimes \mu_G^\vee \circ \phi ) , \\
& \phi |_{\mb W_F} = \bigoplus_{\tau \in \Irr (\mb W_F)^\pm_\phi} (2 e_\tau + \ell_\tau) \tau \oplus
\bigoplus_{\tau \in \Irr (\mb W_F)^0_\phi} 
e_\tau (\tau \; \oplus \; \tau^\vee \otimes \mu_G^\vee \circ \phi ) 
\end{align*}
From \eqref{eq:1.3} we deduce
\begin{multline}\label{eq:1.15}
\rZ_{{G^\vee}_\der}(\phi(\mb W_F)) = \rS \Big( \prod_{\tau \in \Irr (\mb W_F)_\phi^+} 
\rO_{2 e_\tau + \ell_\tau}(\C) \otimes \mr{Id}_{V_\tau} \Big) \times \\
\prod_{\tau \in \Irr (\mb W_F)_\phi^-} \Sp_{2 e_\tau + \ell_\tau} (\C) \otimes \mr{Id}_{V_\pi} 
\times \prod_{\tau \in \Irr (\mb W_F)^0_\phi} 
\GL_{e_\tau} (\C) \otimes \mr{Id}_{V_\tau \oplus V_\tau^\vee} .  
\end{multline}
Relevant for the determination of Hecke algebras are furthermore
\begin{align*}
& \rZ_{{G^\vee}_\der}(\phi(\mb W_F)) \cap {}^L  L = \rS \Big( \prod_{\tau \in \Irr (\mb W_F)_\phi^+} 
\rO_{\ell_\tau}(\C) \times (\C^\times)^{e_\tau} \otimes \mr{Id}_{V_\tau} \big) \times \\
& \hspace{25mm} \prod_{\tau \in \Irr (\mb W_F)_\phi^-} \Sp_{\ell_\tau} (\C) \times (\C^\times)^{e_\tau} 
\otimes \mr{Id}_{V_\pi} \times \prod_{\tau \in \Irr (\mb W_F)^0_\phi} (\C^\times)^{e_\tau} \otimes 
\mr{Id}_{V_\tau \oplus V_\tau^\vee}, \\
& G^\vee_\phi = \rZ_{G^\vee} (\mb \phi (\mb W_F)) = \C^\times \rZ_{{G^\vee}_\der}(\phi (\mb W_F)) , \\
& M = G^\vee_\phi \cap {}^L  L = \C^\times \big( \rZ_{{G^\vee}_\der}(\phi(\mb W_F)) \cap {}^L  L \big) ,\\
& T = \rZ(M)^\circ \cong \C^\times \big( \prod_{\tau \in \Irr (\mb W_F)_\phi^\pm} (\C^\times)^{e_\tau}
\times \prod_{\tau \in \Irr (\mb W_F)^0_\phi} (\C^\times)^{e_\tau} \big) = \rZ ({}^L  L).
\end{align*}
If $G^\vee = \GSO_{2n}(\C)$, we may extend it to $G^{\vee +} := \GO_{2n}(\C)$. That means omitting 
the $\rS$ from \eqref{eq:1.15}, which makes the group (at most) a factor 2 bigger, so that it 
decomposes naturally as a product over the involved $\tau$'s:
\begin{equation}\label{eq:1.16}
\rZ_{{G^\vee}_\der^+}(\phi (\mb W_F)) = \prod_{\tau \in \Irr (\mb W_F)^\pm_\phi \cup
\Irr (\mb W_F)^0_\phi} G^\vee_{\phi,\tau} .
\end{equation}
Then the root system $R(G^\vee_\phi, T)$ decomposes canonically as a disjoint union of the root
systems
\[
R_\tau := R (G^\vee_{\phi,\tau} T,T) = R(G^\vee_{\phi,\tau}, T \cap G^\vee_{\phi,\tau}) .
\]
In \cite[\S 1]{AMS3} a graded Hecke algebra is attached to the data $(G^\vee_\phi, M, u_\phi, 
\epsilon)$. The maximal commutative subalgebra is $\mc O (\mr{Lie}(T))$, the root system is 
$R_\tau$ and the para\-meters of the roots come from \cite{Lus-Cusp1}. The root system and the
parameter functions $c : R_\tau \to \Z_{\geq 0}$ (which are used to construct graded Hecke algebras)
were worked out in \cite[\S 5.3]{AMS3}. To write down the parameters uniformly, we define
\begin{equation}\label{eq:1.50}
a_\tau = \left\{ \begin{array}{ll}
\max \mr{Jord}_\tau (\phi_-) & \mr{Jord}_\tau (\phi_-) \neq \emptyset \\
0 & \mr{Jord}_\tau (\phi_-) = \emptyset, \tau \in \Irr (\mb W_F )_\phi^-  \\
-1 & \mr{Jord}_\tau (\phi_-) = \emptyset, \tau \in \Irr (\mb W_F )_\phi^+ \\
\end{array} \right. .
\end{equation}
Notice that now $a_\tau$ is odd for $\tau \in \Irr (\mb W_F)_\phi^+$ and even for 
$\tau \in \Irr (\mb W_F)_\phi^-$.
For $e_\tau = 0$ the torus $T \cap G^\vee_{\phi,\tau}$ reduces to 1, and there are no roots. 
Otherwise we denote a root of length $\sqrt 2$ by $\alpha$ and a root of length 1 by $\beta$. Now 
the root systems and the parameter can be expressed as in Table~\ref{tab:1}.
When $e_\tau = 1$, we must regard $D_{e_\tau}$ and $A_{e_\tau - 1}$ as the empty root system.
Although $\beta$ is not a root in $C_n$ or $D_n$, \cite[\S 3.2]{AMS3} still allows us
to attach a useful parameter $c(\beta)$.

\begin{table}[h]
\caption{Root systems and graded Hecke algebra parameters for $\tau$} \label{tab:1}
$\begin{array}{ccccc}
\tau \in & \ell_\tau & R_\tau & c(\alpha) & c(\beta) \\
\hline
\Irr (\mb W_F)^+_\phi & = 0 & D_{e_\tau} & 2 & 0 = 1 + a_\tau \\
\Irr (\mb W_F)^+_\phi & > 0 & B_{e_\tau} & 2 & 1 + a_\tau \\
\Irr (\mb W_F)^-_\phi & = 0 & C_{e_\tau} & 2 & c(2\beta) = 2, c(\beta) = 1 = 1 + a_\tau \\
\Irr (\mb W_F)^-_\phi & > 0 & BC_{e_\tau} & 2 & 1 + a_\tau \\
\Irr (\mb W_F)^0_\phi & = 0 & A_{e_\tau - 1} & 2 & --
\end{array}$
\end{table}

Recall that the cuspidal supports in the Bernstein component $\Phi_\enh (G)^{\mf s^\vee}$ are precisely the
twists of $(\phi,\epsilon)$ by elements of 
\[
X_\nr ({}^L L) := \rZ( L^\vee \rtimes \mb I_F )^\circ_{\mb W_F} \cong X_\nr(L) .
\]
Here $\mb W_F$ acts trivially on the type $\GL$ factors of ${}^L L$ and on 
$Z (G_{n_-}) \cong \C^\times$, so
\begin{equation}\label{eq:1.26}
X_\nr ({}^L  L) \cong \rZ(G_{n_-}) \times \prod\nolimits_j \C^\times \mr{Id} = \rZ({}^L  L) = T .
\end{equation}
Without changing $\Phi_\enh (G)^{\mf s^\vee}$, we can bring $(\phi, \epsilon)$ in a somewhat better
position:
\begin{itemize}
\item if $\phi_j : \mb W_F \to \GL_{n_j}(\C)$ differs from $\phi_j^\vee \otimes \mu_G^\vee \circ 
\phi$ by $z \in \rZ(\GL_{n_j}(\C)) \cong X_\nr (\GL_{n_j}(F))$, then we replace $\phi_j$ by 
$z^{1/2} \phi_j$, so that
\[
z^{1/2} \phi_j \cong z^{-1/2} \phi_j^\vee \otimes \mu_G^\vee \circ \phi \cong 
(z^{1/2} \phi_j )^\vee \otimes \mu_G^\vee \circ \phi .
\]
\item if $n_i = n_j$ and $\phi_i,\phi_j$ differ by an element of $X_\nr (\GL_{n_i}(F))$, then we
adjust one of them so that actually $\phi_i = \phi_j$,
\item if $\phi_i,\phi_j \in \Irr (\mb W_F)^0_\phi$ and $\phi_i, \phi_j^\vee \otimes \mu_G^\vee 
\circ \phi$ differ by an element of $\GL_{n_i}(\C)$, then we replace $\phi_j$ by $\phi_i$.
\end{itemize}
Let $\tau' \in \Irr (\mb W_F)_\phi^\pm$ be a twist of $\tau \in \Irr (\mb W_F)_\phi^\pm$ by an 
unramified character, such that $\tau'$ is equivalent with $\tau^{'\vee} \otimes \mu_G^\vee \circ \phi$ 
but not with $\tau$. By the above assumptions on $\phi$, $e_{\tau'} = 0$ if $e_\tau > 0$. Still, 
$\ell_\tau$ and $\ell_{\tau'}$ can be nonzero simultaneously. If $e_\tau > 0$ and 
$\ell_\tau < \ell_{\tau'}$ (resp. $\ell_\tau = \ell_{\tau'} = 0$ and $a_\tau = 0 < a_{\tau'}$)
then we change $\phi_j = \tau$ to $\phi_j = \tau'$, so that the roles of $\ell_\tau$ and 
$\ell_{\tau'}$ (resp. of $a_\tau$ and $a_{\tau'}$) are exchanged.

Let $\rZ(L^\vee)^\circ_\phi \subset \rZ(L^\vee)^\circ = T$ be the subgroup of elements $z$ such that
$z \phi$ is equivalent with $\phi$ in $\Phi_\enh (L)^{\mf s^\vee}$. It is finite and the map
\[
\rZ(L^\vee)^\circ / \rZ(L^\vee)^\circ_\phi \to \Phi_\enh (L)^{\mf s^\vee} :
z \mapsto (z\phi,\epsilon)
\]
is a bijection. The affine Hecke algebra we are constructing has the underlying complex torus
\begin{equation}\label{eq:1.48}
T_{\mf s^\vee} := \rZ(L^\vee)^\circ / \rZ(L^\vee)^\circ_\phi = 
T \big/ \prod\nolimits_j Z (\GL_{n_j}(\C))_{\phi_j} .
\end{equation}
Let $t_\tau$ be the torsion number of $\tau \in \Irr (\mb W_F)$, that is, the order of the
group $\rZ(\GL_{d_\tau}(\C))_\tau$. Then $t_\tau$ is also the number of irreducible constituents
$\theta$ of $\mr{Res}^{\mb W_F}_{\mb I_F} \tau$. We need to distinguish two cases:
\begin{itemize}
\item[(i)] $\theta \cong \theta^\vee \otimes \mu_G^\vee \circ \phi$. Then the same goes for all 
constituents of $\mr{Res}^{\mb W_F}_{\mb I_F} \tau$, because all those are in one orbit for $\mb W_F$.
The proof of \cite[Proposition 4.10.a]{SolParam} (which concerns self-dual representations of
$\mb W_F$) applies and shows that $\tau$ and $\tau'$ have the same sign.
\item[(ii)] $\theta \not\cong \theta^\vee \otimes \mu_G^\vee \circ \phi$ for all eligible $\theta$. 
Again the proof of \cite[Proposition 4.10.a]{SolParam} applies, now it shows that $\tau$ and 
$\tau'$ have different signs.
\end{itemize}
According to these two cases, we divise a new partition of $\Irr (\mb W_F)_\phi^\pm$:
\begin{itemize}
\item $\Irr (\mb W_F)^{++}$ is the set of all $\tau \in \Irr (\mb W_F)_\phi^+$ in case (i) above,
modulo the relation $\tau \sim \tau'$;
\item $\Irr (\mb W_F)^{--}$ is defined in the same way, only starting from $\Irr (\mb W_F )_\phi^-$;
\item $\Irr (\mb W_F)^{+-}$ is the set of all $\tau \in \Irr (\mb W_F)_\phi^{\pm}$ in case (ii) above,
modulo the relation $\tau \sim \tau'$;
\item $\Irr' (\mb W_F)_\phi^\pm = \Irr (\mb W_F)^{++} \cup \Irr (\mb W_F)^{--} \cup \Irr (\mb W_F)^{+-} =
\Irr (\mb W_F)_\phi^\pm / (\tau \sim \tau')$.
\end{itemize}
A computation like for \eqref{eq:1.15} yields
\begin{align}
\nonumber \rZ_{{G^\vee}_\der}(\phi (\mb I_F)) = & \; \rS \Big( \prod_{\tau \in \Irr (\mb W_F)_\phi^{++}} 
\rO_{2 e_\tau + \ell_\tau + \ell_{\tau'}}(\C)^{t_\tau} \Big) \times \prod_{\tau \in \Irr 
(\mb W_F)_\phi^{--}} \Sp_{2 e_\tau + \ell_\tau + \ell_{\tau'}} (\C)^{t_\tau} \\
\label{eq:1.18} & \; \times \prod_{\tau \in \Irr (\mb W_F)_\phi^{+-}} 
\GL_{2 e_\tau + \ell_\tau + \ell_{\tau'}}(\C)^{t_\tau /2} 
\times \prod_{\tau \in \Irr (\mb W_F)^0_\phi} \GL_{e_\tau} (\C)^{t_\tau}  .  
\end{align}
Analogous to \eqref{eq:1.16} we decompose
\begin{equation}\label{eq:1.19}
\begin{split}
& \rZ_{{G^\vee}_\der^+}(\phi (\mb I_F)) = \prod\nolimits_\tau G^\vee_{\phi (\mb I_F), \tau} ,\\
& T_{\mf s^\vee} = \C^\times \underset{\{1,-1\}}{\times} \prod\nolimits_\tau T_{\mf s^\vee,\tau} =
\C^\times \underset{\{1,-1\}}{\times}  \prod\nolimits_\tau \big( \C^\times / 
\rZ( \GL_{d_\tau}(\C) )_\tau \big)^{e_\tau} , \\
& X^* (T_{\mf s^\vee}) \subset \Z \oplus \bigoplus\nolimits_\tau X^* (T_{\mf s^\vee,\tau}) \cong
\Z \oplus \bigoplus\nolimits_\tau \Z^{e_\tau} .
\end{split}
\end{equation}
In each of the above cases, the product or sum runs over $\Irr' (\mb W_F)^\pm_\phi \cup 
\Irr (\mb W_F)^0_\phi$. For comparison with \cite{AMS3} we record the group
\begin{equation} \label{eqn:J}
J = \rZ_{G^\vee}(\phi (\mb I_F)) = \C^\times \rZ_{{G^\vee}_\der}(\phi (\mb I_F)) .
\end{equation}
The root system of $J^\circ$ with respect to the (possibly non-maximal) torus $T$ splits naturally
as a disjoint union of root systems $R(G^\vee_{\phi (\mb I_F),\tau}T ,T)$, indexed as in 
\eqref{eq:1.19}. We note that by the above assumptions on $\tau$ and $\tau'$ we have $\ell_\tau 
\geq \ell_{\tau'} = 0$ and if $\ell_\tau = 0$, then $\ell_{\tau'} = 0$ and $a_\tau \geq a_{\tau'}$. 
With Table~\ref{tab:1} at hand, one checks readily that 
\begin{equation}\label{eq:1.49}
R(G^\vee_{\phi (\mb I_F),\tau}T ,T) = R(G^\vee_{\phi,\tau}T ,T)
\end{equation}
in all cases. (Only the dimensions of the root subspaces for $G^\vee_{\phi (\mb I_F),\tau}$ are 
higher than for $G_{\phi,\tau}^\vee$, namely $t_\tau$ times higher in all but ones cases.) 
In view of \cite[Proposition 3.9]{AMS3}, this means that $(\phi,\epsilon)$ is a good 
basepoint of $\Phi_\enh (L)^{\mf s^\vee}$.

Now the roots for the Hecke algebra that we are after can be found with \cite[Definition 3.11]{AMS3}.
The reduced roots $\alpha \in R(G^\vee_{\phi (\mb I_F),\tau}T ,T)$ need to be scaled by a certain
factor $m_\alpha \in \N$, which we compute next. Let $B_J \supset T_J$ be a $\phi (\Fr_F)$-stable 
Borel subgroup and maximal torus of $J^\circ$, such that $T_J^{\phi (\Fr_F)} \supset T$. A natural
choice for $T_J$ comes from the standard maximal tori $T_{J,\tau}$ in $G^\vee_{\phi (\mb I_F),\tau}$:
\begin{equation}\label{eq:1.47}
T_J = \C^\times \Big( \prod_{\tau \in \Irr (\mb W_F)_\phi^{++} \cup \Irr (\mb W_F)_\phi^{--}}
T_{J,\tau}^{t_\tau} \times \prod_{\tau \in \Irr (\mb W_F)_\phi^{+-}} T_{J,\tau}^{t_\tau / 2}
\times \prod_{\tau \in \Irr (\mb W_F)_\phi^{0}} T_{J,\tau}^{t_\tau} \Big).
\end{equation}
We see that $R(G^\vee_{\phi (\mb I_F),\tau}, T_J)$ has $t_\tau$ irreducible components, 
unless $\tau \in \Irr (\mb W_F)_\phi^{+-}$, then there are $t_\tau /2$. Following 
\cite[Definition 3.11]{AMS3}, $m_\alpha$ equals $t_\tau$ (or $t_\tau / 2$ for 
$\tau \in \Irr (\mb W_F)_\phi^{+-}$) times a number $m'_\alpha$ which is $m_\alpha$ for 
$\mr{Res}^{\mb W_F}_{\mb W_E} \phi$ where $E/F$ is the unramified extension of degree $t_\tau$ (or 
$t_\tau / 2$ for $\tau \in \Irr (\mb W_F)_\phi^{+-}$). By definition $m'_\alpha$ is the smallest
number such that $\ker (m'_\alpha \alpha)$ contains all $t \in T$ for which 
$t \mr{Res}^{\mb W_F}_{\mb W_E} \phi$ is equivalent with $\mr{Res}^{\mb W_F}_{\mb W_E} \phi$. 

The group of $t \in T$ with $t \phi \cong \phi$ factors as a product indexed by all possible $\tau$, 
and the contribution from one $\tau$ consists of $t_\tau$ unramified characters of $\mb W_F$. But 
this group of unramified characters becomes trivial if we pass from $\mb W_F$ to the Weil group of
the degree $t_\tau$ unramified extension of $F$. Hence $m'_\alpha = 1$, unless maybe when
$\tau \in \Irr (\mb W_F)_\phi^{+-}$. In the latter situation we usually have $m'_\alpha = 2$, 
because $\ker (m'_\alpha \alpha)$ has to contain an element $t \in (1 - \phi (\Fr_F )^{t_\tau /2}) 
T_{J,\tau}$ with $\alpha (t) = -1$. The only exception occurs when $\ell_\tau = 0$ and 
$\alpha \in C_{e_\tau}$ is long, then $m'_\alpha = 1$.
We conclude that $m_\alpha = t_\tau$ in all cases, except when $\tau \in \Irr (\mb W_F)_\phi^{+-},
\ell_\tau = 0$ and $\alpha \in C_{e_\tau}$ is long, then $m_\alpha = t_\tau / 2$.

Finally, we are ready to define the root datum for our affine Hecke algebra:
\begin{equation}\label{eq:1.30}
\begin{aligned}
& \mc R_{\mf s^\vee} = \big( R_{\mf s^\vee}, X^* (T_{\mf s^\vee}), 
R_{\mf s^\vee}^\vee, X_* (T_{\mf s^\vee}) \big), \\
& \text{where }R_{\mf s^\vee} = \{ m_\alpha \alpha : \alpha \in R(G^\vee_\phi ,T)_\red \} .
\end{aligned}
\end{equation}
Here $R_{\mf s^\vee}$ is the disjoint union of root subsystems 
\[
R_{\mf s^\vee,\tau} = \{ m_\alpha \alpha : \alpha \in R(G^\vee_{\phi,\tau} T,T)_\red \} .
\]
Notice that $X^* (T_{\mf s^\vee,\tau})$ arises from the part of $X^* (T)$ 
associated to $\tau$ by multiplication with $t_\tau$, where $t_\tau = m_\alpha$ for most $\alpha \in 
R (G^\vee_{\phi,\tau} T,T)$. The multiplication rules in our affine Hecke algebra are determined by 
parameter functions $\lambda, \lambda^* \colon R_{\mf s^\vee} \to \Z_{\geq 0}$, which come from 
\cite[Proposition 3.14]{AMS3}. The outcome of those constructions is summarized in \cite[\S 5.3]{AMS3}:
\begin{itemize}
\item For $\alpha \in R_{\tau,\red}$ a short root in a type $B$ root system, $t_\tau = m_\alpha$,
$c(\alpha) = a_\tau + 1, c^* (\alpha) = a_{\tau'} + 1$ and
\[
\lambda (\alpha) = t_\tau (a_\tau + a_{\tau'} + 2)/2, \qquad
\lambda^* (\alpha) = t_\tau (a_\tau - a_{\tau'}) / 2 .
\]
We note that $\lambda^* (\alpha) \geq 0$ because $\ell_\tau \geq \ell_{\tau'}$.
\item For $\alpha \in R_{\tau,\red}, \tau \in \Irr (\mb W_F)_\phi^{+-}$, $\ell_\tau = 0$, 
$\alpha$ a long root of a type $C$ root system: $c(\alpha) = 2$ and 
\[
\lambda (\alpha) = \lambda^* (\alpha) = m_\alpha = t_\tau / 2 .
\]
\item For all other $\alpha \in R_{\tau,\red}$: $c(\alpha) = 2$ and 
\[
\lambda (\alpha) = \lambda^* (\alpha) = m_\alpha = t_\tau .
\]
\end{itemize}
We note that the operation $\alpha \mapsto m_\alpha \alpha$ preserves the type of the root systems
$R_{\tau,\red}$ from Table~\ref{tab:1}, except that in the case $\tau \in \Irr (\mb W_F)_\phi^{+-}, 
\ell_\tau + \ell_{\tau'} = 0$ type $C_{e_\tau}$ is turned into $B_{e_\tau}$.
 
We also need to determine $W_{\mf s^\vee}$, the stabilizer of $\mf s_L^\vee$ in 
\begin{equation}\label{eq:1.20}
N_{G^\vee}(L^\vee \rtimes \mb W_F) / L^\vee = N_{G^\vee}(L^\vee) / L^\vee.
\end{equation}
Recall the embedding ${}^L L \to {}^L G$ from \eqref{eq:1.13}. For each $j$ the group
$\rN_{G^{\vee +}}(L^\vee)$ possesses an element that exchanges $h_j$ and $\mu_G^\vee (h_-) J h_j^{-T}
J^{-1}$. In terms of representations of $\mb W_F$ (via $\phi$), this 
\begin{equation}\label{eq:1.25}
\text{exchanges } \tau \text{ and } \tau^\vee \otimes \mu_G^\vee \circ \phi .
\end{equation}
Further $\rN_{G^{\vee +}}(L^\vee)$ contains elements that permute the factors $\GL_{n_j}(\C)$ of the
same size. It follows that $\rN_{G^{\vee +}}(L^\vee) / L^\vee$ is isomorphic with a direct product of
Weyl groups of type $B_{e_N}$, where $e_N$ counts the number of $j$'s with $n_j = N$.
The group \eqref{eq:1.20} has index at most two in $\rN_{G^{\vee +}}(L^\vee) / L^\vee$, which comes from
the difference between $\GO_{2n}(\C)$ and $\GSO_{2n}(\C)$.

The group $W_{\mf s^\vee}$ can be represented with elements that normalize $M$ and $T$ and 
centralize $\phi (\mb I_F \times \SL_2 (\C))$, so in particular elements of $J$. Further 
$W_{\mf s^\vee}$ contains $W(R_{\mf s^\vee}) = W(J^\circ,T)$ as a normal subgroup. Fix a standard
Borel subgroup $B^\vee$ of $G^\vee$. That determines a Borel subgroup $B^J$ of $J^\circ$, and hence
a system of positive roots in $R(J^\circ,T)$ and in $R_{\mf s^\vee}$. Let $\Gamma_{\mf s^\vee}$ be
the subgroup of $W_{\mf s^\vee}$ that stabilizes this positive system of roots. By standard results
about finite root system and Weyl groups
\begin{equation}\label{eq:1.29}
W_{\mf s^\vee} = W(R_{\mf s^\vee}) \rtimes \Gamma_{\mf s^\vee} . 
\end{equation}
Let us determine $\Gamma_{\mf s^\vee}$ in terms of the action of $ N_{G^\vee}(L^\vee) / L^\vee$
on the type GL factors of $L^\vee$ and on the tensor factors of $\phi$ (as described above). The
$\tau \in \Irr (\mb W_F)$ with $e_\tau = 0$ do not contribute. If $e_\tau > 0$, then 
$e_{\tau \otimes \chi} = 0$ for every unramified twist $\tau \otimes \chi$ which is not isomorphic to
$\tau$ by our normalization of $\phi$. Hence every element of $\rN_{G^{\vee +}}(L^\vee)$ that stabilizes 
$\mf s_L^\vee$ must already stabilize $\phi$. In other words, $W_{\mf s^\vee}$ equals the
stabilizer of $\phi$ in $\rN_{G^\vee}(L^\vee) / L^\vee$. Thus we can represent $W_{\mf s^\vee}$ with
elements of $\rZ_{{G^\vee}_\der}(\phi)$ that normalize $T$. Let $W_{\mf s^\vee}^+$ and 
$\Gamma_{\mf s^\vee}^+$ be the versions of $W_{\mf s^\vee}$ and $\Gamma_{\mf s^\vee}$ for 
$G^{\vee +}$. From \eqref{eq:1.15} we see that 
\begin{equation}\label{eq:1.33}
W_{\mf s^\vee}^+ = \prod_\tau W^+_{\mf s^\vee,\tau} \cong \prod_{\tau \in \Irr' (\mb W_F)_\phi^\pm} 
W(B_{e_\tau}) \times \prod_{\tau \in \Irr (\mb W_F)^\circ_\phi} W (A_{e_\tau - 1}) .
\end{equation}
Comparing with Table~\ref{tab:1}, we find that 
\begin{equation}\label{eq:1.21}
\Gamma_{\mf s^\vee}^+\cong \prod_{\tau \in \Irr' (\mb W_F)^+_\phi : \ell_\tau = 0} \hspace{-7mm}
W(B_{e_\tau}) / W(D_{e_\tau}) \cong \prod_{\tau \in \Irr' (\mb W_F)^+_\phi : \ell_\tau = 0}
\hspace{-7mm} \pi_0 \big(\rO_{2 e_\tau}(\C) \otimes \mr{Id}_{V_\tau} \big) .
\end{equation}
In \eqref{eq:1.21} every $\tau$ contributes a factor 
\[
\Gamma_{\mf s^\vee,\tau}^+ = \langle r_\tau \rangle \cong \Z / 2 \Z
\]
to $\Gamma_{\mf s^\vee}^+$. For $\tau$ not appearing in \eqref{eq:1.21}, we may put
$\Gamma_{\mf s^\vee,\tau}^+ = 1$.

When $\dim \tau$ is even, $\det (r_\tau) = 1$ and when 
$\dim \tau$ is odd, $\det (r_\tau) = -1$. Hence the $S$ in \eqref{eq:1.15} does not put any condition
on the $r_\tau$ with $\dim \tau$ even. If there exists a $\tau \in \Irr (\mb W_F)^+_\phi$ with
$\ell_\tau > 0$, then we can use $\rO_{\ell_\tau}(\C)$ to make the determinant of a product of
$r_\tau$'s equal to 1. From \cite[\S 4.1]{Mou} we know that this leaves just two possibilities
for $\Gamma_{\mf s^\vee}$:
\begin{itemize}
\item $\Gamma_{\mf s^\vee} = \prod_{\tau \in \Irr' (\mb W_F)^+_\phi : \ell_\tau = 0}
\langle r_\tau \rangle$,
\item if $\mc G$ is a form of $\GSpin_{2n}$ and $\mc L \cong \prod_j \GL_{n_j}$ with
$n_j \in \Z_{>0}$, then 
\[
\Gamma_{\mf s^\vee} = \prod_{\tau \in \Irr' (\mb W_F)^+_\phi : \ell_\tau = 0, \dim \tau \: \mr{even}}
\langle r_\tau \rangle \times \rS \Big( \prod_{\tau \in \Irr' (\mb W_F)^+_\phi : 
\ell_\tau = 0, \dim \tau \: \mr{odd}} \langle r_\tau \rangle \Big) ,
\]
where $S$ denotes the subgroup of elements with determinant 1.
\end{itemize}
Conceivably our affine Hecke algebra could contain the span of $\Gamma_{\mf s^\vee}$ as a twisted
group algebra. But here the 2-cocycle of $W_{\mf s^\vee}$ involved in the Hecke algebra can be computed 
for each $\tau$ separately, and $\langle r_\tau \rangle \cong \Z / 2 \Z$ only has trivial 2-cocycles.

Let us summarise our findings. From \eqref{eq:1.19} we know that $X^* (T_{\mf s^\vee})$ has index 
two in $\Z \oplus \bigoplus_\tau \Z^{e_\tau}$, where $\Z^{e_\tau} \cong X^* (T_{\mf s^\vee,\tau})$. If we
replace $X^* (T_{\mf s^\vee})$ and $X_* (T_{\mf s^\vee})$ in $\mc R_{\mf s^\vee}$ by 
$\oplus_\tau \Z^{e_\tau}$, we get a new root datum $\mc R_{\mf s^\vee,\der}$ that decomposes 
naturally. More precisely, the root datum $\mc R_{\mf s^\vee,\der}$, extended with the finite group 
$\Gamma_{\mf s^\vee}^+$ acting on it, is a direct sum of such extended root data, where the product 
is indexed by $\tau \in \Irr' (\mb W_F )_\phi^\pm \cup \Irr (\mb W_F )_\phi^0$. For each such $
\tau$ the data are (with $\alpha$ a root of length $\sqrt 2$ and $\beta$ a root of another length)
are collected in Table~\ref{tab:2}. Recall from \eqref{eq:1.50} that $a_\tau$ is odd for 
$\tau \in \Irr (\mb W_F)_\phi^+$ and even for $\tau \in \Irr (\mb W_F)_\phi^-$.
 
\begin{table}[h]
\caption{Data from $\mc R_{\mf s^\vee}$ for each $\tau$}\label{tab:2}
$\begin{array}{cccccccc}
a_\tau & a_{\tau'} & X^* (T_{\mf s^\vee,\tau}) \!\! & R_{\mf s^\vee,\tau} & \lambda(\alpha) & 
\lambda(\beta) & \lambda^* (\beta) & \Gamma_{\mf s^\vee}^+ \\
\hline
-1 & -1 & \Z^{e_\tau} & D_{e_\tau} & t_\tau & -- & -- & \hspace{-5mm} \mr{Out} (D_{e_\tau}) \\
0 & -1 & \Z^{e_\tau} & B_{e_\tau} & t_\tau & t_\tau /2 & t_\tau / 2 & 1 \\
0 & 0 & \Z^{e_\tau} & C_{e_\tau} & t_\tau & t_\tau & t_\tau & 1 \\
\geq 1 & -1 & \Z^{e_\tau} & B_{e_\tau} & t_\tau & t_\tau (a_\tau + 1)/2 & t_\tau ( a_\tau + 1) / 2 & 1 \\
\geq 1 & \geq 0 & \Z^{e_\tau} & B_{e_\tau} & t_\tau & t_\tau (a_\tau + a_{\tau'} + 2)/2 & 
t_\tau ( a_\tau - a_{\tau'}) / 2 & 1 \\
\multicolumn{2}{c}{\Irr (\mb W_F)^0_\phi} & \Z^{e_\tau} & A_{e_\tau - 1} & t_\tau & -- & -- & 1
\end{array}$
\end{table}
Here the second, third and fourth lines can be regarded as special cases of the fifth line. 
We write them down nevertheless, because they arise from different lines in Table~\ref{tab:1}.
With Table~\ref{tab:2} and \eqref{eq:1.19} we can finally make the affine Hecke algebra 
associated in \cite{AMS3} to $\mf s^\vee$ (and a parameter $z \in \C^\times$) explicit:
\begin{equation}\label{eq:1.22}
\mc H (\mf s^\vee, z) = \mc H (\mc R_{\mf s^\vee}, \lambda, \lambda^*, z) \rtimes \Gamma_{\mf s^\vee},
\end{equation} 
where $\Gamma_{\mf s^\vee}$ acts on $\mc H (\mc R_{\mf s^\vee}, \lambda, \lambda^*, z)$ via
automorphisms of $\mc R_{\mf s^\vee}$.

\section{M{\oe}glin's classification of discrete series representations} 
\label{sec:Moe}

Arthur famously proved the local Langlands correspondence for symplectic and quasi-split (special)
orthogonal groups over $p$-adic fields \cite{Art}. An analogue for quasi-split unitary groups was 
announced in \cite{Mok} and proven (for all unitary groups) in \cite{KMSW}. As explained in 
\cite{Moe3,MoRe}, Arthur's endoscopic methods can also be applied to (special) orthogonal groups
and general spin groups that are not necessarily quasi-split. In principle that should yield local 
Langlands correspondences for all classical groups over $p$-adic fields. However, not all arguments
have been worked out in detail. For classical groups over local function fields far less is known,
the notable exception being \cite{GaVa}. We address that in Paragraph~\ref{par:close}. Here we make 
M{\oe}glin 's parametrization of discrete series representations \cite{Moe0,MoTa,Moe2,Moe3} more explicit. 
 
Let $G = G_n = \mc G (F)$ be a symplectic group, a special orthogonal group or a general
spin group. When $G$ is an even special orthogonal group or an even 
general spin group, we denote by $G^+$ the associated orthogonal or general pin group, as 
in \eqref{eq:G+}. In the other cases $G^+$ means just $G$. Let $\rZ(G)_s$ be the maximal $F$-split central
torus in $G$. It is isomorphic to $F^\times$ for general spin groups and trivial in the other cases.

We say that an irreducible smooth $G$-representation belongs to the discrete series if it is 
square-integrable modulo centre. More explicitly, that means that $\pi$ has a unitary central 
character and its restriction to the derived group of $G$ is square-integrable. 

The group $\GL_m (F) \times G_n$ is a Levi subgroup of a group $G_{n+m}$ of the same kind as $G_n$ but 
of rank $m$ higher. There is a parabolic induction functor
\[
\times \colon\Rep (\GL_m (F)) \times \Rep (G_n) \to \Rep (G_{n+m}) , 
\]
which up to semisimplification does not depend on the choice of a parabolic subgroup of $G_{n+m}$
with Levi factor $\GL_m (F) \times G_n$. Similarly there  is a parabolic 
induction functor
\[
\times \colon\Rep (\GL_m (F)) \times \Rep (G_n^+) \to \Rep (G^+_{n+m}) . 
\]
Let $\rho \in \Irr (\GL_{d_\rho}(F))$ be unitary and supercuspidal, for some $d_\rho$. For an integer 
$a \geq 1$ we can form the generalized Steinberg representation $\delta (\rho,a) \in \Irr (\GL_m (F))$ 
with $m = d_\rho a$. Take $\pi$ in the discrete series of $G_n^+$ and let $\nu_\pi$ be the character by 
which $\rZ(G)_s$ acts on $\pi$. One says that $(\rho,a)$ lies in the 
Jordan block of $\pi$ if $\delta (\rho,a) \times \pi$ is irreducible but there exists $a' \in a + 2 \Z$ 
such that $\delta (\rho,a') \times \pi$ is reducible. We denote the set of all such pairs $(\rho,a)$ 
by Jord$(\pi)$. That reducibility is only possible if the nontrivial element 
\[
s_\alpha \in N_{G^+_{n + d_\rho a}} \big( \GL_{d_\rho a}(F) \times G_n^+ \big) \big/ 
\big( \GL_{d_\rho a}(F) \times G_n^+ \big),
\] 
see \eqref{eq:1.23}, stabilizes $\delta (\rho,a') \boxtimes \pi$ up to an unramified character.
That in turn implies that the version of $s_\alpha$ with $a=1$ stabilizes $\rho$, or more explicitly 
\begin{equation}\label{eq:1.44}
\rho \cong \rho^\vee \otimes \nu_\pi .
\end{equation}
To Jord$(\pi)$ one associates a finite group $S_\pi$, the $\mh F_2$-vector space with basis
$\{ z_{\rho,a} : (\rho,a) \in \mr{Jord}(\pi) \}$. Via endoscopy, $\pi$ determines a character
$\epsilon_\pi : S_\pi \to \{1,-1\}$, see \cite[Theorem 2.2.1]{Art} and \cite[\S 2.1]{MoRe}.
This requires a Whittaker datum for the quasi-split inner form of $G$, which we will use as input.
Alternatively, $\epsilon_\pi$ can be defined almost entirely using parabolic induction, 
see \cite[p. 147--148]{Moe0} or Paragraphs~\ref{par:Sc}--\ref{par:intertwining}. 

\begin{thm} \textup{[M\oe glin]} \label{thm:1.1} \\  
Let $F$ be a $p$-adic field and consider $\pi$ in the discrete series of $G^+$. 
\enuma{
\item $\Jord(\pi)$ has image in ${}^L G$, by which we mean that the Langlands parameter of 
\[
\mathlarger{\mathlarger{\boxtimes}}_{(\rho,a) \in \mr{Jord}(\pi)} \delta (\rho,a) \in \Irr
\big( \prod\nolimits_{(\rho,a) \in \mr{Jord}(\pi)} \GL_{d_\rho a}(F) \big)
\] 
factors through ${}^L G$.
\item $\Jord(\pi)$ determines precisely the $L$-packet containing $\pi$.
\item When $G \neq G^+$, the restriction of $\pi$ to $G$ is reducible if and only if $d_\rho a$ 
is even for all $(\rho,a) \in \Jord(\pi)$. 
\item Fix a Whittaker datum for the quasi-split inner form of $G$. That and the above determine 
an injection from the discrete series of $G^+$ to the set of pairs
$(\Jord,\epsilon)$ (up to $G^{\vee+}$-conjugacy) for which $\Jord$ has image in ${}^L G$ and 
$\epsilon$ is $G$-relevant, as explained around \eqref{eq:1.40}.
\item $\pi$ is supercuspidal if and only if $\Jord(\pi)$ does not have holes and $\epsilon_\pi$
is alternated in the sense of \eqref{eq:1.7}.
}
\end{thm}

In fact there should be a bijection in part (d), but for our purposes an injection suffices.
Theorem\autoref {thm:1.1}.a entails in particular
\begin{equation}\label{eq:sizeJord}
\sum\nolimits_{(\rho,a) \in \mr{Jord}(\pi)} d_\rho a = \text{size of } G^\vee ,
\end{equation}
where the size of an $N \times N$-matrix is $N$.
Parts (a) and (b) of Theorem\autoref {thm:1.1} are in proven in \cite[\S 2.2--2.5]{Moe2}. When in 
addition $G$ is quasi-split, parts (c) and (d) are shown in \cite[\S 7.1]{Moe3}. Theorem 
\autoref{thm:1.1}.c--d is stated for all our $G^+$ in \cite[\S 2.5]{Moe2}, attributed to Arthur 
\cite{Art}. Later this was worked out for non-quasi-split groups in \cite{MoRe}. We note that 
these sources do include surjectivity in part (d). Part (e) was shown in \cite[Theorem 2.5.1]{Moe2}. 

For the main results in this paper it suffices to know that Theorem\autoref {thm:1.1} holds for 
supercuspidal representations. We point out that unfortunately the proof of those cases of 
Theorem\autoref {thm:1.1} is not yet entirely complete (except when $G$ has very small rank). 
Arthur's book \cite{Art} relies on certain papers that were announced but have not yet 
appeared. Meanwhile most of this has been settled in \cite{MoWa, AGIKMS}, and it is expected 
that the remaining gaps in \cite{Art} will be fixed soon.
The paper \cite{MoRe} uses \cite{Art}, and also leaves some other details to be worked out.

On the other hand, the part of Theorem\autoref {thm:1.1} that classifies discrete series representations
in terms of supercuspidal representations is documented much better. Besides the above references,
it is also treated in \cite{Moe0,MoTa,KiMa,Xu}. We describe this explicitly in Paragraph~\ref{par:Sc}.

\subsection{The method of close local fields} \
\label{par:close} 

To goal of this paragraph is to deduce instances of Theorem\autoref {thm:1.1} for groups over local
function fields (for which very little is in the literature) from Theorem\autoref {thm:1.1} for
groups over $p$-adic fields. To this end we employ the method of close 
fields, a general method to transfer statements from a group over one
local field to the same group over an another local field, provided these fields look sufficiently
similar. Let $F$ be a local field of positive characteristic and let $F'$ be a local field of 
characteristic zero. From the classification of classical groups we see that we can define any 
algebraic group $\mc G = \mc G_n$ as in Section~\ref{sec:Moe} simultaneously over $F$ and over $F'$.

Consider $\pi$ in the discrete series of $\mc G (F)^+$. Let $d$ be the maximum of the depths of 
$\pi$ and of all the $\rho$ that appear in Jord$(\pi)$. We denote the subcategory of 
$\Rep (\mc G (F))$ generated by the representations of depth $\leq d$ by $\Rep(\mc  G (F))_{\leq d}$. 
Let $F'$ be a $p$-adic field which is sufficiently close to $F$, with respect to the depth 
$\mc D := \mc D (p,\mc G,d)$ and the groups $\mc G$, $\GL_m$, $\mc G_m$ with $m \leq \mr{rk}(\mc G)$. 
Here $F$ and $F'$ are at least $\mc D$-close, but usually a lot closer is needed.

By \cite{Gan} the method of close fields yields canonical equivalences of categories
\begin{equation}\label{eq:1.41}
\begin{array}{llll}
\zeta^{\mc G,F,F'} : & \Rep (\mc G_n (F))_{\leq \mc D} & \isom & 
\Rep (\mc G_n (F'))_{\leq \mc D} , \\
\zeta^{\mc G,F,F'} : & \Rep (\GL_m (F))_{\leq \mc D} & \isom & \Rep (\GL_m (F'))_{\leq \mc D} , \\
\zeta^{\mc G,F,F'} : & \Rep (\mc G_{n+m} (F))_{\leq \mc D} & \isom & 
\Rep (\mc G_{n+m} (F'))_{\leq \mc D} ,
\end{array}
\end{equation}
for all $m \leq \mr{rk}(\mc G) = n$. By \cite[Theorem 3.5]{SolParam} these equivalences of categories 
are compatible with normalized parabolic induction. Hence the equivalences \eqref{eq:1.41} 
transfer the condition that $(\rho,a)$ belongs to Jord$(\pi)$ into the condition that 
$\zeta^{\GL_{d_\rho},F,F'}(\pi)$ belongs to Jord$(\zeta^{\mc G,F,F'}(\pi))$. In other words,
\eqref{eq:1.41} induces an injection
\begin{equation}\label{eq:1.42}
\mr{Jord}(\pi) \to \mr{Jord} (\zeta^{\mc G,F,F'}(\pi)) .
\end{equation}
Now the problem arises that Jord$(\pi)$ could be too small, so that \eqref{eq:1.42} would not be
surjective. Then \eqref{eq:sizeJord} fails and Jord$(\pi)$ would not yield a Langlands parameter for 
$\mc G (F)$. For groups over $p$-adic fields this used to be a difficult problem \cite[p. 727]{MoTa}, 
which has only been solved with the endoscopic methods from \cite{Art}.
To carry out the method of close fields completely, we need the following additional input.

\begin{ass}\label{as:1.8}
Fix $\mc G$, a prime $p$ and a depth $d \in \N$. 
There exists a bound $\mc D (p,\mc G,d) \in \Z_{\geq d}$ such that
\begin{itemize}
\item for all $p$-adic fields $F'$,
\item for all unitary supercuspidal representations $\sigma \in \Irr (\mc G (F'))$ of depth $\leq d$,
\item for all $\rho \in \Irr (\GL_{d_\rho}(F'))$ occurring in $\Jord(\sigma)$,
\end{itemize}
the depth of $\rho$ is $\leq \mc D (p,\mc G,d)$.
\end{ass}
For symplectic groups and split special orthogonal groups this assumption is known (for $p > 2$) 
from \cite[Lemma 8.2.3]{GaVa}, in the stronger form $\mc D (p,\mc G,d) = d + 1$. In fact the main
results of \cite{GaVa} imply Theorem\autoref {thm:1.1} for these split groups, including bijectivity in 
part (c). For possibly non-split classical groups (with $p > 2$ but not general spin groups), 
it seems likely that Hypothesis~\ref{as:1.8} follows from \cite{KSS}.

\begin{prop}\label{prop:1.9}
Fix a prime $p$ and a group $\mc G$ as before. Suppose that Hypothesis\autoref {as:1.8} holds for all 
$d \in \N$. Then Theorem\autoref {thm:1.1} holds for $\mc G (F)$, for any local function field $F$.
\end{prop}
\begin{proof}
Write $\zeta^{\mc G,F,F'}(\pi)$ as a subquotient of the parabolic induction of a supercuspidal 
representation $\sigma \boxtimes \rho_1 \boxtimes \cdots \boxtimes \rho_r$ of a Levi subgroup
of $\mc G (F')^+$. Since (normalized) parabolic induction preserves depth \cite[Theorem 5.2]{MoPr}, 
$\sigma$ and all the $\rho_i$ have depth $\leq d$. The Jordan block of $\zeta^{\mc G,F,F'}(\pi)$
consists of the Jordan block of $\sigma$ and some pairs $(\rho,a)$ where $\rho$ is an
unramified twist of $\rho_i$. By Hypothesis\autoref {as:1.8} all $\rho$ appearing in 
Jord$(\zeta^{\mc G,F,F'}(\pi))$ have depth $\leq \mc D (p,\mc G,d)$. We note that $\rho \in 
\Irr (\GL_m (F'))$ where $m \leq \mr{rk}(\mc G)$ by Theorem\autoref {thm:1.1}.a.
Hence every such $\rho$ is in the image of $\zeta^{\GL_m,F,F'}$ for the correct $m$. Then
$(\zeta^{\GL_m,F,F'})^{-1} \rho$ lies in Jord$(\pi)$, and we can conclude that \eqref{eq:1.42}
is in fact a bijection. 

The $L$-parameter $\phi_\rho$ of any $\rho$ from Jord$(\pi)$ has depth $\leq \mc D$, because the local Langlands correspondence
for general linear groups preserves depths \cite[Proposition 4.2]{ABPS1}.
Let $\mb W_F^r$ be the $r$-th filtration subgroup of the absolute Galois group of $F$. Recall from 
\cite[(3.5.1)]{Del} that the $D$-closeness of $F$ and $F'$ is reflected in a group isomorphism
\begin{equation}\label{eq:1.43}
\mb W_F / \mb W_F^{\mc D +} \cong \mb W_{F'} / \mb W_{F'}^{\mc D +} .
\end{equation}
Composition with \eqref{eq:1.43} transfers $\phi_\rho$ to a $L$-parameter for $\GL_{d_\rho}(F')$, say
$\zeta (\phi_\rho)$. When $F$ and $F'$ are very close (for instance $2^{d_\rho} \mc D$-close),
$\zeta (\phi_\rho)$ is indeed the $L$-parameter of $\zeta^{\GL_{d_\rho},F,F'}(\rho)$
\cite[Theorem 6.1]{ABPS1}. We note that we really can chose $F'$ that close to $F$: by
\cite{Del} such a field exists and the above works for any choice of $F'$ that is $\mc D$-close
to $F$. For such an $F'$ composition of the $L$-parameter of 
\[
\mathlarger{\mathlarger{\boxtimes}}_{(\rho',a) \in \mr{Jord}(\zeta^{\mc G,F,F'} (\pi))} \delta (\rho',a)
\]
with \eqref{eq:1.43} yields the $L$-parameter of
\[
\mathlarger{\mathlarger{\boxtimes}}_{(\rho,a) \in \mr{Jord}(\pi)} \delta (\rho,a) .
\]
By Theorem\autoref {thm:1.1}.a the former parameter has image in ${}^L \mc G$, hence so does the 
latter parameter. We define the latter to be the $L$-parameter of $\pi$, like in \eqref{eq:1.8}. 
Then parts (a) and (b) of Theorem\autoref {thm:1.1} hold for $\mc G (F)$.

The (partially defined) character $\epsilon_{\zeta^{\mc G,F,F'} (\pi)}$ is transferred, via 
\eqref{eq:1.42}, to a (partially defined) character $\epsilon_\pi$ of $S_\pi$. Moreover $\epsilon_\pi$ 
is $\mc G (F)$-relevant because $\epsilon_{\zeta^{\mc G,F,F'} (\pi)}$ is $\mc G (F')$-relevant. 

Suppose that two discrete series representations $\pi, \tilde \pi$ of $\mc G (F)^+$ have the same
Jordan block and the same $\epsilon$. Then their transfers to representations of $\mc G (F')$ also
share the same Jordan block and the same $\epsilon$. With Theorem\autoref {thm:1.1}.d we find
$\zeta^{\mc G,F,F'}(\pi) \cong \zeta^{\mc G,F,F'}(\tilde \pi)$. Then \eqref{eq:1.41} says that
$\pi \cong \tilde \pi$.\\
Similarly \eqref{eq:1.41} readily shows that Theorem\autoref {thm:1.1}.d carries over from
$\mc G (F')$ to $\mc G (F)$.
\end{proof}

\subsection{Parametrization of essentially square-integrable representations} \
\label{par:ess}
 
From now on $F$ can be any non-archimedean local field, but we need Hypothesis
\autoref{as:1.8} if $F$ has positive characteristic. 
 
We note that $\mr{Out}(\mc G)$ is trivial except for special orthogonal groups and general spin groups
associated to vector spaces of even dimension $2n$. Then (for $n \neq 2$)
\begin{equation}\label{eq:Out}
\mr{Out}(\mc G) \cong G^{\vee +} / G^\vee \cong \rO_{2n}(\C) / \SO_{2n}(\C) .
\end{equation}
When $\mc G$ is a form of $\SO_4$, we ignore its exceptional automorphisms and instead we use
\eqref{eq:Out} as a definition of $\mr{Out}(\mc G)$. In particular the two-element group 
\eqref{eq:Out} acts naturally on $\Irr (G)$ and on $\Phi_\enh (G)$.

\begin{thm}\label{thm:1.2}
Fix a Whittaker datum for the quasi-split inner form of $G$.
\enuma{
\item Suppose that $\mr{Out}(\mc G)$ is trivial. There exists a canonical injection
\begin{itemize}
\item from the set of discrete series representations of $G$,
\item to the set of discrete bounded parameters in $\Phi_\enh (G)$.
\end{itemize}
\item Suppose that $\mr{Out}(\mc G)$ is nontrivial. There exists an injection 
\begin{itemize}
\item from the set of discrete series representations of $G$,
\item to the set of discrete bounded parameters in $\Phi_\enh (G)$,
\end{itemize}
which intertwines the actions of $\mr{Out}(\mc G)$. The induced injection between $\mr{Out}(\mc G)$-orbits
in these two sets is canonical.
\item The injection in parts (a) and (b) send supercuspidal unitary $G$-representations to
bounded cuspidal $L$-parameters, and non-supercuspidal representations to non-cuspidal
enhanced $L$-parameters.
}
\end{thm}
\begin{proof}
(a) If we apply the local Langlands correspondence for $\GL_{d_\rho}(F)$ to a $\rho$ occurring in Jord$(\pi)$, we obtain
$\phi_\rho \in \Irr (\mb W_F)$. The property \eqref{eq:1.44} translates to 
\begin{equation}\label{eq:1.45}
\phi_\rho \cong \phi_\rho^\vee \otimes \phi_{\nu_\pi} .
\end{equation} 
From \eqref{eq:1.45}, \eqref{eq:1.51} and Theorem\autoref {thm:1.1} we see that
\[
\{ (\phi_\rho,a) : (\rho,a) \in \mr{Jord}(\pi) \}
\]
is the set of Jordan blocks of some $\phi \in \Phi (G)$ with
\begin{equation}\label{eq:1.46}
\phi^\vee \otimes \mu_G^\vee \circ \phi \cong \phi \cong \phi^\vee \otimes \phi_{\nu_\pi} .
\end{equation}
Further $\phi$ is unique by \cite[Theorem 8.1]{GGP}, and discrete because Jord$(\pi)$ does not 
have multiplicities. As $\rho$ (from above) was unitary supercuspidal and in particular tempered, 
$\phi_\rho$ is bounded and therefore $\phi$ is also bounded. 

Under the correspondence Jord$(\pi) \mapsto \mr{Jord}(\phi)$, the group $\mc S_\pi$ 
becomes $\mc S_\phi$. The set of $G$-relevant characters of $\mc S_\phi$ is naturally
in bijection with the set of partially defined characters $\epsilon_\pi$ of $\mc S_\pi$ 
which figures in Theorem\autoref {thm:1.1}.d. Thus we can define the required injection by sending
$\pi$ to $(\phi,\epsilon)$ such that the local Langlands correspondence for $\GL_m$ sends
\begin{equation}\label{eq:1.8}
(\mr{Jord}(\pi), \epsilon_\pi) \quad \text{to} \quad (\mr{Jord}(\phi), \epsilon) .
\end{equation}
(b) The proof of part (a) applies perfectly well to the disconnected reductive group
$G^+$. It provides a canonical injection from the discrete series representations $\pi^+$
of $G^+$ to the pairs $(\phi,\epsilon)$ with $\phi \in \Phi (G) / \mr{Out} (\mc G)$ 
bounded and discrete and $\epsilon \in \mc S^+_\phi$, where $\mc S^+_\phi$ is like $\mc S_\phi$
but computed in $\rO_{2n} (\C)$. We can distinguish two cases:

\noindent
\textbf{Case (i): There exists $(\tau,a) \in \mr{Jord}(\phi)$ with $a \dim \tau$ odd.}\\ 
From \eqref{eq:1.3} we see that the group $\mc S_\phi^+$ contains an element of 
$\rO_{2n}(\C) \setminus \SO_{2n}(\C)$. Hence the preimage of $\phi$ in $\Phi (G)$
is just one equivalence class.

By Theorem\autoref {thm:1.1}.c, $\pi^+ \in \Irr (G^+)$ restricts to an irreducible 
representation $\pi$ of $G$. In particular $\pi$ is stable 
under $\mr{Out}(\mc G)$. Clifford theory tells us that there are precisely two inequivalent
irreducible representations of $G^+$ that restrict to $\pi$.  

As $\mc S_\phi^+ \cong \mh F_2^{|\mr{Jord}(\phi)|}$, we find 
\[
\mc S_\phi^+ \cong \mc S_\phi \times \mh F_2 .
\]
Hence there exist precisely two characters of $\mc S_\phi^+$ that extend $\epsilon |_{\mc S_\phi}$.
We decree that the bijection for the discrete series of $G$ sends $\pi$ to
$(\phi,\epsilon |_{\mc S_\phi})$, that is the only natural possibility and does not disturb the
injectivity we had for $G^+$.

\noindent
\textbf{Case (ii): $a \dim \tau$ is even for all $(\tau,a) \in \mr{Jord}(\phi)$.}\\
Now \eqref{eq:1.3} shows that $\mc S_\phi^+$ does not contain any elements from $\rO_{2n}(\C)
\setminus \SO_{2n}(\C)$, so $\mc S_\phi^+ = \mc S_\phi$. By \cite[Theorem 8.1]{GGP} the preimage of 
$\phi$ in $\Phi (\GSpin (V))$ consists of two equivalence classes, say $\phi'$ and $\phi''$.
Then $\phi''$ is equivalent with $\mr{Ad}(h^\vee) \phi'$ for some $h^\vee \in\rO_{2n}(\C)
\setminus \SO_{2n}(\C)$ and $\mc S_{\phi'}$ is canonically isomorphic with $\mc S_{\phi''}$.

By Theorem\autoref {thm:1.1}.c the restriction of $\pi^+$ to $G$ is reducible. By Clifford 
theory it is the direct sum of two inequivalent irreducible $G$-representations say 
$\pi' \oplus \pi''$, and any element of $G^+ \setminus G$ exchanges $\pi'$ and $\pi''$.
\begin{equation}\label{eq:1.9} 
\text{We choose a bijection between } \{(\phi',\epsilon), (\phi'',\epsilon)\} \text{ and }
\{ \pi', \pi'' \} ,
\end{equation}
and we decree that it gives two instances of the injection for the discrete series of $G$.
Notice that this guarantees $\mr{Out}(\mc G)$-equivariance on these objects.

Combining all instances, we obtain the desired injection for the discrete series of
$G$. Its only noncanonical parts are the choices \eqref{eq:1.9}, which become invisible
when we pass to $\mr{Out}(\mc G)$-orbits. \\
(c) This is clear from the criteria for cuspidality on page \pageref{eq:1.7} and in
Theorem\autoref {thm:1.1}.e.
\end{proof}

For the moment $G$ is a general spin group. Since the centre of $G$ is not compact 
(unlike for the other groups in Section~\ref{sec:Moe}),  we have to distinguish
between discrete series representations and essentially square-integrable representations.
A $G$-representation $\pi$ is called essentially square-integrable if its restriction to
$G_\der$ is square-integrable. If $\pi$ is in addition irreducible, then there exists an unramified
character $\chi \in X_\nr (G)$ such that $\chi \otimes \pi$ has unitary central character, that 
is, $\chi \otimes \pi$ belongs to the discrete series. We can even achieve this with $\chi$
a real power of the norm character of $F^\times \cong G / G_\der$.

Recall from \cite{Hai} that the group $X_\nr (G)$ of unramified characters of $G$ is naturally
isomorphic with $(\rZ(G^\vee)^{\circ,\mb I_F})_{\mb W_F}$, which for our $G$ is just
$\rZ(G^\vee)^\circ \cong \C^\times$. Similarly the group $X_\unr (G)$ of unitary unramified
characters is naturally isomorphic with the maximal compact subgroup $\rZ(G^\vee)^\circ_\cpt$
of $\rZ(G^\vee)^\circ$. The group $X_\nr (G)$ acts on $\Irr (G)$ by tensoring and the group
$\rZ(G^\vee)^{\mb I_F} = \rZ(G^\vee)$ acts on $\Phi_\enh (G)$ by 
\[
z (\phi,\rho) = (z \rho,\phi), \qquad (z\phi) |_{\mb I_F \times \SL_2 (\C)} =
\phi |_{\mb I_F \times \SL_2 (\C)}, \qquad (z\phi)(\Fr_F) = z \, \phi (\Fr_F) .
\]

\begin{thm}\label{thm:1.3}
Let $G$ be a general spin group.
\enuma{
\item The injection in Theorem\autoref {thm:1.2}.a is equivariant for the actions of 
$X_\unr (G) \cong \rZ(G^\vee)^\circ_\cpt$, and by suitable choices the bijection in Theorem 
\autoref{thm:1.2}.b can be made equivariant for these actions.
\item The injection from part (a) extends canonically to an injection 
\begin{itemize}
\item from the set of irreducible essentially square-integrable $G$-representations,
\item to the set of discrete parameters in $\Phi_\enh (G)$.
\end{itemize}
\item The injection in part (b) is equivariant for the actions of $X_\nr (G) \cong \rZ(G^\vee)^\circ$,
and it respects cuspidality.
}
\end{thm}
\begin{proof}
(a) Let $\pi \in \Irr (G)$ and $(\phi,\epsilon) \in \Phi_\enh (G)$ be as in the proof of Theorem 
\autoref{thm:1.2}. For $\chi \in X_\unr (G)$, $\chi \otimes \pi$ is of the same kind. 
From the natural isomorphisms
\[
(\chi \otimes \mr{St}(\rho,a)) \times (\chi \otimes \pi) \cong 
\chi \otimes (\mr{St}(\rho,a) \times \pi)
\]
we see that
\[
\mr{Jord}(\chi \otimes \pi) \quad \text{equals}  \quad
\{ (\chi \otimes \rho, a) : (\rho,a) \in \mr{Jord}(\pi) \} =: \chi \otimes \mr{Jord}(\pi) .
\]
The properties of $\epsilon_\pi$ in \cite[\S 2.5]{Moe2} readily imply that
\[
\epsilon_{\chi \otimes \pi} (z_{\chi \otimes \rho,a}) = \epsilon_\pi (z_{\rho,a}) .
\]
Let $\hat \chi \in \rZ(G^\vee)^\circ$ correspond to $\chi$ via $X_\unr (G) \cong \rZ(G^\vee)^\circ_\cpt$.
Then $\hat \chi \phi$ is still discrete and bounded, while
\[
\mr{Jord}(\hat \chi \phi) \quad \text{equals} \quad
\{ (\hat \chi \tau,a) : (\tau,a) \in \mr{Jord}(\phi) \} =: \hat \chi \mr{Jord}(\phi) .
\]
The action of $\hat \chi$ does not change $\epsilon$ as character of 
\[
\mc S_\phi = \rZ_{{G^\vee}_\der}(\phi) = \rZ_{{G^\vee}_\der}(\hat \chi \phi) = \mc S_{\hat \chi \phi} .
\]
However, the element $z_{\tau,a} \in \mc S_\phi$ is renamed as $z_{\hat \chi \tau,a}$ and to account
for that we rename $\epsilon$ to $\hat \chi \epsilon$.

Suppose now that $\pi$ and $(\phi,\epsilon)$ are matched by Theorem\autoref {thm:1.2}, so \eqref{eq:1.8}
holds. By the known equivariance properties of the local Langlands correspondence for $\GL_m$, 
\[
(\chi \otimes \mr{Jord}(\pi), \epsilon_{\chi \otimes \pi}) \quad \text{is sent to} \quad
(\hat \chi \mr{Jord}(\phi), \hat \chi \epsilon) .
\]
In the setting of Theorem\autoref {thm:1.2}.a, this shows that $\chi \otimes \pi$ is matched with
$(\hat \chi \phi, \hat \chi \epsilon)$, which is the desired equivariance.

In the setting of Theorem\autoref {thm:1.2}.b, only the choices in \eqref{eq:1.9} could disturb this
equivariance for $X_\unr (G) \cong \rZ(G^\vee)^\circ_\cpt$. To prevent that, it suffices to make the 
entirety of the choices \eqref{eq:1.9} in an equivariant way. This can be done as follows. Pick a 
set of representatives for the $(\phi,\epsilon)$ with all $a \dim \tau$ even, modulo the action of 
$\rZ(G^\vee)^\circ_\cpt$. For each of these $\phi$'s we fix a choice \eqref{eq:1.9}, say 
$(\phi',\epsilon) \mapsto \pi'$. Then decree that, for each $\chi \in X_\unr (G)$, 
$(\hat \chi \phi',\epsilon)$ is matched via $\chi \otimes \pi'$. \\
(b) By design, the set of essentially square-integrable irreducible $G$-representations can be
expressed as
\begin{equation}\label{eq:1.10}
\text{discrete series of } G \times_{X_\unr (G)} X_\nr (G) .
\end{equation}
Similarly, it follows from \cite[Lemma 5.1]{Hei1} that the set of discrete parameters in $\Phi_\enh (G)$
can be constructed as
\begin{equation}\label{eq:1.11}
\text{bounded discrete part of } \Phi_\enh (G) \times_{\rZ(G^\vee)^\circ_\cpt} \rZ(G^\vee)^\circ .
\end{equation}
From \eqref{eq:1.10}, \eqref{eq:1.11} and part (a) we deduce an injection from the set of 
essentially square-integrable irreducible $G$-representations to the discrete part of $\Phi_\enh (G)$,
which is equivariant for $X_\nr (G) \cong \rZ(G^\vee)^\circ$.\\
(c) The actions of $X_\nr (G)$ on $\Irr (G)$ and of $\rZ(G^\vee)^\circ$ on $\Phi_\enh (G)$ preserve
cuspidality. Combine that with Theorem\autoref {thm:1.2}.c and the construction of the injection in
part (b). 
\end{proof}

Now $G$ can again be any group as in Section~\ref{sec:Moe}. The set of supercuspidal Bernstein 
components of $\Irr (G)$ is just $\Irr_\cusp (G) / X_\nr (G)$. Recall the notion of a Bernstein 
component of enhanced $L$-parameters from \cite[\S 8]{AMS1}. By definition, the set of 
cuspidal Bernstein components of $\Phi_\enh (G)$ is $\Phi_\cusp (G) / \rZ(G^\vee)^\circ$.
If we apply Theorems\autoref {thm:1.2} and\autoref {thm:1.3} to these sets, we obtain:

\begin{cor}\label{cor:1.4}
Theorems\autoref {thm:1.2} and\autoref {thm:1.3}.b induce an injection
\begin{itemize}
\item from the set of supercuspidal Bernstein components of $\Irr (G)$,
\item to the set of cuspidal Bernstein components of $\Phi_\enh (G)$.
\end{itemize}
This injection is $\mr{Out}(\mc G)$-equivariant and becomes canonical if we pass to $\mr{Out}(\mc G)$-orbits
on both sides.
\end{cor}

\section{Comparison of Hecke algebras for Bernstein components} 
\label{sec:Hecke}

In this section $G$ is a general spin group. All our results are also valid for symplectic
groups and for (special) orthogonal groups, with slightly simpler proofs, see \cite{Hei2,Hei3,Hei4}
(on the $p$-adic side) and \cite{Mou} and \cite[\S 5.3]{AMS3} (on the Galois side). 
In \cite[\S 10]{SolEnd} an extended affine Hecke algebra
\[
\mc H (\mf s) = \End_G ( \Pi_{\mf s})
\]
was attached to $\mf s$, where $\Pi_{\mf s}$ is a particular progenerator of $\Rep (G)^{\mf s}$. 
We have $\Pi_{\mf s}= I_P^G\Pi_{\mf{s}_L}$, where $I_P^G$ is the (normalized) parabolic 
induction functor for $P$  a parabolic subgroup of $G$ with Levi factor $L$. Further 
$\Pi_{\mf{s}_L} := \mr{ind}_{L^1}^L (\sigma_1)$, with $L^1$ the subgroup of $L$ generated 
by all compact subgroups and $\sigma_1$ an irreducible constitutent of $\Res^L_{L^1} \sigma$. 
It is shown in \cite[\S VI.10.1]{Ren} that $\Pi_{\mf s}$
is canonical, in the sense that
\begin{equation}\label{eq:3.50}
\text{up to isomorphism, } \Pi_{\mf s} \text{ depends only on } \mf s.
\end{equation}
We note that \cite[\S 10]{SolEnd}
is applicable because the restriction of $\sigma$ to $L^1$ is multiplicity-free, which
follows from the fact that $\mc L$ is a direct product of reductive groups with centre
of dimension $\leq 1$. In this setup there is a natural equivalence of categories
\begin{equation}\label{eq:3.54}
\begin{array}{ccc}
\Rep (G)^{\mf s} & \isom & \Mod (\End_G (\Pi_{\mf s})^{\op} ) \\
\pi & \mapsto & \Hom_G (\Pi_{\mf s}, \pi) \\
V \otimes_{\End_G (\Pi_{\mf s})} \Pi_{\mf s} & \text{\reflectbox{$\mapsto$}} & V 
\end{array},
\end{equation}
see \cite[Th\'eor\`eme VI.10.2]{Ren}. We will explain the structure of 
$\mc H (\mf s) := \End_G (\Pi_\mf s)$ in terms of the data and the presentation at the start 
of Paragraph~\ref{par:HAL}. From there we will see that it is self-opposite, and we will 
compare it with $\mc H (\mf s^\vee,z)$.

Before we come to that, we need to match Bernstein components for $\Irr (G)$ and for $\Phi_\enh (G)$. 
Suppose that the bilinear form on $V$ is given by a symmetric matrix $\hat J$, such that
the isotropic part is made from blocks $\matje{0}{1}{1}{0}$ placed in rows and columns
$j, \dim V + 1 - j$. Let $\mu_G : G \to F^\times$ be the spinor norm, so that Spin$(V) = 
\ker \mu_G$. The Levi subgroup $L = \mc L (F)$ is embedded in $G = \GSpin (V)$ via
\begin{equation}\label{eq:3.59}
\begin{aligned}
& G_{n_-} / F^\times \times \GL_{n_1}(F) \times \cdots \times \GL_{n_k}(F) \to G / F^\times
\cong \SO (V) \\
& (g_-,g_1,\ldots,g_k) \mapsto \big( g_1,\ldots, g_k, g_-, 
\hat J g_k^{-T} \hat J^{-1}, \ldots , \hat J g_k^{-T} \hat J^{-1} \big) 
\end{aligned}
\end{equation}
It is difficult to write down the actual embedding $L \to G$ in such terms, to study that the root
datum from \cite{AsSh} is more useful. The group $\rN_{\mr{GPin} (V)}(L)$ contains an element that 
exchanges $g_j$ and $\hat J g_j^{-T} \hat J^{-1}$, and the same time multiplies
$g_-$ with $\det (g_j)$. As automorphism of $L$ , it is given by 
\begin{equation}\label{eq:1.23}
(g_-,g_1,\ldots,g_k) \mapsto
(\det (g_j) g_-,g_1,\ldots,g_{j-1},\hat J g_j^{-T} \hat J^{-1}, g_{j+1},\ldots,g_k ).
\end{equation}
We record that the effect of \eqref{eq:1.23} on irreducible representations is
\[ 
\sigma_- \boxtimes \sigma_1 \boxtimes \cdots \boxtimes \sigma_k \mapsto
\sigma_- \boxtimes \sigma_1 \boxtimes \cdots \boxtimes \sigma_{j-1} \boxtimes (\sigma_j^\vee
\otimes \nu_{\sigma_-} \circ \det) \boxtimes \sigma_{j+1} \boxtimes \cdots \boxtimes \sigma_k ,
\]
where $\nu_{\sigma_-}$ is character by which the central subgroup $F^\times \subset G_{n_-}$
acts on $\sigma_-$.

Further $\rN_G (L)$ contains elements that act on $L$ by permuting some type GL factors of the
same size. The group $\rN_{\mr{GPin}(V)} (L) / L$ is generated by elements of these two kinds, 
and is isomorphic to a direct product of Weyl groups of type $B_e$. For $\rN_G (L)/L$ the only
difference is that the elements from \eqref{eq:1.23} are subject to a determinant condition if 
dim$(V)$ is even. Notice that these descriptions match those after \eqref{eq:1.20}. Thus there
are canonical isomorphisms
\begin{equation}\label{eq:1.24}
\rN_G (L)/L \cong \rN_{G^\vee}(L^\vee)/L^\vee \quad \text{and} \quad
\rN_{\mr{GPin}(V)}(L)/L \cong \rN_{G^{\vee +}}(L^\vee) / L^\vee .
\end{equation}

\begin{thm}\label{thm:1.5}
\enuma{
\item There exists a injection 
\begin{itemize}
\item from the set of supercuspidal Bernstein components of $\Irr (L)$,
\item to the set of cuspidal Bernstein components of $\Phi_\enh (L)$.
\end{itemize}
This bijection is equivariant for the natural actions of \eqref{eq:1.24} and becomes canonical
if we pass to $\mr{Out}(\mc G_{n_-})$-orbits.
\item Let $L$ run through a set of representatives for the conjugacy classes of Levi subgroups
of $G$. The corresponding instances of part (a) provide an injection 
\begin{itemize}
\item from the set of Bernstein components of $\Irr (G)$,
\item to the set of Bernstein components of $\Phi_\enh (G)$.
\end{itemize}
This injection becomes canonical if we pass to $\mr{Out}(\mc G)$-orbits.
}
\end{thm}
\begin{proof}
(a) The injection and the canonicity follow from Corollary\autoref {cor:1.4}, while the equivariance
can be seen from our explicit formulas for the actions of \eqref{eq:1.24}, namely \eqref{eq:1.13},
\eqref{eq:1.25} and \eqref{eq:1.23}.\\
(b) By definition Bernstein components of $\Irr (G)$ are parametrized by supercuspidal Bernstein
components for Levi subgroups of $G$. Further $\mf s_L \subset \Irr_\cusp (L)$ and 
$\mf s_{L'} \subset \Irr_\cusp (L')$ give the same Bernstein component for $\Irr (G)$ if and 
only if $\mf s_L$ and $\mf s_{L'}$ are $G$-conjugate. Analogous statements hold for Bernstein
components of $\Phi_\enh (G)$ \cite[\S 8]{AMS1}, which yields the desired bijection. By the 
equivariance in part (a), this bijection does not depend on the choice of the representative
Levi subgroups.
\end{proof}

With Theorem\autoref {thm:1.2} we consider $\sigma = \pi (\phi_\sigma,\epsilon_\sigma) \in \Irr_\cusp (L)$. 
Then $\mf s_L = X_\nr (L) \sigma$ is the image of $\mf s_L^\vee$ under Theorem\autoref {thm:1.5}.b.
The injectivity and $X_\nr (L)$-equivariance in Theorem\autoref {thm:1.3} say that this extends to 
an injection from $\mf s_L$ to $\mf s_L^\vee$. Then the equivariance in Theorem\autoref {thm:1.5}.a 
guarentees that the groups $W_{\mf s}$ and $W_{\mf s^\vee}$ are canonically isomorphic.

We may assume that $\sigma$ has been normalized like $\phi$ after \eqref{eq:1.26}. Then
the group $W_{\mf s^\vee}$ can also be described as the stabilizer of 
\begin{equation}\label{eq:1.27}
\sigma  \; = \; \sigma_- \: \boxtimes \: 
\mathlarger{\mathlarger{\boxtimes}}_\rho \: \rho^{\boxtimes e_\rho} \; = \; 
\sigma_- \: \boxtimes \: 
\mathlarger{\mathlarger{\boxtimes}}_j \: {\rho_j}^{\boxtimes e_j} .
\end{equation}
in $\rN_G (L) / L$. The stabilizer $W_{\mf s}^+$ of $\sigma$ in $\rN_{\mr{GPin}(V)}(L)/L$
decomposes as a direct product of subgroups $W_{\mf s,\rho}$. From \eqref{eq:1.23} we see that
\begin{itemize}
\item if $\rho \not\cong \rho^\vee \otimes \nu_{\sigma_-} \circ \det$, then 
$W_{\mf s,\rho} \cong S_{e_\rho} \cong W (A_{e_\rho - 1})$,
\item if $\rho \cong \rho^\vee \otimes \nu_{\sigma_-} \circ \det$, then 
$W_{\mf s,\rho} \cong W(B_{e_\rho}) = W(C_{e_\rho})$, which can sometimes be interpreted
better as $W(D_{e_\rho}) \rtimes \mr{Aut}(D_n)$.
\end{itemize}
Using Theorem\autoref {thm:1.5}, we can express the complex torus underlying $\mc H (\mf s)$ as 
\begin{equation}\label{eq:3.18}
T_{\mf s} = \mf s_L \cong \mf s_L^\vee = T_{\mf s^\vee},
\end{equation}
Further, the action of $W_{\mf s}^+$ on $T_{\mf s}$ can be identified with the action of
$W_{\mf s^\vee}^+$ on $T_{\mf s^\vee}$. The map $\chi \mapsto \sigma \otimes \chi$ provides a
homeomorphism
\[
T_{\mf s} \cong X_\nr (L) / X_\nr (L,\sigma) ,
\]
where $X_\nr (L,\sigma) \subset X_\nr (L)$ is the stabilizer of $\sigma \in \Irr_\cusp (L)$.
Hence
\[
X^* (T_{\mf s}) \cong L_\sigma / L^1 
\text{ where } L_\sigma = \bigcap\nolimits_{\chi \in X_\nr (L,\sigma)} \ker \chi .
\]
More explicitly, $L/L^1 \cong \Z \times \prod_j \Z^{e_j}$ and 
\[
L_\sigma / L^1 \text{ is the subgroup } \Z \times \prod\nolimits_j (t_{\rho_j} \Z )^{e_j} ,
\]
where the first factor $\Z$ comes from $\rZ(\mc G)^\circ \cong \GL_1$. We recall that $t_\rho$ 
denotes the torsion number of $\rho$, that is, the number of unramified characters of 
$\GL_{d_\rho}(F)$ that stabilize $\rho \in \Irr (\GL_{d_\rho}(F))$.
We write the root datum for $\mc H (\mf s)$ as
\[
\mc R_{\mf s} = \big( \Sigma_{\mf s}, X^* (T_{\mf s}), \Sigma_{\mf s}^\vee, X_* (T_{\mf s}) \big).
\]
As explained in \cite[\S 3]{SolEnd}, the root system $\Sigma_{\mf s}^\vee$ comes from the roots 
$\alpha \in \Sigma_\red (G, \rZ(L))$ for which the so-called Harish-Chandra $\mu$-function $\mu^\alpha$
has a zero on $\mf s_L$. Then $\Sigma_{\mf s}$ consists of multiples of some elements of
$\Sigma_\red (G, \rZ(L))^\vee \cong \Sigma_\red (G^\vee,\rZ(L^\vee))$, just like $R_{\mf s^\vee}$ in
\eqref{eq:1.30}.

The group $W_{\mf s}^+$ acts naturally on $\mc R_{\mf s}$ and contains $W(\Sigma_{\mf s})$.
Our choice of a Borel subgroup $B^\vee$ of $G^\vee$ yields a system of positive roots 
$\Sigma_{\mf s}^+$ in $\Sigma_{\mf s}$. If $\Gamma_{\mf s}^{(+)}$ denotes the stabilizer of
$\Sigma_{\mf s}^+$ in $W_{\mf s}^{(+)}$, then 
\begin{equation}\label{eq:1.28}
W_{\mf s}^+ = \Gamma_{\mf s}^+ \ltimes W(\Sigma_{\mf s}) \quad \text{and} \quad
W_{\mf s} = \Gamma_{\mf s} \ltimes W(\Sigma_{\mf s}) .
\end{equation}
To match this decomposition with \eqref{eq:1.29}, we need to compare the underlying root systems.
In \cite[\S 3]{SolEnd} an element 
\begin{equation}\label{eq:1.32}
h_\alpha^\vee \in (L_\sigma \cap L_\alpha^1) / L^1 \subset L_\sigma / L^1 = X^* (T_{\mf s}) 
\end{equation}
was associated to each $\alpha \in  \Sigma_\red (G, \rZ(L))$. Here $L_\alpha$ is the Levi subgroup
of $G$ which contains $L$ and the root subgroups $\rU_{\alpha'}$ (for $\alpha' \in R (G,S)$ with
$\alpha' |_{\rZ(L)} \in \Q \alpha$) and whose semisimple rank is one higher than that of $L$.
In fact $(L_\sigma \cap L_\alpha^1) / L^1 \cong \Z$, $h_\alpha^\vee$ generates this group and
is pinned down by the requirement $\nu_F (\alpha (h_\alpha^\vee)) > 0$. Then
\begin{equation}\label{eq:1.31}
\Sigma_{\mf s} = \{ h_\alpha^\vee : \mu^\alpha \text{ has a zero on } \mf s_L \} .
\end{equation}
Recall that $R_{\mf s^\vee}$ is a disjoint union of irreducible root systems 
\[
R_{\mf s^\vee,\tau} = R (G^\vee_{\phi,\tau}T,T) = R (G^\vee_{\phi (\mb I_F),\tau} T,T)
\]
which are given explicitly in Table~\ref{tab:2}. Similarly, by \cite[Proposition 1.13]{Hei3}, 
$\Sigma_{\mf s}$ is a disjoint union of irreducible root systems $R_{\mf s,\rho}$, each one
coming from the factors $\GL_{m_j}(F)$ of $L$ with $\sigma_j = \rho$. By 
\cite[Proposition 1.15]{Hei3} (generalized to our setting) the groups
$W_{\mf s}^+$ and $\Gamma_{\mf s}^+$ decompose canonically as direct products of subgroups 
$W_{\mf s,\rho}^+$ and $\Gamma_{\mf s,\rho}^+$.

We fix one $\rho$ and we let $\tau \in \Irr (\mb W_F)$ be its image under the local Langlands correspondence for 
$\GL_{m_j}(F)$. By design $e_\tau = e_\rho > 0$. Recall from \eqref{eq:1.33} and \eqref{eq:1.21}
that $W_{\mf s^\vee}^+$ and $\Gamma_{\mf s^\vee}^+$ decompose canonically as direct products 
of subgroups $W_{\mf s^\vee,\tau}^+$ and $\Gamma_{\mf s,\tau}^+$. By Theorem\autoref {thm:1.5}.a
\begin{equation}\label{eq:1.34}
W_{\mf s,\rho}^+ \cong W_{\mf s^\vee,\tau}^+ 
\qquad \text{for all } \tau \in \Irr (\mb W_F) \text{ with } e_\rho > 0. 
\end{equation}
Since Jord$(\sigma_-)$ and Jord$(\phi_-)$ correspond via the local Langlands correspondence, $\ell_\tau > 0$ if and only if 
$\rho$ appears in Jord$(\sigma_-)$. We write 
\[
a_\rho = \max \{ a : (\rho,a) \in \mr{Jord}(\sigma_-) \} ,
\]
which equals $a_\tau$. Let $\rho'$ correspond to $\tau'$ via the local Langlands correspondence for $\GL_{d_\rho}(F)$,
so $\rho'$ is an unramified twist of $\rho$ which is not isomorphic to $\rho$, but still
$\rho' \cong \rho^{'\vee} \otimes \nu_{\sigma_-} \circ \det$.

\begin{prop}\label{prop:1.6}
There is a canonical bijection $R_{\mf s,\rho} \to R_{\tau,\red}$ which respects positivity of
roots. In particular $W(R_{\mf s,\rho}) \cong W (R_\tau)$ and 
$\Gamma_{\mf s,\rho}^+ \cong \Gamma_{\mf s^\vee,\tau}^+$.
\end{prop}
\begin{proof}
The proof of \cite[Proposition 1.13]{Hei3} shows that for $\GL_{d_\rho}(F)^{e_\rho} \subset L$ 
the roots $\alpha : t \mapsto t_i t_j^{-1}$ with $1 \leq i,j \leq e_\rho, i \neq j$ can be 
treated entirely like roots for some general linear group. Hence the associated function
$\mu^\alpha$ has a zero on $\mf s_L$ and $h_\alpha^\vee \in \Sigma_{\mf s}$. Thus $R_{\mf s,\rho}$
always contains a root subsystem of type $A_{e_\rho - 1}$. In terms of $\mc R_{\mf s}$:
the corresponding part of $X^* (T_{\mf s})$ can be identified with $(t_\rho \Z)^{e_\rho}$
and $A_{e_\rho - 1}$ is embedded there as the elements $h_\alpha^\vee = t_\rho \alpha^\vee$
with $\alpha^\vee \in \Z^{e_\rho}$ of the usual form 
\[
(0,\ldots,0,1,0,\ldots,0,-1,0,\ldots,0) .
\]
In view of the description of $W_{\mf s^\vee}^+$ following \eqref{eq:1.27}, $R_{\mf s,\rho}$ is
a $S_{e_\rho}$-stable reduced root subsystem of $BC_{e_\rho}$. In other words, it has type
$A_{e_\rho - 1}$, $B_{e_\rho}$, $C_{e_\rho}$ or $D_{e_\rho}$. We check all the cases in Table~\ref{tab:1}. 
\begin{itemize}
\item $\tau \in \Irr (\mb W_F)^0_\phi, \ell_\tau = 0$. Then $\rho \not\cong \rho^\vee \otimes
\nu_{\sigma_-} \circ \det$ and from \eqref{eq:1.23} we see that 
$R_{\mf s,\rho} \cong A_{e_\rho - 1} \cong R_\tau$.
\item $\tau \in \Irr (\mb W_F)^\pm_\phi , \ell_\tau > 0$. Here $\rho \cong \rho^\vee 
\otimes \nu_{\sigma_-} \circ \det$ and $a_\tau = a_\rho > 0$. Since  $(\rho,a_\rho) \in 
\mr{Jord}(\sigma_-)$, $\rho \otimes |\cdot|^{(a_\rho + 1)/2} \times \sigma_-$ is reducible 
\cite[\S 3.2]{Moe3}. Hence the automorphism \eqref{eq:1.23} comes from a root $\alpha$ for 
which $\mu^\alpha$ has a zero on $T_{\mf s}$. In the picture \eqref{eq:1.32} that becomes 
$h_\alpha^\vee = t_\rho \alpha^\vee$ with $\alpha^\vee$ a standard basis vector of 
$\Z^{e_\rho}$. In particular $R_{\mf s,\rho}$ has type $B_{e_\rho}$, just like $R_{\tau,\red}$.
\item $\tau \in \Irr (\mb W_F)^-_\phi, \ell_\tau = 0$. Again $\rho \cong \rho^\vee 
\otimes \nu_{\sigma_-} \circ \det$, but now $\rho$ does not occur in Jord$(\sigma_-)$. Still
\eqref{eq:1.23} fixes $\rho$, and by \cite[p. 1610]{Hei2} $\rho \otimes |\cdot|^{1/2} \rtimes 
\sigma_-$ is reducible. This is like the previous case, only with $a_\tau = a_\rho = 0$.
Notice that $\ell_{\tau'} = a_{\tau'} = a_{\rho'} = 0$ as well. Again we find 
$R_{\mf s,\rho} \cong B_{e_\rho}$, while $R_\tau \cong C_{e_\rho}$.
\item $\tau \in \Irr (\mb W_F)^+_\phi, \ell_\tau = 0$. Now \eqref{eq:1.23} fixes 
$\rho \cong \rho^\vee \otimes \nu_{\sigma_-} \circ \det$ although $\rho$ does not occur in
Jord$(\sigma_-)$. By \cite[p. 1610]{Hei2}, $\rho \times \sigma_-$ is reducible. By our 
assumptions on $\sigma$, $\ell_{\tau'} = 0$, so $\rho' \times \sigma_-$ is also reducible.
Then the shape of $\mu^\alpha$ \cite[(3.7)]{SolEnd} entails that $\mu^\alpha$ is constant
on $T_{\mf s}$, for $\alpha$ associated to \eqref{eq:1.23}. Hence $R_{\mf s,\rho}$ does not
contain short roots from $B_{e_\rho}$ or long roots from $C_{e_\rho}$. 

Consider a root in $D_{e_\rho} \setminus A_{e_\rho - 1}$, so of the form 
\[
\beta = (0,\ldots,0,1,0,\ldots,0,1,0,\ldots,0) .
\]
Via a suitable reflection $s_\alpha$ with $\alpha$ as before, $\beta$ is associate to a root 
$\beta' \in A_{e_\rho - 1}$. Since $s_\alpha \in W_{\mf s}^+$, $\mu^\beta = \mu^{\beta'} \circ
s_\alpha$. As $\mu^{\beta'}$ has a zero on $T_{\mf s}$, so does $\mu^\beta$. Therefore 
$R_{\mf s,\tau}$ contains 
\[
h_\beta^\vee = t_\rho \beta^\vee = t_\rho (0,\ldots,0,1,0,\ldots,0,1,0,\ldots,0) ,
\]
and $R_{\mf s,\rho} \cong D_{e_\rho} \cong R_\tau$.
\end{itemize}
In all cases there is indeed a natural bijection $R_{\mf s,\rho} \to R_{\tau,\red}$: the 
identity on all roots except the short roots in the second case, those are multiplied by 2.
The bijection preserves positivity of roots, so it induces an isomorphism from the stabilizer
$\Gamma_{\mf s,\rho}^+ \subset W_{\mf s,\rho}^+$ of $R_{\mf s,\rho}^+$ to the stabilizer 
$\Gamma_{\mf s^\vee,\tau}^+ \subset W_{\mf s^\vee,\tau}^+$ of $R_\tau^+$.
\end{proof}

Now we analyse the $q$-parameters for $\mc H (\mf s)$.
In view of the shape of $\mu^\alpha$ \cite[(3.7) and (3.10)]{SolEnd}, the condition 
\eqref{eq:1.31} on $\alpha \in \Sigma_\red (G, \rZ(L))$ is equivalent with $q_\alpha > 1$, where 
$q_\alpha$ comes from $\mu^\alpha$ and will also be a $q$-parameter for $\mc H (\mf s)$. The 
parameter functions $\lambda, \lambda^* : \Sigma_{\mf s} \to \R_{\geq 0}$ and the parameters 
\[
q_\alpha = q_F^{(\lambda (\alpha) + \lambda^* (\alpha))/2} ,\quad 
q_\alpha^* = q_F^{(\lambda (\alpha) - \lambda^* (\alpha))/2} 
\]
were computed in \cite{Hei2,Hei3}. Although these papers were 
written for Sp$(V)$ and SO$(V)$, the same arguments apply in our setting, that was checked in 
\cite{SolEnd,SolParam}. The $q$-parameters on $R_{\mf s,\rho}$ are expressed in terms of
$t_\rho$ and $a_\rho$. More precisely, by \cite[Proposition 3.4]{Hei2} the $q$-parameters are:
\begin{itemize}
\item if $R_{\mf s,\rho} \cong B_{e_\rho}$ and $\alpha$ is a short root, then 
$q_\alpha = q_F^{t_\rho (a_\rho + 1)/2}$ and $q_\alpha^* = q_F^{t_\rho (a_{\rho'} + 1)/2}$,
\item otherwise $q_\alpha = q_F^{t_\rho}$ and $q_\alpha^* = 1$.
\end{itemize}
With $q_F$ as $q$-base that gives
\begin{itemize}
\item $\lambda (\alpha) = t_\rho (a_\rho + a_{\rho'} + 2)/ 2, \lambda^* (\alpha) =
t_\rho (a_\rho - a_{\rho'})/2$ if $\alpha$ is a short root in $B_{e_\rho}$,
\item $\lambda (\alpha) = \lambda^* (\alpha) = t_\rho$ otherwise.
\end{itemize}
In the case $\tau, \tau' \in \Irr (\mb W_F)^-_\phi, \ell_\tau + \ell_{\tau'} = 0$ we find 
$q_\beta = q_\beta^* = q_F^{t_\rho /2}$ for the short roots $h_\beta^\vee$ in $B_{e_\rho}$. 
As explained in \cite[proof of Theorem 4.9]{SolParam}, we may replace $h_\beta^\vee$ by a long 
root $(h_\beta^\vee)^2 = h_{\beta/2}^\vee$ of $C_{e_\rho}$, and simultaneously put
\begin{equation}\label{eq:3.19}
q_{\beta/2} = q_F^{t_\rho} ,\qquad q_{\beta/2}^* = 1 ,\qquad 
\lambda (\beta/2) = \lambda^* (\beta/2) = t_\rho .
\end{equation}
With that improvement, the bijection $R_{\mf s,\rho} \to R_{\tau,\red}$ in Proposition
\autoref{prop:1.6} becomes simply the restriction of the canonical bijection $X^* (T_{\mf s})
\to X_* (T_{\mf s^\vee})$ to reduced roots. That yields a canonical isomorphism of root data
\[
\mc R_{\mf s} \cong \mc R_{\mf s^\vee,\red} ,
\]
where the subscript red means that (for the involved non-reduced root systems $BC_e$) we take 
only the indivisible roots and the non-multipliable coroots. Comparing with page \pageref{eq:1.20}, 
we see that the parameter functions $\lambda,\lambda^*$ for $\mc H (\mf s)$ are the same as those 
for $\mc H (\mf s^\vee,z)$ with $z = q_F^{1/2}$. Thus we find a canonical isomorphism of affine
Hecke algebras
\begin{equation}\label{eq:1.35}
\mc H (\mc R_{\mf s}, \lambda, \lambda^*,q_F^{1/2}) \cong 
\mc H (\mc R_{\mf s^\vee}, \lambda, \lambda^*,q_F^{1/2}) .
\end{equation}
Here we have $q_F^{1/2}$ instead of $q_F$ because in the setup of \cite{AMS3} the indeterminate
$z^2$ was a replacement of the usual $q$ in affine Hecke algebras. 

We recall from \eqref{eq:1.21} that 
\[
\Gamma_{\mf s^\vee}^+ = 
\prod\nolimits_{\tau \in \Irr (\mb W_F)^+_\phi, \ell_\tau = 0} \langle r_\tau \rangle ,
\]
where $r_\tau$ is the nontrivial automorphism of $R_{\mf s^\vee,\tau} \cong D_{e_\tau}$.
With Proposition\autoref {prop:1.6} we deduce that 
\[
\Gamma_{\mf s}^+ = 
\prod\nolimits_{\tau \in \Irr (\mb W_F)^+_\phi, \ell_\tau = 0} \langle r_\rho \rangle ,
\]
where $\rho$ corresponds to $\tau$ via the local Langlands correspondence for $\GL_{\dim \tau}(F)$ and $r_\rho$ is the
nontrivial automorphism of $R_{\mf s,\rho} \cong D_{e_\rho}$. For each such $\tau$ we define
$J_{r_\tau}$ as in \cite[\S 4.6]{Hei3}, it is unique up a factor $\pm 1$. Then the arguments
from \cite{Hei3} remain valid in our setting (see \cite[\S 10]{SolEnd}) and they show that
\begin{equation}\label{eq:1.37}
\mc H (\mf s) \cong \mc H (\mc R_{\mf s}, \lambda, \lambda^*,q_F^{1/2}) \rtimes \Gamma_{\mf s},
\end{equation}
where $\Gamma_{\mf s}$ acts on $\mc H (\mc R_{\mf s}, \lambda, \lambda^*,q_F^{1/2})$ via
automorphisms of $\mc R_{\mf s}$. We note that the algebra on the right is canonically 
isomorphic to its own opposite:
\begin{equation}\label{eq:1.36}
\begin{array}{lcl}
\mc H (\mc R_{\mf s}, \lambda, \lambda^*,q_F^{1/2}) \rtimes \Gamma_{\mf s} & \isom &
\big( \mc H (\mc R_{\mf s}, \lambda, \lambda^*,q_F^{1/2}) \rtimes \Gamma_{\mf s} \big)^{\op} \\
f T'_w & \mapsto & T'_{w^{-1}} f \hspace{2cm} f \in \mc O(T_{\mf s}), w \in W_{\mf s}.
\end{array}
\end{equation}
Here $T'_w$ with $w \in W(R_{\mf s}) \rtimes \Gamma_{\mf s}$ denotes a product of a standard 
generator of $\mc H (\mc R_{\mf s}, \lambda, \lambda^*,q_F^{1/2})$ and an element of 
$\Gamma_{\mf s}$.

\begin{thm}\label{thm:1.7}
\enuma{
\item There exists an algebra isomorphism
\[
\mc H (\mf s) \cong \mc H (\mf s^\vee, q_F^{1/2})
\]
which extends the isomorphism $\mc O (T_{\mf s}) \cong \mc O (T_{\mf s^\vee})$ 
given by Theorem\autoref {thm:1.3}. This isomorphism is canonically determined up to:
\begin{enumerate}[(i)]
\item the action of $\mr{Out}(\mc G)$,
\item conjugation by elements of $\mc O (T_{\mf s})^\times$,
\item adjusting the image of $\Gamma_{\mf s}$ in $\mc H (\mf s)$ by a character of 
$\Gamma_{\mf s}$,
\item Let $\beta$ be a short simple root in a root system $B_{e_\rho}$, and suppose that
$R_{\mf s,\rho}$ has type $D_{e_\rho}$ or that $R_{\mf s,\rho}$ has type $B_{e_\rho}$
and $q_\beta^* = 1$. Then we may replace $s_\beta$ by $h_\beta^\vee s_\beta \in X^* (T_{\mf s})
\rtimes W(R_{\mf s})$ and $T'_{s_\beta}$ by $T'_{h_\beta^\vee s_\beta}$ in 
$\mc H (\mc R_{\mf s}, \lambda, \lambda^*,q_F^{1/2}) \rtimes \Gamma_{\mf s^\vee}^+$.
\end{enumerate}
\item Composition with \eqref{eq:1.36} yields an algebra isomorphism
\[
\mc H (\mf s)^{\op} \cong \mc H (\mf s^\vee, q_F^{1/2}) ,
\]
which is exactly as canonical as part (a).
} 
\end{thm}
\textbf{Remarks.} The condition $q_\beta^* = 1$ in (iv) is equivalent with $\lambda (\beta) = 
\lambda^* (\beta)$, and also with $\ell_{\tau'} = 0, a_{\tau'} = -1$. 
Canonical choices for (iii) and (iv) will come from Proposition\autoref {prop:3.2} and Paragraph~\ref{par:intertwining}. That will render the
isomorphisms in Theorem\autoref {thm:1.7} canonical up to $\mr{Out}(\mc G)$ and inner automorphisms.
\begin{proof} 
(a) From \eqref{eq:1.22}, Theorem\autoref {thm:1.3} and Proposition\autoref {prop:1.6} we get an algebra
isomorphism
\begin{equation}\label{eq:1.38}
\mc H (\mf s^\vee, q_F^{1/2}) \cong 
\mc H (\mc R_{\mf s}, \lambda, \lambda^*,q_F^{1/2}) \rtimes \Gamma_{\mf s} .
\end{equation}
It is canonical up to the action of $\mr{Out}(\mc G)$ on supercuspidal representations, see Theorem
\autoref{thm:1.2}. We fix a bijection $\mf s_L \to \mf s_L^\vee$ as in Theorem\autoref {thm:1.3}, then we do 
not have to worry about $\mr{Out}(\mc G)$ any more.
We compose \eqref{eq:1.38} with \eqref{eq:1.37} to obtain the isomorphism in the statement. 

Any two such isomorphisms differ by an automorphism $\psi$ of 
$\mc H (\mc R_{\mf s}, \lambda, \lambda^*,q_F^{1/2}) \rtimes \Gamma_{\mf s}$. We need to investigate
the possibilities for $\psi$. Since the isomorphism 
\[
\mc O (T_{\mf s^\vee}) \cong \mc O (T_{\mf s}) \subset \mc H (\mf s)
\]
has been fixed, $\psi$ is the identity on $\mc O (T_{\mf s})$. Any such $\psi$ extends naturally
to an automorphism $\psi_e$ of
\[
\C (T_{\mf s})^{W_{\mf s}} \underset{\mc O (T_{\mf s})^{W_{\mf s}}}{\otimes} 
\mc H (\mc R_{\mf s}, \lambda, \lambda^*,q_F^{1/2}) \rtimes \Gamma_{\mf s},
\]
an algebra which by \cite[\S 5]{Lus-Gr} is isomorphic to $\C (T_{\mf s}) \rtimes W_{\mf s}$.
As $\psi_e$ is the identity on $\C (T_{\mf s})$ and $W_{\mf s}$ acts faithfully on $T_{\mf s}$,
$\psi_e$ must send any $w \in W_{\mf s}$ to $\theta_w w$ for some $\theta_w \in \C (T_{\mf s})^\times$.

For a simple reflection $s_\alpha \in W(R_{\mf s})$ there are unique $f_1,f_2 \in \C (T_{\mf s})$
such that $T'_{s_\alpha} = f_1 s_\alpha + f_2$, see \cite[5.1.(a)]{Lus-Gr}. Then
\[
\psi (T'_{s_\alpha}) = f_1 \psi_e (s_\alpha) + f_2 = f_1 \theta_{s_\alpha} s_\alpha + f_2 ,
\]
so by the invertibility of $\psi$ we must have 
\[
\theta_{s_\alpha} \in \mc O (T_{\mf s})^\times = \C^\times \times X^* (T_{\mf s}).
\]
Write $\theta_{s_\alpha} = z \theta_x$ with $z \in \C^\times$ and $x \in X^* (T_{\mf s})$.
(We write $\theta_x$ to emphasize that we regard $x$ as an element of $\mc O (T_{\mf s})$.) Then
\[
1 = s_\alpha^2 = \psi_e (s_\alpha )^2 = (z \theta_x s_\alpha )^2 = z^2 \theta_x \theta_{s_\alpha (x)}
s_\alpha^2 = z^2 \theta_{x + s_\alpha (x)} .
\]
Hence $z = \pm 1$ and $s_\alpha (x) = -x$, which implies $x \in \Z h_\alpha^\vee$. (Here we do not use
\eqref{eq:3.19}, in the sense that we do not replace $B_{e_\tau}$ even when that is possible.) For 
every $x \in \Z h_\alpha^\vee$, $f_1 \theta_x s_\alpha + f_2$ satisfies the same quadratic equation as
$T'_{s_\alpha}$, that follows from a computation in $\C (T_{\mf s})$ which uses that 
$\theta_x s_\alpha$ is a reflection in the same direction as $s_\alpha$. On the other hand
$-f_1 \theta_x s_\alpha + f_2$ does not satisfy that quadratic relation, so $z = 1$ and
\begin{equation}\label{eq:1.39}
\psi_e (s_\alpha) = \theta_{n_\alpha h_\alpha^\vee} s_\alpha \qquad \text{for some } n_\alpha \in \Z.
\end{equation}
Let $\alpha^\sharp \in R_{\mf s^\vee}$ be the coroot associated to $h_\alpha^\vee \in R_{\mf s}$.
In the algebra $\C [\Hom (\Z R_{\mf s^\vee}, \Z) \rtimes W(R_{\mf s^\vee})]$, the element
\eqref{eq:1.39} can be rewritten as 
\[
\psi_e (s_\alpha) = \theta_y s_\alpha \theta_{-y} \quad \text{when } 
\langle y, \alpha^\sharp \rangle = n_\alpha .
\]
Guided by this formula we define $y \in \Hom_\Z (\Z R_{\mf s^\vee}, \Z)$ by
$\langle y, \alpha^\sharp \rangle = n_\alpha$ for all simple coroots $\alpha^\sharp$. 
Embed $\Hom_\Z (\Z R_{\mf s^\vee}, \Z)$ in $\Q R_{\mf s}$ and form the lattice
\[
X_e := X^* (T_{\mf s}) + \Hom_\Z (\Z R_{\mf s^\vee}, \Z) \subset X^* (T_{\mf s}) \otimes_\Z \Q.
\]
Then $\psi_e$ extends to the automorphism of $\C [X_e \rtimes W(R_{\mf s})]$, given by 
conjugation with $\theta_y$. Hence $\psi$ is also conjugation with $\theta_y$, at least on
$\mc H (\mc R_{\mf s}, \lambda, \lambda^*,q_F^{1/2})$. 
For $y \in X^* (T_{\mf s})$ that is simply an inner automorphism, which accounts for (ii).

There are only few other possible $y$. For each $\tau$ with $e_\tau > 0$, we have a direct summand 
\[
(\Z^{e_\tau}, R_{\mf s,\tau}, \Z^{e_\tau}, R_{\mf s^\vee,\tau}) \quad \text{of} \quad \mc R_{\mf s},
\]
where $R_{\mf s,\tau}$ has type $A_{e_\tau - 1}, B_{e_\tau}$ or $D_{e_\tau}$.
For type $A_{e_\tau - 1}$, $\Z^{e_\tau}$ surjects onto\\ 
$\Hom_\Z (\Z R_{\mf s^\vee,\tau}, \Z)$. Otherwise $\Hom_\Z (\Z R_{\mf s^\vee,\tau}, \Z)$ is 
spanned by $\Z^{e_\tau}$ and\\ $y = (1,1,\ldots,1)/2$. Conjugation by $\theta_y$
on $\C[\Z^{e_\tau} \rtimes W(B_{e_\tau})]$ sends $s_\beta$ to $h_\beta^\vee s_\beta$ and fixes
the other simple reflections. When $R_{\mf s,\tau} \cong D_{e_\rho}$, this gives an automorphism
of $\mc H ( D_{e_\rho}, q_F^{1/2} ) \rtimes \langle s_\beta \rangle$ and of $\mc H (\mf s)^{\op}$.

However, when $R_{\mf s,\tau}$ has type $B_{e_\rho}$ conjugation by $\theta_y$ only extends to an
automorphism of $\mc H (\mc R_{\mf s}, \lambda, \lambda^*,q_F^{1/2})$ if $q_\beta^* = 1$, 
because the $q$-parameters $q_\beta q_\beta^*$ of $s_\beta$ and $q_\beta (q_\beta^* )^{-1}$ of 
$h_\beta^\vee s_\beta$
need to be equal for such an automorphism. That gives the choices for $\psi$ described in (iv).
Notice that this excludes the cases $C_{e_\tau}$ that could arise via \eqref{eq:3.19}.

It remains to investigate automorphisms $\psi$ of  
$\mc H (\mc R_{\mf s}, \lambda, \lambda^*,q_F^{1/2}) \rtimes \Gamma_{\mf s}$ that restrict to
the identity on $\mc H (\mc R_{\mf s}, \lambda, \lambda^*,q_F^{1/2})$. As above we deduce that
for each $\gamma \in \Gamma_{\mf s}$ there exist $z \in \{\pm 1\}$ and $x \in X^* (T_{\mf s})$
such that $\psi (\gamma) = z \theta_x \gamma$. Just like conjugation by $\gamma$, conjugation
by $\psi (\gamma)$ is a product of diagram automorphisms of $D_{e_\tau}$ on
\[
\psi \big( \mc H (\mc R_{\mf s}, \lambda, \lambda^*,q_F^{1/2}) \big) = 
\mc H (\mc R_{\mf s}, \lambda, \lambda^*,q_F^{1/2}) .
\]
Hence $z \theta_x$ must lie in the centre of $\mc H (\mc R_{\mf s}, \lambda, \lambda^*,q_F^{1/2})$,
which means that $\langle x ,\alpha^\sharp \rangle = 0$ for every coroot $\alpha^\sharp$.
Looking at the rank of $R_{\mf s,\tau}$, we see that $x$ lives only in the $\Z^{e_\tau}$ for
which $R_{\mf s,\tau} \cong A_{e_\tau -1}$. The part of $x$ in the associated direct summand of
$X^* (T_{\mf s})$ is a multiple of $(1,1,\ldots,1)$. In particular $\theta_x$ commutes with $\gamma$.
As $\gamma$ has finite order in the finite group $\Gamma_{\mf s}$:
\[
1 = \gamma^{\mr{ord} \gamma} = \psi (\gamma)^{\mr{ord} \gamma} = 
(z \theta_x \gamma)^{\mr{ord} \gamma} = z^{\mr{ord} \gamma} \theta_x^{\mr{ord} \gamma}
\gamma^{\mr{ord} \gamma} = z^{\mr{ord} \gamma} \theta_{\mr{ord} (\gamma) x} .
\]
This implies that ord$(\gamma) x = 0$ and $x = 0, \psi (\gamma) = \pm \gamma$. We deduce that
there exists a character $\epsilon : \Gamma_{\mf s} \to \{\pm 1\}$ such that 
$\psi (\gamma) = \epsilon (\gamma) \gamma$.\\
(b) This follows from part (a) and \eqref{eq:1.37}.
\end{proof}

\subsection{Versions for \texorpdfstring{$G^+$}{G^+}} \
\label{par:G+}

There also exists a version of Theorem\autoref {thm:1.7} for $G^+$. Let 
$\mc L^+ = \rZ_{\mc G^+}(\rZ(\mc L)^\circ)$ be the $F$-Levi subgroup of $\mc G^+$ with identity 
component $\mc L$. It has the same shape:
\begin{equation}\label{eq:3.58}
L^+ = G^+_{n_-} \times \GL_{n_1}(F) \times \cdots \times \GL_{n_k}(F) .
\end{equation}
Usually $[L^+ : L] = [G^+ : G]$, but there are exceptions to that. Recall from \eqref{eq:1.60} 
that $G_{n_-} \subset GL (V_-)$ for some linear subspace $V_- \subset V$. If (and only if) 
$V_- = \{0\}$, we have $G_{n_-}^+ = G_{n_-} = \{e\}$ and $L^+ = L$. This can happen for the split
groups $O(V)$ and $\GSpin (V)$ with $\dim V$ even, even though $G^+ \neq G$ for those groups.

The whole theory behind $\mc H (\mf s^\vee, z)$ \cite{AMS1,AMS2,AMS3} was written for possibly
disconnected complex reductive groups, so it applies to $G^+$. The set of cuspidal Bernstein
components in $\Phi_\enh (L^+)$ is
\[
\Phi_\cusp (L^+) / (\rZ(L^+)^{\mb I_F,\circ})_{\mb W_F} = \Phi_\cusp (L^+) / \rZ(L)^\circ .
\]
An element in there is the same as an element $(\phi,\epsilon) \in \Phi_\cusp (L) / \rZ(L)^\circ$
together with an extension of $\epsilon \in \Irr (\mc S_\phi)$ to $\epsilon^+ \in \Irr (\mc S_\phi^+)$.
Let $\mf s^{+\vee}$ denote the Bernstein component determined by $(\phi,\epsilon^+)$, and similarly
without the $+$. We note that there is a canonical bijection
\[
\Phi_\enh (L)^{\mf s^\vee} \to \Phi_\enh (L^+ )^{\mf s^{+\vee}} : 
(z \phi, \epsilon) \mapsto (z \phi, \epsilon^+) .
\]
The same arguments as in Paragraph~\ref{par:HAL} shows that 
\begin{equation}\label{eq:3.1}
\mc H (\mf s^{+\vee},z) = 
\mc H (\mc R_{\mf s^\vee}, \lambda, \lambda^*, z) \rtimes \Gamma_{\mf s^\vee}^+ .
\end{equation}
The arguments in Paragraph~\ref{par:ess} and for Theorem\autoref {thm:1.5} lead to a canonical injection
from the set of Bernstein components of $\Irr (G^+)$ to the set of Bernstein components in 
$\Phi_\enh (G^+)$, say $\mf s^+ \mapsto \mf s^{+\vee}$. It relates to Theorem\autoref {thm:1.5} by 
$\Res_G^{G^+}$, as in the proof of Theorem\autoref {thm:1.2}. On the level of representations and 
enhanced $L$-parameters of $L^+$, by Theorems\autoref {thm:1.2} and\autoref {thm:1.3} each instance 
$\mf s^+ \mapsto \mf s^{+\vee}$ comes a bijection
\begin{equation}\label{eq:3.2}
\Irr (L^+ )^{\mf s_L^+} \cong T_{\mf s} \longrightarrow 
T_{\mf s^\vee} \cong \Phi_\enh (L^+ )^{\mf s_L^{+\vee}} .
\end{equation}
As justified by \eqref{eq:3.2}, we will sometimes write $T_{\mf s^+}$ for $T_{\mf s}$, or
$T_{\mf s^{+\vee}}$ for $T_{\mf s^\vee}$.

The theory used to construct and analyse $\mc H (\mf s)$ is not known for arbitrary disconnected
reductive groups. For $O(V)$ and $\GPin (V)$ (the only disconnected instances of $G^+$) we can work 
it out by hand though. First we need a good progenerator $\Pi_{\mf s_L^+}$ for 
$\Rep (L^+)^{\mf s_L^+}$ with $\mf s_L = [L,\sigma]_L$. We start from $\Pi_{\mf s_L } = 
\mr{ind}_{L^1}^L (\sigma_1 )$, where $L^1$ is the subgroup of $L$ generated by all 
compact subgroups and $\sigma_1$ is an irreducible constituent of $\Res^L_{L^1} \sigma$.
We distinguish two cases.

Suppose first that $\mr{Out} (\mc G_{n_-})$ does not stabilize $\mf s_L$, or that $L^+ = L$. Then
$\mr{ind}_L^{L^+} (\sigma' )$ is irreducible for all $\sigma' \in \Irr (L)^{\mf s_L}$, and
\begin{equation}\label{eq:3.51}
\mr{ind}_L^{L^+} (\Pi_{\mf s_L}) = \mr{ind}_{L^1}^{L^+} (\sigma_1 ) =: \Pi_{\mf s_L^+}
\end{equation}
is a progenerator of Rep$(L^+)^{\mf s_L^+}$ for the same reasons as for $\Pi_{\mf s_L}$.
If $L^+ = L$, then clearly $\End_{L^+}(\Pi_{\mf s_L^+}) = \End_L (\Pi_{\mf s_L})$. 
Otherwise, since $L$ is normal in $L^+$,
\[
\mr{Res}^{L^+}_L \mr{ind}_L^{L^+} (\Pi_{\mf s_L}) = \Pi_{\mf s_L} \oplus l \cdot \Pi_{\mf s_L} =
\Pi_{\mf s_L} \oplus \Pi_{\mf s'_L}
\]
where $l \in L^+ \setminus L$ and $\mf s'_L = l \cdot \mf s_L$. Further, by Frobenius reciprocity
\begin{equation}\label{eq:3.3}
\End_{L^+} \big( \mr{ind}_L^{L^+} \Pi_{\mf s_L} \big) \cong \Hom_L \big( \Pi_{\mf s_L},
\mr{ind}_L^{L^+} \Pi_{\mf s_L} \big) \cong 
\Hom_L \big( \Pi_{\mf s_L}, \Pi_{\mf s_L} \oplus \Pi_{\mf s'_L} \big) .
\end{equation}
By the Bernstein decomposition of $\Rep (L)$ this equals $\End_L (\Pi_{\mf s_L})$, which by 
\eqref{eq:3.18} is naturally isomorphic with $\mc O (T_{\mf s})$.

Suppose now that $\mr{Out} (\mc G_{n_-})$ stabilizes $\mf s_L$ and $L^+ \neq L$. 
Since $X_\nr (G_{n_-}) = \{1\}$, $\mr{Out} (\mc G_{n_-})$ stabilizes every $\sigma' \in 
\Irr (L)^{\mf s_L}$. Clifford theory tells us that $\sigma$ extends in two ways to a 
representation of $L^+$, say $\sigma^+$ and $\sigma^-$. For an
unramified character $\chi \in X_\nr (L^+) \cong X_\nr (L)$ we put
\[
(\sigma \otimes \chi )^+ = \sigma^+ \otimes \chi \quad \text{and} \quad
(\sigma \otimes \chi )^- = \sigma^- \otimes \chi .
\]
This yields two Bernstein components $\Irr (L^+ )^{\mf s_L^+} = X_\nr (L^+) \sigma^+$ and
$\Irr (L^+ )^{\mf s_L^-} = X_\nr (L^+) \sigma^-$, both naturally in bijection with 
$\Irr (L)^{\mf s_L}$. We note that $\mf s_L^+$ and $\mf s_L^-$ are in different 
$\rN_{G^+}(L^+)$-orbits, because they are inequivalent on $G_{n_-}^+$ and $\rN_{G^+}(L^+) / G_{n_-}^+$ 
only adjusts $\Irr_\cusp (L^+)$ on the type GL factors of $L^+$. The Bernstein decomposition 
for $L^+$ enables us to write
\begin{equation}\label{eq:3.52}
\ind_L^{L^+} (\Pi_{\mf s_L}) = \Pi_{\mf s_L^+} \oplus \Pi_{\mf s_L^-} \quad
\text{with } \Pi_{\mf s_L^\pm} \in \Rep (L^+ )^{\mf s_L^\pm}.
\end{equation}
Then $\Pi_{\mf s_L^+}$ is a progenerator of $\Rep (L^+)^{\mf s_L^+}$ and its restriction to 
$L$ is just $\mr{ind}_{L^1}^L (\sigma) = \Pi_{\mf s_L}$. All the elements of 
$\mc O (T_{\mf s})$ determine $L^+$-endomorphisms of $\mr{ind}_{L^{+1}}^{L^+} (\sigma^+)$, so
\[
\mc O (T_{\mf s}) = \End_L ( \Pi_{\mf s_L}) = \End_{L^+} (\Pi_{\mf s_L^+}) .
\]
In both above cases we constructed a progenerator $\Pi_{\mf s_L^+}$ of 
$\Rep (L^+)^{\mf s_L^+}$, with $L^+$-endomorphism algebra $\mc O (T_{\mf s})$. We define
\[
\Pi_{\mf s^+} = I_{P^+}^{G^+} (\Pi_{\mf s_L^+}) ,
\]
where $P^+$ is the semidirect product of $L^+$ and the unipotent radical of $P$. 

\begin{prop}\label{prop:3.1}
The representation $\Pi_{\mf s^+}$ is a canonical progenerator of $\Rep (G^+ )^{\mf s^+}$. 

Induction from $G$ to $G^+$ gives an injective algebra homomorphism $\End_G (\Pi_{\mf s}) \to 
\End_{G^+}(\Pi^{\mf s^+})$, which is bijective when $\mr{Out}(\mc G_{n_-}) \mf s_L = \mf s_L$.
\end{prop}
\begin{proof}
The $L^+$-representation $\Pi_{\mf s_L^+}$ is canonical because in both cases 
\eqref{eq:3.51} and \eqref{eq:3.52} it arises naturally from the $L$-representation
$\Pi_{\mf s_L}$, which we already knew is canonical \eqref{eq:3.50}.

Suppose that $\mr{Out}(\mc G_{n_-}) \mf s_L \neq \mf s_L$, or that $L^+ = L$. Then 
\begin{equation}\label{eq:3.4}
\Pi_{\mf s^+} = I_{P^+}^{G^+} \big( \mr{ind}_L^{L^+} (\Pi_{\mf s_L}) \big) = 
\mr{ind}_G^{G^+} \big( I_P^G (\Pi_{\mf s_L}) \big) = \mr{ind}_G^{G_+} (\Pi_{\mf s}) .
\end{equation}
Now $\mr{ind}_G^{G^+}$ yields an algebra homomorphism
\begin{equation}\label{eq:3.55}
\End_G (\Pi_{\mf s}) \to \End_{G^+} (\mr{ind}_G^{G^+} \Pi_{\mf s} ) = \End_{G^+} (\Pi_{\mf s^+}) ,
\end{equation}
which is injective because $\Pi_{\mf s} \subset \Res^{G^+}_G \Pi_{\mf s^+}$. As $G$ is open 
in $G^+$, $\mr{ind}_G^{G^+}$ preserves projectivity. Moreover $G$ has finite index in $G^+$, 
so \eqref{eq:3.4} shows that $\Pi_{\mf s^+}$ is finitely generated and projective. For any 
nonzero $\tau \in \Rep (G)^{\mf s^+}$, the part of $\Res^{G^+}_G \tau$ in $\Rep (G)^{\mf s}$ 
generates $\tau$ so is nonzero. Hence 
\[
\Hom_{G^+} (\Pi_{\mf s^+},\tau) = \Hom_G (\Pi_{\mf s}, \tau) \neq 0 ,
\] 
which shows that $\Pi_{\mf s^+}$ generates $\Rep (G^+)^{\mf s^+}$.

Next we suppose that $\mr{Out}(\mc G_{n_-}) \mf s_L = \mf s_L$ and $L^+ \neq L$. Then
\begin{align}\label{eq:3.5}
& \Res^{G^+}_G \Pi_{\mf s^+} = I_P^G (\Pi_{\mf s_L^+} \big|_L ) = 
I_P^G (\Pi_{\mf s_L}) = \Pi_{\mf s},\\
\label{eq:3.6} & \mr{ind}_G^{G^+} (\Pi_{\mf s}) = 
I_{P^+}^{G^+} \big( \mr{ind}_L^{L^+} (\Pi_{\mf s_L}) \big) = 
I_{P^+}^{G^+} \big( \Pi_{\mf s_L^+} \oplus \mr{ind}_{L^{+1}}^{L^+} (\sigma^-) \big) \\
\nonumber & \hspace{17mm} = I_{P^+}^{G^+} ( \Pi_{\mf s_L^+})  \oplus I_{P^+}^{G^+} \big( 
 \mr{ind}_{L^{+1}}^{L^+} (\sigma^-) \big) = \Pi_{\mf s^+} \oplus I_{P^+}^{G^+} (\Pi_{\mf s_L^-}) .
\end{align}
Since $\mf s_L^+$ and $\mf s_L^-$ are in different $\rN_{G^+}(L^+)$-orbits, 
$\mf s^+ \neq \mf s^- = [L^+ ,\sigma^- ]_{G^+}$. By the Bernstein decomposition $\Rep (G^+)^{\mf s^+}$
and $\Rep (G^+ )_{\mf s^-}$ are orthogonal subcategories of $\Rep (G^+)$, so 
\begin{equation}\label{eq:3.7}
\End_{G^+} (\mr{ind}_G^{G^+} (\Pi_{\mf s}) = \End_{G^+} \big(  \Pi_{\mf s^+} \oplus \Pi_{\mf s^-} \big)
= \End_{G^+} (\Pi_{\mf s^+}) \oplus \End_{G^+} (\Pi_{\mf s^-}).
\end{equation}
From \eqref{eq:3.7}, $\mr{ind}_G^{G^+}$ and \eqref{eq:3.5} we obtain algebra homomorphisms
\begin{equation}\label{eq:3.8}
\End_G (\Pi_{\mf s}) \to \End_{G^+}(\Pi_{\mf s^+}) \to \End_G (\Pi_{\mf s^+}) = \End_G (\Pi_{\mf s}) .
\end{equation}
The composition of these homomorphisms is the identity and $\End_{G^+}(\Pi_{\mf s^+})$ is naturally
a subalgebra of $\End_G (\Pi_{\mf s^+})$, from which we conclude that \eqref{eq:3.8} consists of
isomorphisms.

By the same argument as in the first part, $\mr{ind}_G^{G^+} (\Pi_{\mf s})$ is finitely generated and
projective. In view of \eqref{eq:3.6}, so is its direct summand $\Pi_{\mf s^+}$. Let $\tau \in 
\Rep (G^+)_{\mf s^+}$ be nonzero. By \eqref{eq:3.6}
\[
\Hom_{G^+}(\Pi_{\mf s^+},\tau) = \Hom_{G^+}(\Pi_{\mf s^+} \oplus \Pi_{\mf s^-},\tau) =
\Hom_{G^+} \big( \mr{ind}_G^{G^+} \Pi_{\mf s}, \tau) = \Hom_G (\Pi_{\mf s},\tau) .
\]
As we already saw above, the right hand side is nonzero. Therefore $\Pi_{\mf s^+}$ is indeed
a progenerator of $\Rep (G^+)_{\mf s^+}$.
\end{proof}

We define $\mc H (\mf s^+) = \End_{G^+} (\Pi_{\mf s^+})$, then Proposition\autoref {prop:3.1} shows that
there is an equivalences of categories
\begin{equation}\label{eq:3.9}
\begin{array}{ccc}
\Mod (\mc H (\mf s^+ )^{\op} ) & \isom & \Rep (G)^{\mf s^+} \\
V & \mapsto & V \otimes_{\mc H (\mf s^+)} \Pi_{\mf s^+} 
\end{array} .
\end{equation}

\begin{prop}\label{prop:3.2}
There exists an algebra isomorphism
\[
\mc H (\mf s^+ ) \cong \mc H (\mf s^{+\vee}, q_F^{1/2}) =
\mc H (\mc R_{\mf s^\vee}, \lambda, \lambda^*, q_F^{1/2}) \rtimes \Gamma^+_{\mf s^\vee} \cong
\mc H (\mf s^{+\vee}, q_F^{1/2})^{\op} .
\]
It extends the isomorphism $\mc O (T_{\mf s}) \cong \mc O (T_{\mf s^\vee})$ induced by \eqref{eq:3.2}
and is canonical up to the operations (ii),(iii),(iv) in Theorem\autoref {thm:1.7}.
\end{prop}
\begin{proof}
With the progenerators $\Pi_{\mf s^+}$ at hand, the paper \cite{SolEnd} also applies to $G^+$.
Therefore all the arguments in Section~\ref{sec:Hecke} remain valid. The only difference with the
proof of Theorem\autoref {thm:1.7} is that we do not have to replace $W_{\mf s}^+$ by $W_{\mf s}$ any more. 
\end{proof}

From the above proof we see that in the description of Proposition\autoref {prop:3.2} the map 
$\mc H (\mf s) \to \mc H (\mf s^+)$ from Proposition\autoref {prop:3.1} becomes just the inclusion
\begin{equation}\label{eq:3.10}
\mc H (\mc R_{\mf s^\vee}, \lambda, \lambda^*, q_F^{1/2}) \rtimes \Gamma_{\mf s^\vee} \longrightarrow
\mc H (\mc R_{\mf s^\vee}, \lambda, \lambda^*, q_F^{1/2}) \rtimes \Gamma^+_{\mf s^\vee} .
\end{equation}
\textbf{Remark.} In Proposition\autoref {prop:4.4} we will fix choices for (iii) and (iv) from Proposition 
\autoref{prop:3.2}, depending only on a Whittaker datum for the quasi-split inner form of $G$. That will 
make the isomorphisms in Proposition\autoref {prop:3.2} canonical up to inner automorphisms. 
Via \eqref{eq:3.10}, that also determines choices for (iii) and (iv) in Theorem\autoref {thm:1.7}.

\begin{lem}\label{lem:3.3}
\enuma{
\item Suppose that $\mr{Out}(\mc G_{n_-}) \mf s_L = \mf s_L$. Then the restriction map \\
$\Rep (G^+)^{\mf s^+} \to \Rep (G)^{\mf s}$ is an equivalence of categories.
\item Suppose that $\mr{Out}(\mc G_{n_-}) \mf s_L \neq \mf s_L$ and that all the direct
factors $\GL_m (F)$ of $L$ have $m$ even. Then $\mr{ind}_G^{G^+} \colon\Rep (G)^{\mf s} 
\to \Rep (G^+)^{\mf s^+}$ is an equivalence of categories. 
\item In the remaining cases $\Rep (G)^{\mf s}$ and $\Rep (G^+)^{\mf s^+}$ are not naturally
equivalent. 
}
\end{lem}
\begin{proof}
(a) Via \eqref{eq:3.5} and \eqref{eq:3.10}, the restriction is induced by the algebra homomorphism
$\mc H (\mf s) \to \mc H (\mf s^+)$. In Proposition\autoref {prop:3.1} we saw that it is an isomorphism.\\
(b) The second condition implies that $\rN_{G^+}(L^+) / L^+ \cong N_G (L) / L$. Hence
$\Gamma_{\mf s}^+ = \Gamma_{\mf s}$, which together with \eqref{eq:3.10} means that the map
$\mc H (\mf s) \to \mc H (\mf s^+)$ from Proposition\autoref {prop:3.1} is an algebra isomorphism.
That yields equivalences of categories
\begin{equation}\label{eq:3.11}
\begin{array}{ccccccc}
\Rep (G)^{\mf s} & \leftrightarrow & \Mod (\mc H (\mf s)^{\op}) & \leftrightarrow &
\Mod (\mc H (\mf s^+ )^{\op}) & \leftrightarrow & \Rep (G^+ )^{\mf s^+} \\
V \otimes_{\mc H (\mf s)} \Pi_{\mf s} & \text{\reflectbox{$\mapsto$}} & V & \leftrightarrow & V^+ &
\mapsto & V^+ \otimes_{\mc H (\mf s^+)} \Pi_{\mf s^+} 
\end{array}\!.\!
\end{equation}
By the first condition, \eqref{eq:3.4} holds. Hence $V \otimes_{\mc H (\mf s)} \Pi_{\mf s}$ is mapped
by \eqref{eq:3.11} to
\[
V^+ \otimes_{\mc H (\mf s^+)} \Pi_{\mf s^+} = V \otimes_{\mc H (\mf s)} \mr{ind}_G^{G^+} (\Pi_{\mf s})
= \mr{ind}_G^{G^+} \big( V \otimes_{\mc H (\mf s)} \Pi_{\mf s} \big) . 
\]
(c) The assumption says  that $L$ has a direct factor $\GL_m (F)$ with $m$ odd, and that 
$\mr{Out}(\mc G_{n_-}) \mf s_L = \{ \mf s_L, \mf s'_L \}$ with $\mf s'_L = l^- \cdot \mf s_L \neq \mf s_L$ 
for any $l^- \in G_{n_-}^+ \setminus G_{n_-}$. Consider an element $s_\alpha \in N_{G^+}(L^+)$ which acts
in this factor $\GL_m (F)$ by $g \mapsto \hat{J} g^{-T} \hat{J}^{-1}$ and on $L$ as in \eqref{eq:1.23}.
Then $\det (s_\alpha) = -1$ because $m$ is odd, so $s_\alpha L \notin W_{\mf s}$. On the other hand
$s_\alpha l^-$ stabilizes $\mf s_L^+$, so $s_\alpha l^- L^+ = s_\alpha L^+ \in W_{\mf s}^+$. Thus
$W_{\mf s} \neq W_{\mf s}^+$, which by \eqref{eq:3.10} means that the inclusion $\mc H (\mf s) \to
\mc H (\mf s^+)$ is not an isomorphism.
\end{proof}

\subsection{Langlands parameters via Hecke algebras} \
\label{par:LLC}

Let $\mf s = [L,\sigma]_G$ be an inertial equivalence class for $G$.
Recall the canonical progenerator $\Pi_{\mf s}$ and the equivalence of categories
\begin{equation}\label{eq:3.12}
\Hom_G (\Pi_{\mf s},?) :
\Rep (G)^{\mf s} \isom \Mod (\mc H (\mf s)^{\op}) = \Mod (\End_G (\Pi_{\mf s})^{\op} )
\end{equation}
from \eqref{eq:3.54}. Let us fix an isomorphism as in Theorem\autoref {thm:1.7}.a. 
That and \eqref{eq:1.36} induce equivalences of categories
\begin{equation}\label{eq:3.13}
\Mod (\mc H (\mf s)^{\op}) \cong \Mod \big( \mc H (\mf s^\vee, q_F^{1/2})^{\op} \big) 
\cong \Mod \big( \mc H (\mf s^\vee, q_F^{1/2}) \big).
\end{equation}
From \eqref{eq:3.12} and \eqref{eq:3.13} (or \eqref{eq:3.9} and Proposition
\autoref{prop:3.2} for $G^+$) we obtain equivalences of categories 
\begin{equation}\label{eq:3.16}
\begin{array}{lll}
\Rep (G)^{\mf s} & \cong & \Mod \big( \mc H (\mf s^\vee, q_F^{1/2}) \big) , \\
\Rep (G^+)^{\mf s^+} & \cong & \Mod \big( \mc H (\mf s^{+\vee}, q_F^{1/2}) \big) .
\end{array}
\end{equation}
It was shown in \cite[Theorem 3.18]{AMS3} that there is a canonical bijection
\begin{equation}\label{eq:3.14}
\Irr \big( \mc H (\mf s^\vee, q_F^{1/2}) \big) \longleftrightarrow \Phi_\enh (G)^{\mf s^\vee} ,
\end{equation}
and similarly for $G^+$. This proceeds via reduction to the graded Hecke algebras mentioned 
around \eqref{eq:1.50}, which are then studied in terms of varieties of Langlands parameters, 
perverse sheaves and equivariant homology \cite{AMS2}. Disconnected complex reductive groups
and the associated Hecke algebras are an integral part of \cite{AMS2,AMS3}, and therefore 
\eqref{eq:3.14} works in the same way for $G^+$ as for $G$. Following \cite{AMS3}, we denote 
the image of $(\phi,\epsilon) \in \Phi_\enh (G)^{\mf s^\vee}$ under \eqref{eq:3.14} by 
\[
\bar{M} (\phi, \epsilon, q_F^{1/2}) \in \Irr \big( \mc H (\mf s^\vee, q_F^{1/2}) \big).
\]
\textbf{Remark.} In \eqref{eq:3.14} $\epsilon$ is a representation of $\mc S_\phi$, a group 
in which all elements have order at most two. Therefore $\epsilon$ may identified with its 
contragredient $\epsilon^\vee$. By construction \cite{AMS2,AMS3} 
$\bar{M} (\phi, \epsilon, q_F^{1/2})$ is the unique irreducible quotient of a standard 
$\mc H (\mf s^\vee, q_F^{1/2})$-module
\begin{equation}\label{eq:3.56}
\bar E ( \phi,\epsilon,q_F^{1/2}) = 
\Hom_{\mc S_\phi} \big( \epsilon, \bar E (\phi,q_F^{1/2}) \big).
\end{equation}
Reinterpreting $\epsilon$ as $\epsilon^\vee$, that becomes the irreducible quotient of
\[
\Hom_{\mc S_\phi} \big( \epsilon^\vee, \bar E (\phi,q_F^{1/2}) \big) =
\big( \epsilon \otimes \bar E (\phi,q_F^{1/2}) \big)^{\mc S_\phi} .
\]
This might be the most natural setup for comparison with endoscopic methods, that is 
indicated for instance by \cite[\S 2]{MoRe} and \cite[Definition 2.7.6]{Kal2}.

Let $\bar{M} (\phi, \epsilon, q_F^{1/2})^{\op}$ be $\bar{M} (\phi, \epsilon^\vee, q_F^{1/2})$ 
considered as irreducible right $\mc H (\mf s^\vee, q_F^{1/2})$-module via \eqref{eq:1.36}. 
Via Theorem\autoref {thm:1.7}.a it corresponds to an irreducible right $\mc H (\mf s)$-module, 
to which we can apply \eqref{eq:3.12}.

\begin{thm}\label{thm:3.4}
The maps \eqref{eq:3.12}, \eqref{eq:3.13} and \eqref{eq:3.14} induce a bijection
\[
\begin{array}{ccc}
\Irr (G)^{\mf s} & \longleftrightarrow & \Phi_\enh (G)^{\mf s^\vee} \\
\pi (\phi,\epsilon) & \text{\reflectbox{$\mapsto$}} & (\phi,\epsilon)
\end{array} .
\] 
It satisfies the following properties:
\enuma{
\item The cuspidal support maps form a commutative diagram
\[
\begin{array}{ccc}
\Irr (G)^{\mf s} & \longleftrightarrow & \Phi_\enh (G)^{\mf s^\vee} \\
\downarrow \Sc & & \downarrow \Sc \\
\Irr ( L)^{\mf s_L} / W_{\mf s} & \longleftrightarrow & \Phi_\cusp (L)^{\mf s^\vee} / W^{\mf s^\vee} 
\end{array} .
\]
In particular $(\phi,\epsilon)$ is cuspidal if and only if $\pi (\phi,\epsilon)$ is supercuspidal.
\item $\pi (\phi,\epsilon)$ is essentially square-integrable if and only if $\phi$ is discrete.
\item $\pi (\phi,\epsilon)$ is tempered if and only if $\phi$ is bounded.
\item For any $\chi \in X_\nr (G)$, corresponding to $\hat \chi \in 
( \rZ(G^\vee)^{\mb I_F,\circ})_{\mb W_F}$, there is a canonical isomorphism 
$\pi (\hat \chi \phi, \epsilon) = \chi \otimes \pi (\phi,\epsilon)$.
\item The $\rZ (G)_s$-character of $\pi (\phi,\epsilon)$ equals the character of 
$\rZ (G)_s$ determined by the image of $\phi$ in $\Phi (\rZ(G)_s )$.
\item Suppose that Theorem\autoref {thm:1.7} can be made canonical, up to $\mr{Out}(\mc G)$ and
inner automorphisms of the involved algebras. Then the above bijection is canonical
up to $\mr{Out}(\mc G)$.
}
All the above statements also hold with $G^+$ instead of $G$. Then part (f) relies on a canonical
version of Proposition\autoref {prop:3.2}, and we can omit $\mr{Out}(\mc G)$.
\end{thm}
\textbf{Remark.} Parts (b)--(e) were already predicted in \cite[\S 10]{Bor}. In fact Borel formulated more 
general versions of (d) and (e), which in principle can also be checked in our setup. We refrain from 
taking that up here, because it will boil down to properties of endoscopy which fall outside the scope
of this paper. The canonicity requirements in part (f) will be established in Proposition\autoref {prop:4.4}.
\begin{proof}
(a) The central character of $\bar{M} (\phi, \epsilon, q_F^{1/2})$ is described in 
\cite[Theorem 3.18.a]{AMS3}. It lies in $T_{\mf s^\vee} / W_{\mf s^\vee}$ and by construction equals
$W_{\mf s^\vee} \Sc (\phi,\epsilon)$. Similarly the central character of $\Hom_G (\Pi_{\mf s}, 
\pi (\phi,\epsilon)) \in \Irr (\mc H (\mf s)^{\op})$ lies in $T_{\mf s} / W_{\mf s}$ and by
\cite[Condition 4.1 and Lemma 6.1]{SolComp} it equals $W_{\mf s} \Sc (\pi (\phi,\epsilon))$.\\
(b) By \cite[Theorem 4.9.a]{SolComp} the map \eqref{eq:3.12} respects temperedness. The equivalence
\eqref{eq:3.13} does so as well, because by Proposition\autoref {prop:1.6} the isomorphism in Theorem
\autoref{thm:1.7}.b preserves the notion of positive roots (which determines the conditions for temperedness,
see e.g. \cite[p. 215]{SolComp}). By \cite[Theorem 3.18.c]{AMS3}, under the map \eqref{eq:3.14} 
temperedness of irreducible representations corresponds to boundedness of (enhanced) $L$-parameters.\\
(c) This is similar to part (b), now we use \cite[Theorem 4.9 and Proposition 4.10]{SolComp}, 
Proposition\autoref {prop:1.6} and \cite[Theorem 3.18.d]{AMS3}.\\
(d) This follows from \cite[Lemma 4.3.c]{SolComp} and \cite[Theorem 3.18.e]{AMS3}.\\
(e) First we reduce to the cuspidal case. Clearly $\pi (\phi,\epsilon)$ and $\Sc (\pi (\phi,\epsilon))$
have the same $\rZ(G)$-character. 
Recall that ${Z (G)_s}^\vee = G^\vee / {G^\vee}_\der \cong \C^\times$. The quotient map
$G^\vee \to {Z (G)_s}^\vee$ is the similitude character $\mu_G^\vee$, so the image of $\phi$
in $\Phi (Z (G)_s)$ is $\mu_G^\vee \circ \phi$. The cuspidal support map for enhanced
$L$-parameters only changes things in ${G^\vee}_\der$ (and modifies the enhancements), so
$\mu_G^\vee \circ \phi = \mu_G^\vee \circ \phi_c$ where $\Sc (\phi,\epsilon) = (\phi_c,\epsilon_c)$.
In view of part (a), $(\phi_c,\epsilon_c)$ is the enhanced $L$-parameter of 
$\Sc (\pi (\phi,\epsilon)) =: \pi_c$. 

The $GL$-factors of $L^\vee$ lie in ${G^\vee}_\der$, so they are contained in the kernel of 
$\mu_G^\vee$. Hence $\mu_G^\vee \circ \phi$ depends only on the component of $\phi_c$ in 
$G_{n_-}^\vee$, let us call the latter $\phi_-$. On the other hand, $Z (G)_s$ is contained in 
the factor $G_{n_-}$ of $L$, so the $Z (G)_s$-character of $\pi_c$ depends only on the component 
of $\pi_c$ in $G_{n_-}$, say $\pi_- \in \Irr_\cusp (G_{n_-})$. 

It remains to compare the $\rZ(G)_s$-character $\nu_{\pi_-}$ of $\pi_-$ with $\mu_G^\vee \circ
\phi_-$. Those agree by \eqref{eq:1.46}.\\
(f) This holds because it can be checked at every step in the construction. The progenerator
$\Pi_{\mf s}$ and the bijections \eqref{eq:3.12}, \eqref{eq:3.14} are always canonical, 
and now by assumption \eqref{eq:3.13} is canonical up to $\mr{Out}(\mc G)$.

The proof for $G^+$ is basically the same. To get the bijection we use \eqref{eq:3.9} and Proposition
\autoref{prop:3.2} instead of \eqref{eq:3.12} and \eqref{eq:3.13}. Although in \cite{SolComp} the group 
$\mc G$ is connected, the parts that we use work just as well for $\mc G^+$. For parts (d) and (e) 
it is helpful to note that 
\[
X_\nr (G) \cong X_\nr (G^+) ,\quad \rZ(G^{+\vee}) = \rZ(G^\vee) \quad \text{and} \quad 
\rZ(\mc G) = \rZ(\mc G^+) . \qedhere
\]
\end{proof}

Recall from Theorem\autoref {thm:1.5} that the map $\mf s \mapsto \mf s^\vee$ between sets of Bernstein 
components for $G$ and $G^\vee$ is injective. From that, \eqref{eq:3.16} and Theorem\autoref {thm:3.4} 
we conclude:

\begin{cor}\label{cor:3.5}
Let $\Omega^\vee (G)$ be the image of the map from Theorem\autoref {thm:1.5} in the set of Bernstein
components for $\Phi_\enh (G)$, and define $\Omega^\vee (G^+)$ analogously. Theorem\autoref {thm:1.5} and 
\eqref{eq:3.16} provide equivalences of categories
\[
\begin{array}{ccccc}
\!\! \Rep (G) & \to & \prod\limits_{\mf s^\vee \in \Omega^\vee (G)} \Mod \big( \mc H (\mf s^\vee, 
q_F^{1/2}) \big) & = & \Mod \Big( \bigoplus\limits_{\mf s^\vee \in \Omega^\vee (G)} 
\mc H (\mf s^\vee, q_F^{1/2}) \Big) , \\
\!\!\Rep (G^+) & \to & \prod\limits_{\mf s^{+\vee} \in \Omega^\vee (G^+)} 
\Mod \big( \mc H (\mf s^{+\vee}, q_F^{1/2}) \big) & = & \Mod \Big( 
\bigoplus\limits_{\mf s^{+\vee} \in \Omega^\vee (G^+)} \mc H (\mf s^{+\vee}, q_F^{1/2}) \Big) .
\end{array}
\]
For irreducible representations the equivalences of categories and \eqref{eq:3.14} provide injections,
which are unions of instances of Theorem\autoref {thm:3.4}:
\[
\begin{array}{lll}
\Irr (G) & \to & \bigsqcup_{\mf s^\vee \in \Omega^\vee (G)} \Phi_\enh (G)^{\mf s^\vee} , \\
\Irr (G^+) & \to & \bigsqcup_{\mf s^{+\vee} \in \Omega^\vee (G^+)} \Phi_\enh (G^+)^{\mf s^{+\vee}} .
\end{array}
\]
\end{cor}

The image of the parametrization maps in Corollary\autoref {cor:3.5} is a union of Bernstein 
components of enhanced $L$-parameters. Surjectivity on the cuspidal level in Theorem 
\autoref{thm:1.1}.d would imply surjectivity in Corollary\autoref {cor:3.5}, then $\Omega^\vee (G)$ 
would be the set of all Bernstein components in $\Phi_\enh (G)$. 
That is known when $F$ is a $p$-adic field, from \cite{Art} and \cite{MoRe}. When $F$ is a local 
function field, that surjectivity has been shown for symplectic and for split special orthogonal 
groups, assuming $p>2$ \cite{GaVa}. By Proposition\autoref {prop:1.9}, the Hypothesis\autoref {as:1.8}
suffices to obtain such surjectivity.\\

Suppose now that $M \subset G$ is a Levi subgroup which contains $L$. It is a direct product of a 
group of the same type as $G$ and of factors $\GL_m (F)$, so all the previous results apply just
as well to $L$. Then $\mc H (\mf s_M) = \End_M (\Pi_{\mf s_M})$ embeds in $\mc H (\mf s)$ via 
normalized parabolic induction and $\mc H (\mf s_M^\vee, q_F^{1/2})$ embeds naturally in 
$\mc H (\mf s^\vee ,q_F^{1/2})$. As isomorphism 
\begin{equation}\label{eq:3.17}
\mc H (\mf s_M)^{\op} \cong \mc H (\mf s_M^\vee,q_F^{1/2})
\end{equation}
we can simply take the restriction of $\mc H (\mf s)^{\op} \cong \mc H (\mf s^\vee,q_F^{1/2})$
from Theorem\autoref {thm:1.7}.b. The same works for $M^+ \subset G^+$, using Paragraph~\ref{par:G+}.
In this setting we can compare the equivalences of categories \eqref{eq:3.16} and their analogues
for $M,M^+$, using normalized parabolic induction.

Let $\bar{E}(\phi,\epsilon,q_F^{1/2})$ be the standard $\mc H (\mf s^\vee,q_F^{1/2})$-module
associated to $(\phi,\epsilon)$ in \cite[\S 2.2 and Theorem 3.18]{AMS3}. By definition
$\bar{M}(\phi,\epsilon,q_F^{1/2})$ is the unique irreducible quotient (``Langlands quotient")
of $\bar{E}(\phi,\epsilon,q_F^{1/2})$.
We let $\pi_{st}(\phi,\epsilon)$ be the image of $\bar{E}(\phi,\epsilon,q_F^{1/2})$ under 
\eqref{eq:3.16}, and we use analogous notations for $G^+, M, M^+$, with superscripts $M$ or +. 
Let us point out that for bounded $\phi$ (and in fact for almost all $\phi$):
\[
\bar{E}(\phi,\epsilon,q_F^{1/2}) = \bar{M}(\phi,\epsilon,q_F^{1/2}) \quad \text{and} \quad
\pi_{st}(\phi,\epsilon) = \pi (\phi,\epsilon) .
\]

\begin{thm}\label{thm:3.6}
Let $(\phi,\epsilon^M) \in \Phi_\enh (M)^{\mf s_M}$ be bounded, or a twist of a bounded parameter
by an element of $\rZ(M^\vee)$ which is positive with respect to $M^\vee B^\vee$ in the sense of
\cite[Appendix A]{AMS2}. Then
\[
I_{MU}^G \big( \pi_{st} (\phi,\epsilon^M) \big) \cong \bigoplus\nolimits_\epsilon \Hom_{\mc S_\phi^M} 
(\epsilon^M, \epsilon) \otimes \pi_{st} (\phi,\epsilon) ,
\]
where the sum runs over all $\epsilon \in \Irr (\mc S_\phi)$ with $\Sc (\phi,\epsilon) =
\Sc (\phi, \epsilon^M)$. The same holds for $M^+ \subset G^+$.
\end{thm}
\begin{proof}
By \cite[Lemma 3.19.a]{AMS3} this holds for $\bar{E}(\phi,\epsilon^M,q_F^{1/2})$ and
$\mr{ind}_{\mc H (\mf s^\vee_M,q_F^{1/2})}^{\mc H (\mf s^\vee,q_F^{1/2})}$. We note that the condition
in \cite[Lemma 3.19.a]{AMS3} is fulfilled by \cite[Proposition A.3]{AMS2} and the assumed properties 
of $\phi$. Via \eqref{eq:3.13} and \eqref{eq:3.17} we obtain the corresponding statement for modules
of $\mc H (\mf s)^{\op}$ and $\mc H (\mf s_M)^{\op}$. By \cite[Condition 4.1 and Lemma 6.1]{SolComp}
the equivalences \eqref{eq:3.12} commute with normalized parabolic induction, which enables us to 
transfer the statement to representations of $G$ and $M$. The same proof works for $M^+ \subset G^+$.
\end{proof}

\section{Comparison of Langlands parameters}
\label{sec:Langlands}

In this section we will compare the enhanced $L$-parameters for $G$ obtained via the endoscopic
methods of Arthur and M\oe glin with the enhanced $L$-parameters associated to irreducible
$\mc G(F)$-representations in Theorem\autoref {thm:3.4}. Although endoscopy only seems to be available 
when $F$ is a $p$-adic field, in Paragraph~\ref{par:close} we showed how the resulting parametrization
can be transferred to classical groups over local function fields. That requires Hypothesis 
\autoref{as:1.8} (which we hope to lift in the future). Then M\oe glin's constructions to find 
enhanced $L$-parameters make sense for any classical group over a non-archimedean local field. Since 
that applies to $G^+$ rather than $G$, we will focus on $G^+$-representations in this paragraph.

We will compare them with our method via Hecke algebras, increasing the classes of 
representations under consideration step by step. For supercuspidal representations the enhanced 
$L$-parameters in Theorem\autoref {thm:3.4} are by definition equal to those constructed in Theorems 
\autoref{thm:1.1} and\autoref {thm:1.2}. The relation between the discrete series and the supercuspidal 
representations of classical groups is due to M\oe glin and Tadi\'c \cite{Moe0,MoTa}, also proven 
with different methods by Kim and Mati\'c \cite{KiMa}.

\subsection{Cuspidal supports of essentially square-integrable representations} \
\label{par:Sc}

There are cuspidal support maps both for irreducible $G$-representations and for $\Phi_\enh (G)$. Recall
from Theorem\autoref {thm:3.4}.a that these maps commute with the assignment of enhanced $L$-parameters
via Hecke algebras. We want to check that the same holds for M{\oe}glin's parameters of discrete series
representations. (Since the cuspidal support maps commute with tensoring by unramified characters,
that implies the same statement for essentially square-integrable representations.) The initial steps
to determine the cuspidal support of $(\phi,\epsilon) \in \Phi_\enh (G)$ are:
\begin{itemize}
\item Replace $(\phi,\epsilon)$ by $\phi |_{\mb W_F}$ and $(\phi |_{\SL_2 (\C)}, \epsilon)$, where
$\phi (\SL_2 (\C))$ lies in $H := \rZ_{{G^\vee}_\der}(\phi (\mb W_F))$ and 
\[
\mc S_\phi = \pi_0 (\rZ_{{G^\vee}_\der} (\phi)) = \pi_0 \big( \rZ_H (\phi |_{\SL_2 (\C)}) \big) .
\]
\item From $(\phi |_{\SL_2 (\C)}, \epsilon)$ we extract the triple
\[
s_\phi = \phi \matje{q_F^{1/2}}{0}{0}{q_F^{-1/2}} 
,\; u_\phi = \phi \begin{pmatrix} 1 & 1 \\ 0 & 1 \end{pmatrix} ,\; 
\epsilon \in \Irr \big( \pi_0 (\rZ_H (s_\phi,u_\phi)) \big).
\]
Such triples can be regarded as $H$-valued enhanced $L$-parameters which are trivial on $\mb W_F$, and
that provides a notion of cuspidal support for such triples. Up to $H$-conjugacy the triple 
$(s_\phi,u_\phi,\epsilon)$ contains precisely the same information as $(\phi |_{\SL_2 (\C)}, \epsilon)$. 
\item The cuspidal support of $(s_\phi, u_\phi, \epsilon)$, in the group $H$, is another triple
$(t,v,\epsilon_c)$ with $t \in H$ conjugate to $s_\phi, v \in H$ unipotent, $t v t^{-1} = v^{q_F}$ and
$\epsilon_c \in \Irr \big( \pi_0 (\rZ_H (t,v)) \big)$. 
\item The cuspidal support of $(\phi,\epsilon)$ is an enhanced $L$-parameter $(\phi_c,\epsilon_c)$
reconstructed from $(\phi |_{\mb W_F},t,v,\epsilon_c)$, so with $\phi_c \matje{1}{1}{0}{1} = v$ and
$\phi_c \big( w, \matje{q_F^{1/2}}{0}{0}{q_F^{-1/2}} \big) = \phi (w) t$ for any arithmetic Frobenius
element $w \in \mb W_F$.
\end{itemize}
We work this out further for discrete enhanced $L$-parameters of $G^+$. (That is a little easier than
for $G$, and yields basically the same information.) From \eqref{eq:1.6} we know that 
$H = \rZ_{{G^{+\vee}}_\der}(\phi (\mb W_F))$ is a direct product of orthogonal and symplectic groups
over $\C$. To complete the above characterization of $\Sc (\phi,\epsilon)$, it suffices to describe
the cuspidal support for triples $(s,u,\epsilon)$ in $\rO_n (\C)$ or $\Sp_{2n}(\C)$.
For that we use the detailed analysis from \cite{Lus-Int} and \cite[\S 5]{Mou}. Fortunately, it turns 
out that there are only very few possibilities for the cuspidal supports \cite[\S 10]{Lus-Int}. \\

\noindent \textbf{Symplectic case} \\
Take a Levi subgroup $L_d = \Sp_{d (d+1)}(\C) \times \GL_1 (\C)^{n - d(d+1)/2}$ of $\Sp_{2n}(\C)$
and let $u_d \in \Sp_{d(d+1)}$ be a unipotent element with Jordan blocks of sizes 
$\{2,4,\ldots,2d\}$. Take any semisimple element $s \in L_d$ with $s u_d s^{-1} = u^{q_F}$. Then
$\pi_0 \big( \rZ_{\Sp_{2n} (\C)} (s,u_d) \big) \cong \F_2^d$ with basis $\{z_2, z_4,\ldots z_{2d}\}$,
and $\epsilon_d (z_{2j}) = (-1)^j$ gives a cuspidal triple $(s,u_d,\epsilon_d)$.

Given a triple $(s,u,\epsilon)$, the only options for $\Sc (s,u,\epsilon)$ are $(s,u_d,\epsilon_d)$
with $d \in \Z_{\geq 0}$. We write
\[
d' = \left\{ \begin{array}{ll}
d+1 & \text{if } d \text{ is even},\\
-d & \text{if } d \text{ is odd}.
\end{array} \right.
\]
In \cite[\S 12]{Lus-Int} the cuspidal support of $(u,\epsilon)$ is computed via this number $d'$,
which is called the defect of $(u,\epsilon)$. Assume for simplicity that all Jordan blocks of $u$
have different size $i_1,i_2,\ldots,i_r$ which are even (this is the case if $(s,u)$ comes
from a discrete $L$-parameter). Write $\pi_0 \big( \rZ_{\Sp_{2n}(\C)} (s,u) \big) = \F_2^r$ with basis
$\{ z_{i_1}, z_{i_2}, \ldots, z_{i_r} \}$. If $r$ is even, we define a new $\tilde \epsilon$ by 
adding $i_0 = 0$ with $\epsilon' (z_0) = 1$, apart from that $\epsilon' = \epsilon$. Then the 
advanced combinatorics in \cite[\S 11]{Lus-Int} entails that 
$d' = \sum_j (-1)^{j+r} \epsilon' (z_{i_j}) \in 1 + 2\Z$. Hence
\begin{equation}\label{eq:4.1}
d = \left\{\begin{array}{ll}
-1 + \sum_j (-1)^{j+r} \epsilon' (z_{i_j}) & \text{if } d' > 0 ,\\
-\sum_j (-1)^{j+r} \epsilon' (z_{i_j}) & \text{if } d' < 0 .
\end{array} \right.
\end{equation}
\textbf{Orthogonal case}\\
Take a Levi subgroup $L_d =\rO_{d^2}(\C) \times \GL_1 (\C)^{(n-d^2)/2}$ of $\rO_n (\C)$ (so with $d \equiv
n$ mod 2) and let $u_d \in\rO_{d^2}(\C)$ be a unipotent element with Jordan blocks of sizes 
$(1,3,\ldots,2d-1)$. Let $s \in L_d$ be semisimple such that $s u_d s^{-1} = u_d^{q_F}$. Then 
$\pi_0 (\rZ_{\rO_{d^2}(\C)} (s,u_d)) \cong \F_2^d$ with basis $\{z_1,z_3,\ldots, z_{2d-1}\}$ and 
$\epsilon_d (z_{2j-1}) = (-1)^j$ and $-\epsilon_d$ give two cuspidal triples $(s,u_d,\pm \epsilon_d)$.

Given a triple $(s,u,\epsilon)$ for $\rO_n (\C)$, the options for $\Sc (s,u,\epsilon)$ are
$(s,u_d,\pm \epsilon_d)$ with $d \in \Z_{\geq 0}$ of the same parity as $n$. In this case $d$ is the
defect of $(u,\epsilon)$ \cite[\S 13]{Lus-Int}. Suppose that all Jordan blocks of $u$ have different
sizes $1_1,i_2,\ldots,i_r$, which are all odd (as for discrete $L$-parameters). Then
\cite[\S 13]{Lus-Int} entails that 
\begin{equation}\label{eq:4.2}
d = \big| \sum\nolimits_j (-1)^j \epsilon (z_{i_j}) \big| .
\end{equation}
By that and \cite{AMS1}, $\Sc (s,u,\epsilon) = (s,u_d,\pm \epsilon_d)$ where the sign is determined by
$\pm \epsilon_d (z) = \epsilon (z)$ for some $z \in\rO_{d^2} (\C) \setminus \SO_{d^2} (\C)$. We embed
$\rO_{d^2}(\C)$ in $\rO_n (\C)$ so that the subgroup $\rO_1 (\C) \subset \rZ_{\rO_{d^2}(\C)}(u_d)$, which comes
from the Jordan block of size $1$, is contained in a subgroup $\rO_{i_m}(\C) \subset \rZ_{\rO_n (\C)}(u)$
which comes from a Jordan block of size $i_m$. Then we take $z = z_1$, and we find (using that
$i_m$ is odd)
\begin{equation}\label{eq:4.3}
\pm \epsilon_d (z_1) = \epsilon (z_1) = \epsilon (z_1)^{i_m} = \epsilon (z_{i_m}) .
\end{equation}
This determines the sign, and thus fixes $\Sc (s,u,\epsilon)$.

\begin{prop}\label{prop:4.1}
M\oe glin's parametrization of the discrete series of $G^+$ is compatible with the cuspidal support
maps, in the following sense. For a discrete series representation $\pi \in \Irr (G^+)$ with
$\Sc (\pi) \in \Irr (L^+)$, $\Sc (\phi_\pi, \epsilon_\pi)$ is $\rN_{G^{+\vee}} (L^{+\vee})$-conjugate
to $(\phi_{\Sc (\pi)}, \epsilon_{\Sc (\pi)})$.
\end{prop}
\begin{proof}
In \cite{Moe0} the cuspidal support of $\pi$ is studied in relation with $\phi_\pi$ and $\epsilon_\pi$.
The M\oe glin parameter of $\Sc (\pi)$ is obtained via a recursive procedure, whose important steps
are mentioned on \cite[p. 147]{Moe0}.

Suppose first that $a,a' \in \mr{Jord}_\rho (\pi)$ are adjacent (that is, no $b$ inbetween $a$ and $a'$ 
belongs to $\mr{Jord}_\rho (\pi)$) and that $\epsilon_\pi (\rho,a) = \epsilon_\pi (\rho,a')$. Then
$\{(\rho,a),(\rho,a')\}$ can be removed from Jord$(\pi)$, and the new $(\mr{Jord}',\epsilon')$
corresponds to a discrete series representation with the same cuspidal support as $\pi$ (apart from
$(a+a') d_\rho$ extra factors $\GL_1 (\C)$ in the Levi subgroup from $\Sc (\pi)$). This enables us 
to reduce to the cases where $\epsilon_\pi$ is alternated in the sense that $\epsilon_\pi (\rho,a) = 
-\epsilon_\pi (\rho,a')$ whenever $a,a' \in \mr{Jord}_\rho (\pi)$ are adjacent. 

Suppose now that $\epsilon_\pi$ is alternated. 
\begin{enumerate}[(i)]
\item If $\mr{Jord}_\rho (\pi)$ consists of even numbers $a$ and $\epsilon_\pi (\rho,a) = -1$ for
the minimal such $a$, then $\mr{Jord}_\rho (\Sc (\pi)) = \{2,4,\ldots,2d\}$ with 
$d = |\mr{Jord}_\rho (\pi)|$ and $\epsilon_{\Sc (\pi)} (\rho,2a) = (-1)^a$.
\item If $\mr{Jord}_\rho (\pi)$ consists of even numbers $a$ and $\epsilon_\pi (\rho,a) = 1$ for
the minimal such $a$, $\mr{Jord}_\rho (\Sc (\pi)) = \{2,4,\ldots,2d\}$ with 
$d = |\mr{Jord}_\rho (\pi)| - 1$ and $\epsilon_{\Sc (\pi)} (\rho,2a) = (-1)^a$.
\item If $\mr{Jord}_\rho (\pi)$ consists of odd numbers, then $\mr{Jord}_\rho (\Sc (\pi)) =
\{1,3,\ldots,2d-1\}$ where $d = |\mr{Jord}_\rho (\pi)|$ and $\epsilon_{\Sc (\pi)}(\rho,1) =
\epsilon (\rho,a)$ for the minimal $a \in \mr{Jord}_\rho (\pi)$. This last property is implicit
on \cite[p. 147]{Moe0}, which mentions that here $\epsilon$ does not change if we pass from 
$\pi$ to $\Sc (\pi)$. 
\end{enumerate}
If we now compute the above numbers $d$ in terms of the original $\epsilon_\pi$, we recover
precisely \eqref{eq:4.1} and \eqref{eq:4.2}. In case (iii) we can embed $\rO_{d^2}(\C)$ in
$\rO_n (\C)$ such that the part $P_{2a-1} : \SL_2 (\C) \to\rO_{2a-1}(\C)$ of $\mr{Jord}_\rho (\Sc (\pi))$
lands in the subgroup $\rO_{2 i_a - 1}(\C) \subset\rO_n (\C)$ that contains the image of the
part $P_{2 i_a - 1}$ of $\mr{Jord}_\rho (\pi)$. Then the aforementioned property 
$\epsilon_{\Sc (\pi)}(\rho,1) = \epsilon (\rho,a)$ becomes \eqref{eq:4.3}. 
\end{proof}

\subsection{Jordan blocks of discrete series representations} \
\label{par:Jordan}

First we check that the Jordan blocks of a discrete series representation $\pi$ of $G^+ = G_n^+$
can be read off directly from the enhanced $L$-parameter assigned to it by Theorem\autoref {thm:3.4}.

\begin{lem}\label{lem:4.6}
Let $(\phi,\epsilon) \in \Phi_\enh (G^+)$ be bounded and discrete, such that $\pi (\phi,\epsilon)$
is defined in Theorem\autoref {thm:3.4}. Then $\Jord(\pi (\phi,\epsilon))$
corresponds to $\Jord(\phi)$ under the local Langlands correspondence for general linear groups.
\end{lem}
\begin{proof}
Let $\rho \in \Irr_\cusp (\GL_{d_\rho}(F))$ with $\rho \cong \rho^\vee \otimes 
\nu_{\pi (\phi,\epsilon)}$, a condition which by \eqref{eq:1.44} is fulfilled by all Jordan blocks
of $\pi (\phi,\epsilon)$. Recall that a pair $(\rho,a)$ belongs to Jord$(\pi (\phi,\epsilon))$ 
if and only if 
\begin{itemize}
\item $\delta (\rho,a) \times \pi (\phi,\epsilon)$ is irreducible and
\item $\delta (\rho,a') \times \pi (\phi,\epsilon)$ is reducible for some $a' \in a + 2\Z$. 
\end{itemize}
Let $\tau \in \Irr (\mb W_F)$ be the $L$-parameter of $\rho$. Then $\tau \cong \tau^\vee \otimes
\mu_G^\vee \circ \phi$ by \eqref{eq:1.45} and Theorem\autoref {thm:3.4}.e, as needed for Jord$(\phi)$
by \eqref{eq:1.4}. We write
\[
\psi = \tau \otimes P_a \times \phi = (\tau \oplus \tau^\vee \otimes \mu_G^\vee \circ \phi) \otimes
P_a \oplus \phi ,
\]
where in the middle we work in $\GL_m (F) \times G_n^+$ and on the right in $G_{n+m}^+$.
Theorem\autoref {thm:3.6} tells us that
\begin{equation}\label{eq:4.15}
\delta (\rho,a) \times \pi (\phi,\epsilon) = \bigoplus\nolimits_\eta 
\Hom_{\mc S_\phi^+} (\epsilon,\eta) \otimes \pi (\psi,\eta) ,
\end{equation}
where the sum runs over all $\eta \in \mc S_\psi^+$ with $\Sc (\psi,\eta) = \Sc (\psi,\epsilon)$.
The groups $\mc S_\phi^+$ and $\mc S_\psi^+$ can be compared with \eqref{eq:1.3}.
\begin{enumerate}[(i)]
\item When sgn$(\tau \otimes P_a) \neq \mr{sgn}({G^\vee}_\der)$: $\mc S_\psi^+ = \mc S_\phi^+$ and
\eqref{eq:4.15} is always reducible.
\item When sgn$(\tau \otimes P_a) = \mr{sgn}({G^\vee}_\der)$ and $(\tau,a) \in \mr{Jord}(\phi)$:
again $\mc S_\psi^+ = \mc S_\phi^+$ and \eqref{eq:4.15} is reducible.
\item When sgn$(\tau \otimes P_a) = \mr{sgn}({G^\vee}_\der)$ and $(\tau,a) \notin \mr{Jord}(\phi)$:
$\mc S_\psi^+ = \mc S_\phi^+ \times \{1,z_{\tau,a}\}$. Then $\Sc (\psi,\eta) = \Sc (\psi,\epsilon)$
for every extension $\eta$ of $\epsilon$ to $\mc S_\psi^+$, because $\tau \otimes P_a$ occurs with
even multiplicity in $\psi$ and hence does not influence the cuspidal support. In this case 
\eqref{eq:4.15} is a direct sum of two inequivalent irreducible representations.
\end{enumerate}
We compare this with the aforementioned characterization of Jord$(\pi (\phi,\epsilon))$. 
The reducibility of $\delta (\rho,a') \times \pi (\phi,\epsilon)$ rules out case (i), and 
$\delta (\rho,a) \times \pi (\phi,\epsilon)$ is reducible in case (ii) but not in case (iii). 
We conclude that $(\rho,a) \in \mr{Jord}(\pi (\phi,\epsilon))$ if and only if $(\tau,a) \in 
\mr{Jord}(\phi)$.
\end{proof}

Recall that the parametrization of the discrete series in Theorem\autoref {thm:1.1} involves the Jordan
blocks of $\phi$ and a character $\epsilon_\pi : \mc S_\pi \to \{\pm 1\}$. To facilitate a comparison 
with our Hecke algebra methods, we revisit M\oe glin's construction of $\epsilon_\pi$ \cite{Moe0} and 
we show that it shares some properties with the constructions behind Theorem\autoref {thm:3.4}.

Let $(\phi,\epsilon) \in \Phi_\enh (G^+ )^{\mf s^{+\vee}}$ be discrete and bounded, and let 
$\pi (\phi,\epsilon)$ be the discrete series representation of $G^+ = G_n^+$ associated to it by 
Theorem\autoref {thm:3.4}. Recall that $\mc S_\phi^+$ is the $\F_2$-vector space with basis
$\{z_{\tau,a} : (\tau,a) \in \mr{Jord}(\phi) \}$. For such a $\tau$ we let $\rho$ be the corresponding
representation of $\GL_{d_\rho}(F)$ and we write $\epsilon (z_{\rho,a}) = \epsilon (z_{\tau,a})$.
Here we use that by Lemma\autoref {lem:4.6} the Jordan blocks of $\phi$ and of $\pi (\phi,\epsilon)$ are 
matched via the local Langlands correspondence for general linear groups.

\begin{prop}\label{prop:4.2}
Let $a > a' \in \mr{Jord}_\rho (\pi (\phi,\epsilon))$ be adjacent. 
\enuma{
\item $\epsilon (z_{\rho,a}) = \epsilon (z_{\rho,a'})$ if and only if $\pi (\phi,\epsilon)$ embeds
in $\delta (\rho, (a-1)/2,(1-a')/2) \times \tilde \pi$ for some discrete series representation
$\tilde \pi$ of $G^+_{n - d_\rho (a + a')/2}$. Moreover, in this case 
\[
\mr{Jord}(\tilde \pi) = \mr{Jord}(\pi (\phi,\epsilon)) \setminus \{ (\rho,a), (\rho,a') \}
\]
and $\tilde \pi = \pi (\tilde \phi,\tilde \epsilon)$ where $\tilde \epsilon = 
\epsilon |_{\mc S_{\tilde \phi}}$.
\item Suppose that $a_-$ is the minimal element of $\mr{Jord}_\tau (\phi)$ and that it is even.
Then part (a) also holds with $a = a_- ,\; a' = 0$, provided we put $\epsilon (z_{\rho,0}) = 1$.
}
\end{prop} 
\begin{proof}
(a) Suppose that $\pi (\phi,\epsilon)$ is a subrepresentation of $\delta (\rho, (a-1)/2,(1-a')/2) 
\times \tilde \pi$.
We write $M = \GL_{d_\rho (a+a')/2} \times G_{n - d_\rho (a+a')/2}$, so that $MU$ is a parabolic 
subgroup of $G_n$ with Levi factor $M$. From Theorem\autoref {thm:3.4} we get $\tilde \pi = \pi (\tilde 
\phi,\tilde \epsilon)$ for some discrete bounded $\tilde \phi \in \Phi (M^+)$. The $L$-parameter of 
\[
\delta (\rho, (a-1)/2,(1-a')/2) \quad \text{is} \quad 
\tau \otimes P_{(a+a')/2} \otimes |\cdot|^{(a - a')/4}, 
\]
and $|\cdot|^{(a - a')/4}$ is in positive position with respect to $M^+ U$. 
Thus Theorem\autoref {thm:3.6} is applicable, and it says that
\begin{multline*}
\delta (\rho, (a-1)/2,(1-a')/2) \times \tilde \pi = I_{M^+ U}^{G^+} \pi \big( \tau \otimes 
P_{(a+a')/2} \otimes |\cdot|^{(a - a')/4} \times \tilde \phi, \tilde \epsilon \big) \\
= \bigoplus\nolimits_{\epsilon'} \Hom_{\mc S_\phi^{M^+}} (\tilde \epsilon, \epsilon' ) \otimes
\pi \big( P_{(a+a')/2} \otimes |\cdot|^{(a - a')/4} \times \tilde \phi, \epsilon' \big) ,
\end{multline*}
where the sum runs over all $\epsilon' \in \Irr (\mc S_\phi^+)$ with 
\[
\Sc (\phi,\epsilon) = \Sc \big( \tau \otimes P_{(a+a')/2} \otimes 
|\cdot|^{(a - a')/4} \times \tilde \phi, \epsilon' \big).
\]
It follows that $\phi$ is $G^{+\vee}$-conjugate to $P_{(a+a')/2} \otimes |\cdot|^{(a - a')/4} 
\times \tilde \phi$ and that $\Hom_{\mc S_\phi^{M^+}} (\tilde \epsilon, \epsilon )$ is nonzero. 
We deduce that 
\begin{equation}\label{eq:4.10}
\mc S_\phi^+ = \mc S_{\tilde \phi}^+ \times \langle z_{\tau,a} , z_{\tau',a} \rangle
\end{equation}
and that $\tilde \epsilon = \epsilon |_{\mc S_{\tilde \phi}^+}$. Our assumption entails that 
$\pi (\phi,\epsilon)$ and $\delta (\rho, (a-1)/2,(1-a')/2) \times \tilde \pi$ have the same 
cuspidal support. Now Theorem\autoref {thm:3.4}.a and the formulas \eqref{eq:4.1} and \eqref{eq:4.2}
for the cuspidal support of enhanced $L$-parameters show that $\epsilon (z_{\rho,a}) = 
\epsilon (z_{\rho,a'})$.

Conversely, suppose that $\epsilon (z_{\rho,a}) = \epsilon (z_{\rho,a'})$. Write $\phi$ as 
$L$-parameter $\phi_{a,a'} \times \tilde \phi$ for $M$. Then \eqref{eq:4.10} holds and we can take
$\tilde \epsilon = \epsilon |_{\mc S_{\tilde \phi}^+}$. The $L$-parameter $\phi_{a,a'} \in 
\Phi (\GL_{(a+a')/2}(F))$ is discrete and its cuspidal support consists of
terms $\tau |\cdot|^r$ with $r \in \R$. Hence $\phi_{a,a'} = \tau \otimes P_{(a+a')/2} \otimes 
|\cdot|^r$ for some $r \in \R$. Embedding in $\Phi (G_n)$ and comparing with the shape of 
$\phi$ we find
\begin{equation}\label{eq:4.13}
P_{(a+a')/2} \otimes |\cdot|^r \oplus P_{(a+a')/2} \otimes |\cdot|^{-r} = P_a \oplus P_{a'} ,
\end{equation}
or at least up to conjugation in $\GL_{a+a'}(\C)$. That entails $r = (a-a')/4$, from which
we deduce that
\[
\pi (\phi_{a,a'}) = \delta (\rho,(a+a')/2) \otimes |\cdot|^{(a-a')/4} = \delta (\rho,(a-1)/2,(1-a')/2).
\]
By Theorem\autoref {thm:3.6}
\begin{equation}\label{eq:4.11}
I_{M^+ U}^{G^+} \pi \big( \delta (\rho,(a-1)/2,(1-a')/2) \boxtimes \pi (\tilde \phi, \tilde \epsilon) 
\big) = \bigoplus\nolimits_{\epsilon'} \Hom_{\mc S_\phi^{M^+}} (\tilde \epsilon, \epsilon' ) \otimes
\pi (\phi,\epsilon') ,
\end{equation}
where the sum runs over all $\epsilon' \in \Irr (\mc S_\phi^+)$ with 
\[
\Sc (\phi,\epsilon') =
\Sc \big( \phi_{a,a'} \times \tilde \phi, \tilde \epsilon \big).= \Sc (\phi,\epsilon) .
\]
We note that in \eqref{eq:4.11} we may use irreducible representations instead of the 
standard modules from Theorem\autoref {thm:3.6}, because the latter are irreducible (since 
$\phi$ and $\tilde \phi$ are bounded and $\phi_{a,a'}$ is a twist of a discrete 
parameter by an unramified character).

To get a nonzero contribution to \eqref{eq:4.11}, $\epsilon' |_{\mc S_{\tilde \phi}^+}$ must equal 
$\tilde \epsilon = \epsilon |_{\mc S_{\tilde \phi}^+}$. Then we see from \eqref{eq:4.1} and 
\eqref{eq:4.2} that $\epsilon' (z_{\rho,a}) = \epsilon' (z_{\rho,a'})$. In other words, the only
nontrivial contributions to \eqref{eq:4.11} come from $\epsilon$ and one other $\epsilon'$, and it
reduces to
\begin{equation}\label{eq:4.12}
I_{M^+ U}^{G^+} \pi \big( \delta (\rho,(a-1)/2,(1-a')/2) \boxtimes \pi (\tilde \phi, \tilde \epsilon) 
\big) = \pi (\phi,\epsilon) \oplus \pi (\phi,\epsilon') . 
\end{equation}
In particular $\pi (\phi,\epsilon)$ embeds in the left hand side of \eqref{eq:4.12}, which can 
be written as
\[
\delta (\rho,(a-1)/2,(1-a')/2) \times \pi (\tilde \phi, \tilde \epsilon) .
\]
As $\tilde \phi$ is discrete and bounded, Theorem\autoref {thm:3.4}.b,c guarantees that 
$\pi (\tilde \phi, \tilde \epsilon)$ belongs to the discrete series.

That proves the equivalence. The description of Jord$(\tilde \pi)$ occurs at various places in the
above arguments, it is seen most clearly from \eqref{eq:4.13}.\\ 
(b) This can be shown in the same way as part (a). Notice that $a_-$ needs to be even to make
sense of the $\SL_2 (\C)$-representation $P_{a_-/2}$.
\end{proof}

It was shown in \cite[Proposition 5.3 and Lemme 5.4]{Moe0} that the M\oe glin para\-meters of 
discrete series representations also satisfy Proposition\autoref {prop:4.2}. 

Next we zoom in on a particular class defined in \cite[\S 1]{Moe0}, completely positive discrete 
series representations. By \cite[Proposition 5.3]{Moe0}, among the discrete series these are 
precisely the $\pi$ for which $\epsilon_\pi$ is alternated:
\begin{equation}\label{eq:4.14}
\epsilon (z_{\rho,a}) = -\epsilon (z_{\rho,a'}) \text{ for adjacent } a,a' \in \mr{Jord}_\rho (\pi).
\end{equation}
A few useful properties of such representations follow directly from our description of the
cuspidal support maps.

\begin{cor}\label{cor:4.11}
Let $\pi$ be a completely positive discrete series representation of $G^+$. 
\enuma{ 
\item $\pi$ is uniquely determined by $\Jord(\pi)$ and $\Sc (\pi)$. 
\item $\mr{Jord}_\rho (\pi) \neq \emptyset$ if and only if 
$\mr{Jord}_\rho (\Sc (\pi)) \neq \emptyset$.
}
\end{cor}
\begin{proof}
(a) Under the condition \eqref{eq:4.14} we see from \eqref{eq:4.1}, \eqref{eq:4.2} and 
\eqref{eq:4.3} that $\epsilon_\pi$ is uniquely determined by $\phi_\pi$ and 
$\Sc (\phi_\pi,\epsilon_\pi)$. Combining that with Theorem\autoref {thm:1.1} and Proposition 
\autoref{prop:4.1}, we see that $\pi$ is uniquely determined by Jord$(\pi)$ and $\Sc (\pi)$.\\
(b) This follows from \eqref{eq:4.1} and \eqref{eq:4.2}: under the condition \eqref{eq:4.14}
these numbers $d$ cannot be 0.
\end{proof}

\begin{lem}\label{lem:4.3}
Let $\pi$ be a completely positive discrete series representation of $G^+$ and let 
$(\phi_\pi,\epsilon_\pi)$ be its M\oe glin parameter. Then the $G^+$-representation $\pi'$
attached to $(\phi_\pi,\epsilon_\pi)$ by Theorem\autoref {thm:3.4} is isomorphic to $\pi$.
\end{lem}
\begin{proof}
By Theorem\autoref {thm:3.4}.b,c $\pi'$ is discrete series and from Lemma\autoref {lem:4.6} we know that
Jord$(\pi')$ and Jord$(\pi)$ both correspond to Jord$(\phi_\pi)$, so $\pi$ and $\pi'$
have precisely the same Jordan blocks. By Theorem\autoref {thm:3.4}.a and Proposition\autoref {prop:4.1}
both $\Sc (\pi)$ and $\Sc (\pi')$ have enhanced $L$-parameter $\Sc (\phi_\pi,\epsilon_\pi)$,
so $\Sc (\pi) \cong \Sc (\pi')$.\\
By \eqref{eq:4.14} and Proposition\autoref {prop:4.2}.a $\pi'$ cannot be embedded in 
$\delta (\rho,(a-1)/2,(1-a')/2) \times \tilde \pi$ for adjacent $a > a' \in \mr{Jord}_\rho (\pi) 
= \mr{Jord}_\rho (\pi')$ and a discrete series representation $\tilde \pi$. Then
\cite[\S 5]{Moe0} entails that $\pi'$ is a completely positive discrete series 
representation. Now Corollary\autoref {cor:4.11}.a shows that $\pi \cong \pi'$.
\end{proof}

\subsection{Intertwining operators for discrete series representations} \
\label{par:intertwining}

For general discrete series representations $\pi$, Proposition\autoref {prop:4.2} achieves a kind of
reduction to the completely positive instances without changing cuspidal supports. Indeed, when
$\pi$ is not completely positive, Proposition\autoref {prop:4.2} allows us to replace it by 
$\tilde \pi$, to which we can apply the same considerations. By recursion that yields a 
completely positive discrete series representation that we denote $\pi^+$. In the process 
some direct factors of $\mc S_\phi$ are removed, so we lose information about $\epsilon_\pi$. 
Most values of $\epsilon_\pi$ can be reconstructed from data for $\pi^+$, but not all. For the 
missing ones we will need to study certain normalized intertwining operators.

Suppose that $(\rho,a) \in \mr{Jord}(\pi)$ with $a$ odd and that $\mr{Jord}_\rho (\pi^+)$ is 
empty. Such $\rho$ provide the only parts of $\epsilon_\pi$ that cannot be recovered from
$\epsilon_{\pi^+}$. We note the $L$-parameters of such $\rho$ are precisely the $\tau \in 
\Irr (\mb W_F )_\phi^+$ for which $\ell_\tau = 0 < e_\tau$.

By Corollary\autoref {cor:4.11}.b we may equally well assume that $\mr{Jord}_\rho (\Sc (\pi))$ 
is empty. Then Proposition\autoref {prop:4.2} leaves two possibilities for $\epsilon_\pi$ on
$\mr{Jord}_\rho (\pi)$, distinguished by $\epsilon_\pi (z_{\rho,a_-})$ where $a_- = \mr{min}
(\mr{Jord}_\rho (\pi))$. The characterization of $\epsilon_\pi (z_{\rho,a_-}) \in \{\pm 1\}$ from 
\cite[\S 6.1.1]{Moe0} involves several steps, which we recall next. Write 
\[
\Sc (\pi) = \sigma_1 \boxtimes \cdots \boxtimes \sigma_d \boxtimes \sigma_-  \in \Irr (L^+) , 
\] 
where $\sigma_i \in \Irr (\GL_{n_i}(F))$ and $\sigma_- \in \Irr (G_{n_-}^+)$. Then $\sigma_-$ is
the partial cuspidal support of $\pi$, as used in \cite[(1)]{Moe0}. There is a
holomorphic family of intertwining operators
\begin{equation}\label{eq:4.17}
J(s_\beta, \rho \nu^b \times \sigma_-) \in \Hom_{G_{n_- + d_\rho}^+} (\rho \nu^b \times \sigma_ - ,
\rho \nu^{-b} \times \sigma_-) ,
\end{equation}
where $b \in \C$ and $\nu (g) = |\det (g)|_F$. This family can be normalized via a choice of
an intertwining operator in the case $b = 0$, an element
\begin{equation}\label{eq:4.26}
J(s_\beta, \rho \times \sigma_-) \in \End_{G_{n_- + d_\rho}^+} (\rho \times \sigma_-) .
\end{equation}
There are two possibilities which square to the identity, we choose the one from 
\cite[\S 6.1.2]{Moe0} and \cite[\S 2.4]{Art}, which is determined by compatibility with endoscopy
and depends on a Whittaker datum for the quasi-split inner form of $G$. \\

\textbf{Remark.} The normalization of the operators \eqref{eq:4.26} will make 
Theorem\autoref {thm:1.7} and Proposition\autoref {prop:3.2} canonical (up to inner automorphisms).\\

From \eqref{eq:4.17} one obtains a family of intertwining operators
\begin{equation}\label{eq:4.18}
J(s_\beta \times s_\beta, \rho \nu^{b_1} \times \rho \nu^{b_2} \times \sigma_-) : \rho \nu^{b_1} 
\times \rho \nu^{b_2} \times \sigma_- \to \rho \nu^{-b_1} \times \rho \nu^{-b_2} \times \sigma_- ,
\end{equation}
which reduces to \eqref{eq:4.17} (tensored with the identity on one of the $\rho \nu^{b_i}$) upon
applying normalized Jacquet restriction. The same works with more factors $\rho \nu^{b_i}$. 

In $\GL_{2 d_\rho}(F)$, the element $s_{12}$ that exchanges the two blocks of
$\GL_{d_\rho}(F) \times \GL_{d_\rho}(F)$ induces an intertwining operator
\begin{equation}\label{eq:4.19}
J(s_{12}, \rho \nu^{b_1} \times \rho \nu^{b_2} \times \sigma_-) \in 
\Hom_{G_{n_- + 2 d_\rho}^+}( \rho \nu^{b_1} \times \rho \nu^{b_2} \times \sigma_-,  
\rho \nu^{b_2} \times \rho \nu^{b_1} \times \sigma_-), 
\end{equation}
where $b_1,b_2 \in \C$. We normalize it so that it depends holomorphically on $b_1 - b_2$ and 
becomes the identity when $b_1 = b_2$. 

Let $e \in \N$ be odd and let $b_1,b_2,\ldots,b_{(e-1)/2} \in \C$. The order two permutation
\[
w_e: = (s_\beta,s_\beta,\ldots,s_\beta) \circ (1 \; e) (2 \; e\!-\!1) \cdots 
((e\!-\!1)/2 \,\, (e\!+\!3)/2)
\]
belongs to the Weyl group $W(B_e)$. The composition of the 
corresponding operators \eqref{eq:4.18} and \eqref{eq:4.19} yields an intertwining operator
\begin{equation}\label{eq:4.20}
J(w_e, \rho \nu^{b_1} \times \cdots \times \rho \nu^{b_{(e-1)/2}} \times \rho \times
\rho \nu^{-b_{(e-1)/2}} \times \cdots \times \rho \nu^{-b_1} \times \sigma_-) ,
\end{equation}
from the indicated $G^+_{n_- + e d_\rho}$-representation to itself. The upshot of \cite[p. 176]{Moe0}
is that the holomorphic family (with variables $b_i$) of intertwining operators \eqref{eq:4.20} 
can be normalized so that each operator \eqref{eq:4.20} squares to the identity and they reduce to 
\eqref{eq:4.18} in the special case $b_i = 0$ for all $i$. All these intertwining operators are unique 
up to scalars, so our conditions leave just the choice of a sign, which in turn is determined by 
$J(s_\beta, \rho \times \sigma_-)$.

Pick $a \in \mr{Jord}_\rho (\pi) \setminus \{a_-\}$ with 
$\epsilon (z_{\rho,a}) = \epsilon (z_{\rho,a_-})$, and embed $\pi$ in
\begin{equation}\label{eq:4.22}
\delta (\rho,(a-1)/2,(1-a_-)/2) \times \tilde \pi \quad \subset \quad 
\delta (\rho,(a-1)/2,(1 + a_-)/2) \times \delta (\rho, a_-) \times \tilde \pi 
\end{equation}
for a discrete series representation $\tilde \pi$ of $G_{n-(a+a_-)/2}^+$. Embed the right hand side in
\begin{equation}\label{eq:4.21}
\delta (\rho,(a-1)/2,(1 + a_-)/2) \times \mr{Ind} \, \Sc (\delta (\rho, a_-)) \times 
\mr{Ind} \, \Sc (\tilde \pi), 
\end{equation}
where Ind stands for normalized parabolic induction. We note that $\sigma_-$ is a factor of 
$\Sc (\tilde \pi)$ and that
\[
\mr{Ind} \, \Sc (\delta (\rho, a_-)) = \rho \nu^{(a_- -1)/2} \times \cdots \times \rho \nu 
\times \rho \times \rho \nu^{-1} \times \cdots \times \rho \nu^{(1 - a_-)/2} , 
\]
which fits with \eqref{eq:4.20}. Then \eqref{eq:4.20} and the identity
on the other factors of \eqref{eq:4.21} induce a self-intertwining operator of \eqref{eq:4.21}.
That operator can be restricted to \eqref{eq:4.22} and thus yields a normalized intertwining operator 
\[
N(\rho,a_-) \in \End_{G^+}
\big( \delta (\rho,(a-1)/2,(1 + a_-)/2) \times \delta (\rho, a_-) \times \tilde \pi \big) .
\]
Then $\epsilon_\pi (z_{\rho,a_-})$ is the scalar by which $N(\rho,a_-)$ acts on $\pi$, or
equivalently
\begin{equation}\label{eq:4.4}
\pi \text{ is fixed pointwise by } \epsilon_\pi (z_{\rho,a_-}) N(\rho,a_-) .
\end{equation}
We emphasize that the one choice of $J(s_\beta, \rho \times \sigma_-)$ determines a normalization
for (potentially) many instances of \eqref{eq:4.4}.

In the Hecke algebras $\mc H (\mf s^\vee,q_F^{1/2})$ and $\mc H (\mf s^{+\vee},q_F^{1/2})$ we
also have intertwining operators, they come from the underlying geometric setup \cite{AMS2}.
In the general setting of Theorem\autoref {thm:3.4}, the way an enhancement 
$\epsilon$ of $\phi$ helps to find the irreducible representation $\pi (\phi,\epsilon)$ is
by applying $\Hom_{\mc S_\phi}(\epsilon,?)$ to a standard module $\pi (\phi,\mf s^\vee)$ 
constructed from $\phi$ and the cuspidal support. 

In the case at hand, for $\pi = \pi (\phi,\epsilon)$ we have 
\[
\mr{Hom}_{\mc S_{\tilde \phi}^+} (\epsilon, \pi (\phi, \mf s^{+\vee})) = 
\delta (\rho,(a-1)/2,(1-a_-)/2) \times \tilde \pi ,
\]
where $\mc S_{\tilde \phi}^+$ comes from $\tilde \pi$. When $\tau$ is the $L$-parameter of $\rho$, 
the geometric setup provides a canonical action of $\{ 1, z_{\tau,a_-}\}$ on this representation, 
by a $G^+$-intertwining operator that we denote $N(\tau,a_-)$. That means
\begin{align}\nonumber  
\pi = \pi (\phi,\epsilon) & = \Hom_{\mc S_\phi^+} (\epsilon, \pi (\phi, \mf s^{+\vee})) =   
\Hom_{\langle z_{\tau,a_-}\rangle} (\epsilon, \delta (\rho,(a-1)/2,(1-a_-)/2) \times \tilde \pi ) \\
\label{eq:4.5} & = \{ \text{fixed points of } \epsilon (z_{\tau,a-}) N(\tau,a_-) \text{ in } 
\delta (\rho,(a-1)/2,(1-a_-)/2) \times \tilde \pi \} .
\end{align}
In contrast with $J(s_\beta, \rho \times \sigma_-)$, these intertwining operators for Hecke algebra
representations do not have to be normalized, they arise naturally. The only freedom we have is
that from Theorem\autoref {thm:1.7}, which we will use next. Let $J(s_\beta, \tau \times \phi_-)$ 
be the canonical intertwining operator associated to $s_\beta$ and the Hecke algebra representation 
corresponding to $\rho \times \sigma_-$ via Proposition\autoref {prop:3.2}. (We suppress 
$\epsilon_{\sigma_-}$ from this notation.)

\begin{prop}\label{prop:4.4}
Let $\mf s^+ = [L^+,\sigma]_{G^+}$ be an arbitrary inertial equivalence class for $G^+$.
Fix a Whittaker datum for the quasi-split inner form of $G$, so that the intertwining operators
\eqref{eq:4.26} are determined.

The isomorphism $\mc H (\mf s^+)^{\op} \cong \mc H (\mf s^{+\vee},q_F^{1/2})$ from Proposition
\autoref{prop:3.2} can be chosen such that the following holds. 
For every  $\tau \in \Irr (\mb W_F )_{\phi_\sigma}^+$ with $e_\tau > 0 = \ell_\tau$, the 
intertwining operators  $J(s_\beta,\tau \times \phi_-)$ and $J(s_\beta, \rho \times \sigma_-)$ 
from \eqref{eq:4.26} agree via the appropriate equivalences of categories from \eqref{eq:3.16} 
induced by the chosen Hecke algebra isomorphism.

A Hecke algebra isomorphism with these properties is canonical up to conjugation by elements of 
$\mc O (T_{\mf s+})^\times$.
\end{prop}
\begin{proof}
Let $\rho'$ be an unramified twist of $\rho$ such that $\rho' \cong \rho^{'\vee} \otimes 
\nu_\pi$ and $\rho' \not\cong \rho$. Because $\rho'$ influences the structure of 
$\mc H (\mf s^{+\vee},q_F^{1/2})$ in the part coming from the same irreducible root system as 
$\rho$, we have to consider $\rho$ and $\rho'$ simultaneously. Let $\tau$ and $\tau'$ be the 
$L$-parameters of respectively $\rho$ and $\rho'$.\\

\noindent
\textbf{The case $\ell_{\tau'} > 0$.} \\
From \eqref{eq:1.12} we know that
the relevant tensor factor of $\mc H (\mc R_{\mf s^\vee,\der},\lambda,\lambda^*,q_F^{1/2}) \rtimes
\Gamma_{\mf s^\vee}^+$ is an affine Hecke algebra $\mc H_{\rho'}$ with underlying root datum 
$(\Z^{e_{\rho'}},B_{e_{\rho'}},\Z^{e_{\rho'}},C_{e_{\rho'}})$. The base point of $T_{\mf s^\vee,\rho}$ 
for $\mc H_{\rho'}$ comes from $\rho'$, and $\rho^{\boxtimes e_{\rho'}}$ is related to this 
basepoint by an order two element of the associated complex torus. The condition 
$e_\tau > 0 = \ell_\tau$ entails that $\ell_\rho = 0, a_\rho = -1$ and $q_\beta^* = 1$ for the 
short roots $\beta$ of $B_{e_\rho}$.

Then Proposition\autoref {prop:3.2} allows us to replace $s_\beta$ by $h_\beta^\vee s_\beta$ in the
isomorphism 
\[
\mc H (\mf s^+)^{\op} \cong \mc H (\mf s^{+\vee},q_F^{1/2}). 
\]
The representation $\rho \times \sigma_-$ does not appear directly in this framework, but it 
does so via a short detour. Pick $\chi \in \Irr (\Z^{e_{\rho'}})$ such that the values 
$\chi_i := \chi (e_i) \in \C^\times$ are in generic position, except that $\chi_{e_{\rho'}} = -1$. 
We identify $\chi_i$ with an unramified character of $\GL_{d_\rho}(F)$, unique up to 
$X_\nr (\GL_{d_\rho}(F),\rho)$. The $\mc H_{\rho'}$-representation 
$\mr{ind}_{\C [\Z^{e_{\rho'}}]}^{\mc H_{\rho'}} \C_\chi$ corresponds to 
\begin{equation}\label{eq:4.24}
\rho' \otimes \chi_1 \times \cdots \times \rho' \otimes 
\chi_{e_{\rho'}-1} \times \rho \times \sigma_- .
\end{equation}
The decomposition of this representation in irreducibles is governed by the component group 
$\mc S_\chi$ of the $L$-parameter, which in this case is just $\langle s_\beta \rangle$, acting
on the last coordinate.

Things become more transparent if we complete the centre of $\mc H_{\rho'}$ at its maximal ideal
associated to $W(B_{e_\rho}) \chi$, as in \cite[\S 7]{Lus-Gr}. That also gives a completion of 
$\mc H_{\rho'}$. Recall \cite[10.3.(a)]{Lus-Gr} that such a completion 
functor does not change the category of modules which admit the central character 
$W(B_{e_\rho}) \chi$. By \cite[Theorem 8.6]{Lus-Gr}, this completion of $\mc H_{\rho'}$ 
is Morita equivalent with the completion (also determined by $\chi$) of the simpler
extended affine Hecke algebra $\C [\Z^{e_{\rho'}}] \rtimes \mc S_\chi$. Moreover, by 
\cite[Theorem 2.1.2.c]{SolAHA} this gives rise to a natural equivalence between the category 
of $\C [\Z^{e_{\rho'}}] \rtimes \mc S_\chi$-modules with central character $\chi$ and the
category of $\mc H_{\rho'}$-modules with central character $W(B_{e_\rho})\chi$. 
Via that equivalence, our induced representation becomes
\begin{equation}\label{eq:4.23}
\mr{ind}_{\C [\Z^{e_{\rho'}}]}^{\C [\Z^{e_{\rho'}}] \rtimes \langle s_\beta \rangle} \C_\chi .
\end{equation}
Since $\chi (h_\beta^\vee) = -1$, the automorphism which exchanges $s_\beta$ and $h_\beta^\vee
s_\beta$ affects the action of $s_\beta$ on \eqref{eq:4.23} by multiplication with -1.
As a consequence the canonical intertwining operator from $s_\beta$ on \eqref{eq:4.23}, or 
equivalently on $\mr{ind}_{\C [\Z^{e_{\rho'}}]}^{\mc H_{\rho'}} \C_\chi$ or \eqref{eq:4.24}, 
is adjusted by a factor by $-1$ by the replacement $s_\beta \mapsto h_\beta^\vee s_\beta$.

The intertwining operator on \eqref{eq:4.23} associated with $s_\beta$ is induced by the
intertwining operator from $s_\beta$ on the representation
\[
\mr{ind}_{\C [\Z]}^{\C [\Z] \rtimes \langle s_\beta \rangle} \C_{-1}
\]
of the smaller algebra $\C [\Z] \rtimes \langle s_\beta \rangle$. By \cite[Theorem 8.6]{Lus-Gr} 
and \cite[Theorem 2.1.2]{SolAHA} the completion of that algebra (with respect to the central 
character $\langle s_\beta \rangle (-1) = -1$) is naturally isomorphic 
to the completion (also with respect to the central character -1) of $\mc H 
(\mf s^{'\vee},q_F^{1/2})$, where $(\tau \times \phi_-, \epsilon_{\sigma_-}) \in 
\Phi (G^+_{n_- + d_\rho} )^{\mf s^{'\vee}}$. In this way the intertwining operator for $s_\beta$
on \eqref{eq:4.23} is related to the intertwining operator $J(s_\beta,\tau \times \phi_-)$, and
multiplying the former by $-1$ entails that the latter is also multiplied by $-1$.

As $J(s_\beta, \rho \times \sigma_-)$ is a priori unique up to a factor $-1$, it follows that we can
match it with $J(s_\beta,\tau \times \phi_-)$ under the appropriate Hecke algebra isomorphism by
making the (unique) correct choice for the image of $T_{s_\beta} \in \mc H (\mf s^{+\vee},
q_F^{1/2})$ in $\mc H (\mf s^+)^{\op}$.\\

\noindent
\textbf{The case $\ell_{\tau'} = 0$.} \\
Here we need to take both $J(s_\beta, \tau \times \sigma_-)$ and $J(s_\beta, \tau' \times \sigma_-)$
into account. The relevant tensor factor of
$\mc H (\mc R_{\mf s^\vee,\der},\lambda,\lambda^*,q_F^{1/2}) \rtimes \Gamma_{\mf s^\vee}^+$ 
is of the form 
\[
\mc H_\rho = \mc H (\mc R_{D_\rho},q_F^{1/2}) \rtimes \mr{Out}(D_{e_\rho}) ,
\]
where $\mc R_{D_m} = (\Z^m, D_m, \Z^m ,D_m)$. As basepoint of the underlying torus we take
$\rho^{\boxtimes e_\rho}$.
We can modify the isomorphism $\mc H (\mf s^+)^{\op} \cong \mc H (\mf s^{+\vee},q_F^{1/2})$ in
four ways on this tensor factor. Namely, write $\mr{Out}(D_{e_\rho}) = \langle s_\beta \rangle$
with $\beta$ a short root in $B_{e_\rho} \supset D_{e_\rho}$. As the image of $s_\beta \in 
\mc H (\mf s^{+\vee},q_F^{1/2})$ in $\mc H (\mf s^+)^{\op}$ we may take $-s_\beta, h_\beta^\vee 
s_\beta, -h_\beta^\vee s_\beta$ or $s_\beta$.

Like in the previous case, we can study the representations of $\mc H_\rho$ induced from
characters $\chi$ of $\Z^{e_\rho}$. We pick the first $e_\rho -1$ coordinates of $\chi$ 
generically in $\C^\times$, and take $\chi_{e_\rho} = \pm 1$. That corresponds to
\[
\rho \otimes \chi_1 \times \cdots \times \rho \otimes 
\chi_{e_{\rho'}-1} \times \rho \otimes \pm 1 \times \sigma_- ,
\]
where $\rho \otimes 1 = \rho$ and $\rho \otimes -1 = \rho'$. We can complete $\mc H_\rho$
with respect to the maximal ideal of $Z(\mc H_\rho)$ associated to $W(B_{e_\rho}) \chi$. 
By \cite[Theorem 8.6]{Lus-Gr} that the completion is Morita equivalent to the completion 
(with respect to the central character $\langle s_\beta \rangle \chi = \chi$) of the simpler 
extended affine Hecke algebra $\C [\Z^{e_{\rho'}}] \rtimes \langle s_\beta \rangle$. Moreover 
\cite[Theorem 2.1.2.c]{SolAHA} provides a natural equivalence between the modules with central 
character $W(B_{e_\rho}) \chi$ or $\chi$ of these two algebras. Thus we transfer the issue
to modules for $\C [\Z^{e_{\rho'}}] \rtimes \langle s_\beta \rangle$.

When $\chi_{e_\rho} = 1$, the intertwining operator for $s_\beta$ on
\begin{equation}\label{eq:4.25}
\mr{ind}_{\C [\Z^{e_\rho}]}^{\C [\Z^{e_\rho}] \rtimes \langle s_\beta \rangle} \C_\chi 
\end{equation}
is induced from the intertwining operator on 
$\mr{ind}_{\C [\Z]}^{\C [\Z] \rtimes \langle s_\beta \rangle} \C_1$. The completion (with
respect to the central character $\chi$) of the algebra 
$\C [\Z] \rtimes \langle s_\beta \rangle$ is naturally isomorphic to the completion (with
respect to the central character 1, which corresponds to the basepoint $\tau$) of the affine 
Hecke algebra for the Bernstein component containing $(\tau \times \phi_-, \sigma_-)$. 
Thus the intertwining operator on
\eqref{eq:4.25} is essentially $J(s_\beta,\tau \times \phi_-)$. Via an instance of
\eqref{eq:3.16} for a Bernstein component of $\Irr (G^+_{n_- + d_\rho})$, the latter corresponds
to $\pm J (s_\beta, \rho \times \sigma_-)$. Possibly adjusting the isomorphism
$\mc H (\mf s^+)^{\op} \cong \mc H (\mf s^{+\vee},q_F^{1/2})$ so that $s_\beta$ goes to 
$-s_\beta$, we can match $J(s_\beta,\tau \times \phi_-)$ and $J (s_\beta, \rho \times \sigma_-)$.

When $\chi_{e_\rho} = -1$, the situation is similar, but now \eqref{eq:4.25} is induced from\\
$\mr{ind}_{\C [\Z]}^{\C [\Z] \rtimes \langle s_\beta \rangle} \C_{-1}$, which comes from
$(\tau' \times \phi_-,\epsilon_{\sigma_-})$. Here the intertwining operator from $s_\beta$
on \eqref{eq:4.25} is essentially $J(s_\beta, \tau' \times \phi_-)$. Via the same instance
of \eqref{eq:3.16} as above, this operator corresponds to $\pm J(s_\beta, \rho' \times \sigma_-)$.
We can still adjust the isomorphism $\mc H (\mf s^+)^{\op} \cong \mc H (\mf s^{+\vee},q_F^{1/2})$
by composition with the automorphism that sends $s_\beta$ to $h_\beta^\vee s_\beta$.
That multiplies $J(s_\beta,\tau' \times \phi_-)$ with -1, while perserving
$J(s_\beta,\tau \times \phi_-)$. Thus, by a suitable choice we can arrange that 
$J(s_\beta, \tau' \times \phi_-)$ corresponds to $J(s_\beta, \rho' \times \sigma_-)$,
without disturbing the previous normalization.
In total we have a unique choice (out of four) for the image of $s_\beta$ under the algebra 
isomorphism, such that both relevant pairs of intertwining operators match up. \\

With the above choices, for $\ell_{\tau'} \geq 0$ and for all relevant $\rho$, we managed to 
fulfill the conditions imposed in the statement. To this end we exploited the freedom provided 
by the operations (iii) and (iv) in Theorem\autoref {thm:1.7} and Proposition\autoref {prop:3.2}, and 
we fixed choices for all instances of (iii) and (iv). These are canonical because the 
intertwining operators \eqref{eq:4.26} are determined by the Whittaker datum for the quasi-split 
inner form of $G$, which is part of our input. In view of Proposition\autoref {prop:3.2}, this 
renders our Hecke algebra isomorphism canonical up to conjugation by elements of 
$\mc O (T_{\mf s^+})^\times$.
\end{proof}

Applying Proposition\autoref {prop:4.4}, we can match many more intertwining operators between
$G^+$-representations with intertwining operators between 
$\mc H (\mf s^{+\vee},q_F^{1/2})$-modules.

\begin{lem}\label{lem:4.7}
Choose an isomorphism $\mc H (\mf s^+)^{\op} \cong \mc H (\mf s^{+\vee},q_F^{1/2})$ as 
in Proposition\autoref {prop:4.4}. For every discrete series 
representation $\pi \in \Irr (G^+)^{\mf s^+}$ and every $(\rho,a_-) \in \mr{Jord}(\pi)$ with 
$a_-$ minimal and odd and $\mr{Jord}_\rho (\Sc (\pi))$ empty, the intertwining operators 
$N(\rho,a_-)$ and $N(\tau,a_-)$ from \eqref{eq:4.4} and \eqref{eq:4.5} coincide on the
representation\\ $\delta (\rho,(a-1)/2,(1-a_-)/2) \times \tilde \pi$ from \eqref{eq:4.22}.
\end{lem}
\begin{proof}
After \eqref{eq:4.26} we described how $J(s_\beta,\rho \times \sigma_-)$ determines
$N(\rho,a_-)$. By Proposition\autoref {prop:4.4}, $J(s_\beta,\rho \times \sigma_-)$ corresponds to
$J(s_\beta, \tau \times \phi_-)$, so it remains to check that the latter determines
$N(\tau,a_-)$ in the same way.

The constructions around \eqref{eq:4.22} and \eqref{eq:4.21} work analogously for modules
of Hecke algebras, which reduces our task to comparing 
\begin{equation}\label{eq:4.29}
J(w_e, \rho \nu^{b_1} \times \cdots \times \rho \nu^{b_{(e-1)/2}} \times \rho \times
\rho \nu^{-b_{(e-1)/2}} \times \cdots \times\rho \nu^{-b_1} \times \sigma_-) 
\end{equation}
from \eqref{eq:4.20} with its version for the appropriate Hecke algebra $\mc H (\mf s^{'\vee},
q_F^{1/2})$. Recall that $\nu^b \in X_\nr (\GL_{d_\rho}(F))$ corresponds to the central element 
$q_F^b \in\GL_{d_\rho}(\C)$, and then $\rho \nu^b$ corresponds to $q_F^b \tau$. In the geometric 
setup from \cite{AMS2}, given $w_e$ there is a canonical intertwining operator 
\begin{equation}\label{eq:4.27}
J \Big( w_e, \mr{ind}_{\mc O (T_{\mf s^{'\vee}}) }^{\mc H (\mf s^{'\vee},q_F^{1/2})} 
\pi \big( q_F^{b_1} \tau \times \cdots \times q_F^{b_{(e-1)/2}} \tau \times
\tau \times \cdots \times q_F^{-b_1} \tau \times \phi_-,\epsilon_{\sigma_-} \big) \Big) ,
\end{equation}
from the indicated module to itself. (Here the symbols $\times$ refer to an $L$-parameter with 
values in a direct product of groups, not to parabolic induction.) This operator has order 2,
and it comes as a member of an algebraic family parametrized by $b_i \in \R$.
When all $b_i$ are equal to 0, the permutation $(1 \; e) (2 \; e\!-\!1) \cdots 
((e\!-\!1)/2 \,\, (e\!+\!3)/2)$ lies in the connected component of the centralizer group of the 
$L$-parameter in \eqref{eq:4.27}, and the canonical intertwining operator associated to that 
permutation is just the identity. Hence for $b_i = 0$ the operator \eqref{eq:4.27} reduces to 
\begin{equation}\label{eq:4.28}
J \Big( s_\beta \times \cdots \times s_\beta, 
\mr{ind}_{\mc O (T_{\mf s^{'\vee}})}^{\mc H (\mf s^{'\vee},q_F^{1/2})} 
\pi \big( \tau \times \cdots \times \tau \times \phi_-,\epsilon_{\sigma_-} \big) \Big) .
\end{equation}
This operator is induced by $J(s_\beta, \tau \times \phi_- )$ on each of the $e = e_\rho$ 
coordinates, in the following sense. Upon completion of $\mc H (\mf s^{'\vee},q_F^{1/2})$ with
respect to the central character associated to $\tau \times \cdots \times \tau \times \phi_-$
(like in the proof of Proposition\autoref {prop:4.4}) 
we obtain an $e$-fold tensor product of modules 
\[
\mr{ind}_{\C [\Z]}^{\C [\Z] \rtimes \langle s_\beta \rangle} \pi (\tau \times \phi_-) .
\]
Then \eqref{eq:4.28} can be identified with the $e$-fold tensor product of the operators
$J(s_\beta,\tau \times \phi_-)$ on these modules. This is the same procedure as in 
\eqref{eq:4.18}, so Proposition\autoref {prop:4.4} guarantees that \eqref{eq:4.28} agrees with 
\eqref{eq:4.18} for $b_i = 0$ and the correct number of factors. Since all instances of 
\eqref{eq:4.27} square to the identity and they are part of a continuous family, all these
instances are fixed when we know \eqref{eq:4.28}. That is completely analogous to the
situation in \eqref{eq:4.20}. Therefore \eqref{eq:4.29} and \eqref{eq:4.27} agree via
a Hecke algebra isomorphism as in Proposition\autoref {prop:4.4}. 
\end{proof}

After all these preparations, we are ready to compare the two parametrizations of arbitrary
discrete series representations of $G^+$.

\begin{prop}\label{prop:4.5}
Choose a Hecke algebra isomorphism $\mc H (\mf s^+)^{\op} \cong \mc H (\mf s^{+\vee},q_F^{1/2})$ 
as in Proposition\autoref {prop:4.4}. 
Let $(\phi,\epsilon) \in \Phi_\enh (G^+)^{\mf s^{+\vee}}$ be discrete and let $\pi \in 
\Irr (G^+)$ be associated with $(\phi,\epsilon)$ by Theorems\autoref {thm:1.1} and\autoref {thm:1.3}.
Then $\pi (\phi,\epsilon)$ is isomorphic to $\pi$.
\end{prop}
\begin{proof}
Write $\phi$ as $z_\phi \phi_b$ with $z_\phi \in \rZ(G^\vee)$ and $\phi_b \in \Phi (G)$ bounded
and discrete. Let $\chi_\phi \in X_\nr (G^+)$ correspond to $z_\phi$. By Theorem\autoref {thm:1.3}
$\pi = \chi_\phi \otimes \pi_b$ where $\pi_b$ corresponds to $(\phi_b,\epsilon)$. On the other
hand Theorem\autoref {thm:3.4}.d says that
\[
\pi (\phi,\epsilon) = \pi (z_\phi \phi_b,\epsilon) = \chi_\phi \otimes \pi (\phi_b,\epsilon) .
\]
Therefore it suffices to prove the proposition under the additional assumption that
$\phi$ is bounded.

Applying Proposition\autoref {prop:4.2} repeatedly, we find that $\pi (\phi,\epsilon)$ embeds in  
\begin{equation}\label{eq:4.7}
\prod\nolimits_{\rho,a,a'} \delta (\rho,(a-1)/2,(1-a')/2) \times \pi (\tilde \phi, \tilde \epsilon) ,
\end{equation}
where $\mr{Jord}(\tilde \phi) \subset \mr{Jord}(\phi)$ and $\tilde \epsilon = 
\epsilon |_{\mc S_{\tilde \phi}^+}$ is alternated in the sense of \eqref{eq:4.14}. Here the product
runs over some triples with $\epsilon (z_{\rho,a}) = \epsilon (z_{\rho,a'})$, not necessarily
all such triples. Similarly, by \cite[\S 5]{Moe0} $\pi$ embeds in
\begin{equation}\label{eq:4.8}
\prod\nolimits_{\rho,a,a'} \delta (\rho,(a-1)/2,(1-a')/2) \times \tilde \pi
\end{equation}
with $\mr{Jord}(\tilde \pi) \subset \mr{Jord}(\pi)$ and $\epsilon_{\tilde \pi} = 
\epsilon |_{\mc S_{\tilde \pi}^+}$ alternated. By Lemma\autoref {lem:4.3} both $\pi (\tilde \phi,
\tilde \epsilon)$ and $\tilde \pi$ are completely positive discrete series representations. 
Further $\tilde \pi$ and $\pi (\tilde \phi,\tilde \epsilon)$ have the same Jordan blocks, because 
both are obtained from $\mr{Jord}(\pi) = \mr{Jord}(\pi (\phi,\epsilon))$ by removing the pairs 
$(\rho,a),(\rho,a')$ that appear in the product. By Theorem\autoref {thm:3.4}.a and Proposition 
\autoref{prop:4.1}, $\pi (\phi,\epsilon)$ and $\pi$ have the same supercuspidal support. From 
\eqref{eq:4.7} and \eqref{eq:4.8} we see that $\pi (\tilde \phi,\tilde \epsilon)$ and $\tilde \pi$ 
also have the same supercuspidal support. With Corollary\autoref {cor:4.11} we deduce that 
$\tilde \pi \cong \pi (\tilde \phi, \tilde \epsilon)$. 

Thus both $\pi (\phi,\epsilon)$ and $\pi$ are subrepresentations of \eqref{eq:4.7}, which is 
isomorphic to \eqref{eq:4.8}. By Theorem\autoref {thm:3.6}, \eqref{eq:4.7} is a direct sum of precisely 
$[\mc S_\phi : \mc S_{\tilde \phi}]^{1/2}$ subrepresentations, which are mutually inequivalent. 
Every factor $\delta (\rho, (a-1)/2,(1-a')/2)$ doubles the number of constituents, because 
\begin{equation}\label{eq:4.9}
\delta (\rho, (a-1)/2,(1-a')/2) \times \pi (\tilde \phi, \tilde \epsilon)
\end{equation}
has length two. We can distinguish three classes of $\rho$'s:

\textbf{Case (i).}
When $\mr{Jord}_\rho (\tilde \pi )$ is nonempty, Proposition\autoref {prop:4.2}.a determines which
summands must be picked to get $\pi$. (This works also for M\oe glin's parametrization, by 
\cite[\S 5]{Moe0}.) Namely, start with $(\rho,b) \in \mr{Jord}(\tilde \pi)$ and an adjacent 
$(\rho,a) \in \mr{Jord}(\pi) \setminus \mr{Jord}(\tilde \pi)$. Then Proposition\autoref {prop:4.2}.a
imposes a condition (recall that $\epsilon$ was given). Next, take $(\rho,a') \in \mr{Jord}(\pi)
\setminus \mr{Jord}(\tilde \pi)$ adjacent to $(\rho,a)$. Of the two choices for a subrepresentation 
of \eqref{eq:4.9}, one fulfills the previous condition and one does not (that is another consequence 
of Proposition\autoref {prop:4.2}.a). Proceeding in this way, now with 
$\tilde \phi \setminus \{ (\rho,a),(\rho,a') \}$ in
the role of $\tilde \phi$, we discover step by step how to pick the right constituent of 
\begin{equation}\label{eq:4.16}
\delta (\rho, (a''-1)/2,(1-a''')/2) \times \pi (\tilde \phi, \tilde \epsilon)
\end{equation}
for other $a'',a''' \in \mr{Jord}(\pi) \setminus \mr{Jord}(\tilde \pi)$ as well.

\noindent
\textbf{Case (ii).}
Suppose that $\mr{Jord}_\rho (\tilde \pi)$ is empty and that $\mr{Jord}_\rho (\pi)$ consists of
even numbers. In this case we may take $a' = 0$, set $\epsilon (\rho,0) = 1$ and use Proposition 
\autoref{prop:4.2}.b. As in the previous case, $\epsilon$ determines which constituents of \eqref{eq:4.9}
and \eqref{eq:4.10} must chosen to enable an embedding of $\pi$.

\noindent
\textbf{Case (iii).}
Suppose  that $\mr{Jord}_\rho (\tilde \pi)$ is empty and that $\mr{Jord}_\rho (\pi)$ consists 
of odd numbers. By Proposition\autoref {prop:4.4} and Lemma\autoref {lem:4.7}, our two parametrizations 
involve the same constituent of $\delta (\rho, (b-1)/2,(1-a_-)/2) \times \pi (\tilde \phi, 
\tilde \epsilon)$, where $b$ is the smallest $a \in \mr{Jord}_\rho (\pi) \setminus \{a_-\}$ such 
that $\epsilon (z_{\rho,a}) = \epsilon (z_{\rho,a_-})$. Once we know that, the method from the 
previous cases tells us which constituent of \eqref{eq:4.16} we have to take, for any adjacent
$a'',a''' \in \mr{Jord}_\rho (\pi)$.

Hence $\pi$ and $\pi (\phi,\epsilon)$ are obtained from \eqref{eq:4.7} by taking the same 
constituents of \eqref{eq:4.16} in all cases, so $\pi \cong \pi (\phi,\epsilon)$.
\end{proof}

\subsection{Tempered representations} \

Consider a bounded $L$-parameter $\phi \in \Phi (G)$. Recall from \eqref{eq:1.2} and \eqref{eq:1.3} 
that we can decompose $(\phi,\C^{2n})$ as 
\begin{equation}\label{eq:4.30}
\bigoplus\nolimits_{\psi \in I^\pm} N_\psi \otimes V_\psi \oplus 
\bigoplus\nolimits_{\psi \in I^0} N_\psi \otimes (V_\psi \oplus V_\psi^\vee) ,
\end{equation}
where $N_\psi$ is a multiplicity space and $V_\psi^\vee$ is endowed with the representation
$\psi^\vee \otimes \mu_G^\vee \circ \phi$. There exists a Levi subgroup $L$ of $G$, unique up to
conjugation, such that $\phi$ factors through $\Phi (L)$ and defines a discrete $L$-parameter for $L$.
Every factor $\GL_m (F)$ of $L$ appears in $G$ as
\[
\{ (A,B) \in \GL_m (F) \times \GL_m (F) : B = J A^{-T} J^{-1} \} .
\]
The same goes for $L^\vee$ and $G^\vee$. 
Hence every $\psi \in I^\pm$ which appears with multiplicity $\mu$ in $\phi |_{\GL_m (\C)}$,
accounts for multiplicity $2 \mu$ in \eqref{eq:4.30}. In view of \eqref{eq:1.5}, the part of 
$\phi$ with image in the factor $G_{n_-}^\vee$ of $L^\vee$ is precisely 
$\prod_{\psi \in I^+ : \dim N_\psi \mr{ odd}} \psi$, while the part of $\phi$ in the type
GL factors of $L^\vee$ is
\[
\bigoplus\nolimits_{\psi \in I^\pm} \lfloor \dim (N_\psi) / 2 \rfloor \psi \; \oplus \;
\bigoplus\nolimits_{\psi \in I^0} \dim (N_\psi) \psi .
\]
For $\psi \in I^0$ this involves a choice of $\psi$ or $\psi^\vee \otimes \mu_G^\vee \circ \phi$,
but that hardly matters because both will appear equally often when we pass to $G^\vee$. For
the component groups of $\phi$ it is a bit easier to work with $G^+$ and $L^+$, so we consider
$\phi$ as element of $\Phi (G^+)$ and as $\phi_L \in \Phi (L^+)$. By these we mean just 
$\Phi (G)$ and $\Phi (L)$, only with component groups of $\phi$ or $\phi_L$ computed in
$G^{\vee +}$ or $L^{\vee +}$. In the description of $\mc S_\phi$ 
following \eqref{eq:1.3}, passing to $G^+$ replaces $\mc S_\phi$ by $\mc S_\phi^+$, which
means that we forget the determinant condition ``S'' on $Z_{{G^\vee}_\der}(\phi)$. Thus
\begin{align*}
& \rZ_{{L^{+ \vee}}_\der } (\phi_L) \cong \prod_{\psi \in I^+ : \dim N_\psi \, \mr{odd}} 
\rO_1 (\C) \times \prod_{\psi \in I^\pm} 
\GL_{\lfloor \dim (N_\psi)/2 \rfloor}(\C) \times \prod_{\psi \in I^0} \GL (N_\psi) , \\
& \mc S_{\phi_L}^+ \cong \prod_{\psi \in I^+ : \dim N_\psi \, \mr{odd}} \langle z_\psi \rangle , \\
& \mc S_\phi^+ = \prod_{\psi \in I^+ : N_\psi \neq 0} \langle z_\psi \rangle \; = \;
\mc S_{\phi_L}^+ \times \prod_{\psi \in I^+ : \dim (N_\psi) \in 2 \Z_{>0}} \langle z_\psi \rangle 
\; =: \; \mc S_{\phi_L}^+ \times \mc S^+_{\phi / \phi_L} .
\end{align*}
Let us fix $\epsilon_L \in \Irr (\mc S^+_{\phi_L})$ such that $(\phi_L,\epsilon_L)$ belongs to
the image in of the parametrization map in Theorem\autoref {thm:3.4} for $L^+$. It gives a discrete 
series representation $\pi (\phi_L, \epsilon_L) \in \Irr (L^+)$, which by Proposition\autoref {prop:4.5} 
is the same for the endoscopic method as for the Hecke algebra method. By Theorem\autoref {thm:3.6},
$I_{L^+ U}^{G^+} \pi (\phi_L,\epsilon_L)$ has precisely $| \mc S^+_{\phi / \phi_L} |$ irreducible 
direct summands, which are mutually inequivalent and indexed by
\[
\big\{ \epsilon \in \Irr (\mc S_\phi^+) : \epsilon |_{\mc S_{\phi_L}^+} = \epsilon_L \big\} 
\cong \Irr (\mc S^+_{\phi / \phi_L}) .
\]
The same conclusion was obtained in \cite[Theorem 13.1]{MoTa}. One part of the constructions 
behind Theorem\autoref {thm:3.4} in \cite{AMS2,AMS3} is
\begin{equation}\label{eq:4.33}
\pi (\phi,\epsilon) = \Hom_{\mc S^+_{\phi / \phi_L}} 
\big( \epsilon |_{\mc S^+_{\phi / \phi_L}},  I_{L^+ U}^{G^+} \pi (\phi_L,\epsilon_L) \big) .
\end{equation}
This goes back to \cite[(1.17)]{AMS3}, and from there it is transfered via various Hecke
algebras in \cite{AMS3} to an analogue for $\mc H (\mf s^\vee,q_F^{1/2})$, see \eqref{eq:3.56}.
Using the equivalences of categories \eqref{eq:3.13} and \eqref{eq:3.12}, one arrives at 
\eqref{eq:4.33}. Here the action of $\mc S^+_{\phi / \phi_L}$ comes from intertwining operators 
\[
N (z_\psi, \phi_L, \epsilon_L) \in \mr{End}_{G^+} \big( I_{L^+ U}^{G^+} \pi (\phi_L,\epsilon_L) \big),
\]
one for each generator $z_\psi$ of $\mc S^+_{\phi / \phi_L}$.

On the other hand, an irreducible tempered $G^+$-representation $\pi (\phi )_\epsilon$ is 
constructed with endoscopy in \cite[\S 3.6]{MoRe}, and it is checked that
$I_{L^+ U}^{G^+} \pi (\phi_L, \epsilon_L)$ (called $\sigma$ in \cite{MoRe}) decomposes as
\[
\bigoplus\nolimits_{\epsilon \in \Irr (\mc S_\phi^+) \,:\, \epsilon |_{\mc S_{\phi_L}^+} = 
\epsilon_L } \pi (\phi )_\epsilon .
\]
This decomposition can be achieved with suitable intertwining operators that make 
$\mc S^+_{\phi / \phi_L}$ act on $I_{L^+ U}^{G^+} \pi (\phi_L, \epsilon_L)$ and are normalized
in a way that is compatible with the endoscopic methods in \cite{MoRe}. The appropriate 
normalization stems from \cite[\S 2.3]{Art} and involves $L$-functions and $\epsilon$-factors.
Unfortunately, it becomes untractable in the setting of Hecke algebras. Nevertheless, we can
say more concretely that, for every $\psi \in I^+$ with $\dim N_\psi \in 2 \Z_{>0}$ and
$\pi (\psi) = \delta (\rho,a)$, there is a normalized intertwining operator
\[
N \big( z_{\rho,a}, \pi (\phi_L,\epsilon_L) \big) \in 
\mr{End}_{G^+} \big( I_{L^+ U}^{G^+} \pi (\phi_L,\epsilon_L) \big) 
\]
which squares to the identity. From \cite[\S 2]{MoRe} we see that 
\begin{multline}\label{eq:4.34}
\pi (\phi)_\epsilon = \big( \epsilon |_{\mc S^+_{\phi / \phi_L}} \otimes I_{L^+ U}^{G^+} 
\pi (\phi_L,\epsilon_L) \big)^{\mc S^+_{\phi / \phi_L}} = \\
\{ \text{fixed points of the operators } \epsilon (z_\psi)
N(z_\psi, \pi (\phi_L,\epsilon_L)) \text{ with } \dim N_\psi \in 2 \Z_{>0} \} .
\end{multline}

\begin{lem}\label{lem:4.8}
Pick an inertial equivalence class $\mf s^+$ for $G^+$ and choose an isomorphism
$\mc H (\mf s^+ )^{\op} \cong \mc H (\mf s^{+\vee},q_F^{1/2})$ as in Proposition\autoref {prop:4.4}.

For every bounded $(\phi, \epsilon) \in \Phi_\enh (G^+ )^{\mf s^{+\vee}}$ and every
$\psi \in I^+ \cap \mr{Jord}(\phi)$ with $\pi (\psi) = (\rho,a)$, the intertwining operators 
\[
N \big( z_{\rho,a},\pi (\phi_L,\epsilon_L) \big) \quad \text{and} \quad
N ( z_\psi, \phi_L,\epsilon_L) 
\]
agree via the Hecke algebra isomorphism.
\end{lem}
\begin{proof}
We need to distinguish a few cases.

First we suppose that $\dim N_\psi$ is odd. Then $\psi$ appears in the factor $G_{n_-}^{+\vee}$ of 
$L^+$, and the two intertwining operators of $G^+$-representations under consideration are 
induced by the analogous intertwining operators of $G^+_{n_-}$-representations. The latter
two agree by Lemma\autoref {lem:4.7}.

Now we suppose $\dim N_\psi$ that is even and that $\psi = \tau \otimes P_a$ with 
$\tau \in \Irr (\mb W_F )_\phi^+$. Here $a$ is odd because $\psi \in I^+$. The same arguments
as for Lemma\autoref {lem:4.7} show that $N \big( z_{\rho,a},\pi (\phi_L,\epsilon_L) \big)$
and $N ( z_\psi, \phi_L,\epsilon_L)$ agree, because $N(s_\beta, \rho \times \sigma_-)$ and
$N(s_\beta, \tau \times \phi_-)$ agree by Proposition\autoref {prop:4.4}.

Finally we suppose that $\dim N_\psi$ is even and that $\psi = \tau \otimes P_a$ with 
$\tau \in \Irr (\mb W_F )_\phi^-$. Now $a$ is even because $\psi \in I^+$. In this case we
do not know whether $N(s_\beta, \rho \times \sigma_-)$ and $N(s_\beta, \tau \times \phi_-)$
match via the Hecke algebra isomorphism. But both are unique up to scalars and square to the
identity, so the agree up to a factor $\pm 1$. 

Write $e = a d_\rho$. Motivated by \eqref{eq:4.21}, we want to compare the operators
\begin{equation}\label{eq:4.31}
N \big( w_e, \mr{Ind} \, \Sc (\delta (\rho,a)) \times \sigma_- \big) \quad \text{and} \quad
N (w_e, \Sc (\tau \otimes P_a) \times \phi_-) ,
\end{equation}
where the right hand side is an abbreviation of \eqref{eq:4.27}. From the remarks 
after \eqref{eq:4.20} we know that the former is determined (via a continuous deformation)
by the intertwining operator 
\begin{equation}\label{eq:4.32}
N( (s_\beta \times \cdots \times s_\beta), \rho \times \cdots \times \rho \times \sigma_-) ,
\end{equation}
where $s_\beta$ and $\rho$ both appear $a$ times. For each such factor $\rho$, we get a 
contribution which is induced by $N(s_\beta, \rho \times \sigma_-)$.

Similarly, in the proof of Lemma\autoref {lem:4.7} we saw that $N (w_e, \Sc (\tau \otimes P_a) 
\times \phi_-)$ is determined in the same way by \eqref{eq:4.28} and 
$N(s_\beta, \tau \times \phi_-)$. Writing 
\[
N(s_\beta, \rho \times \sigma_-) = \pm N(s_\beta, \tau \times \phi_-) ,
\]
it follows that, via the appropriate Hecke algebra 
isomorphism, \eqref{eq:4.32} and \eqref{eq:4.28} agree up to a factor $(\pm 1)^a$. Since 
$a$ is even they really agree, and so do the two sides of \eqref{eq:4.31}. We note that 
\[
N \big( z_{\rho,a},\pi (\phi_L,\epsilon_L) \big) \quad \text{and} \quad
N ( z_\psi, \phi_L,\epsilon_L) 
\]
are induced by \eqref{eq:4.31}, on both sides in the same way as in \eqref{eq:4.21}, so with the
identity on factors not involved in \eqref{eq:4.31}. We combine that with the above analysis of
\eqref{eq:4.31} to establish the lemma in this case.
\end{proof}

From Proposition\autoref {prop:4.5}, Lemma\autoref {lem:4.8}, \eqref{eq:4.33} and \eqref{eq:4.34} we conclude:

\begin{cor}\label{cor:4.9}
In the setting of Proposition\autoref {prop:4.4}, let $(\phi,\epsilon) \in \Phi_\enh (G^+ )^{\mf s^+}$ be
bounded. Then $\pi (\phi,\epsilon) \in \Irr (G^+)$ from Theorem\autoref {thm:3.4} is isomorphic with 
the tempered representation $\pi (\phi )_\epsilon$ from \cite{MoRe}.
\end{cor}

Together with Theorem\autoref {thm:3.4}.d, Corollary\autoref {cor:4.9} implies that 
\begin{equation}\label{eq:4.35}
\pi (\hat \chi \phi)_\epsilon \cong \chi \otimes \pi (\phi )_\epsilon 
\end{equation}
for all bounded $(\phi,\epsilon) \in \Phi_\enh (G^+)$ and all unitary $\chi \in X_\nr (G^+ )$.

\subsection{Irreducible smooth representations} \

With the Langlands classification \cite[Th\'eor\`eme VII.4.2]{Ren} one can 
construct and parametrize all irreducible smooth representations of a reductive $p$-adic group 
in terms of the irreducible tempered representations of its Levi subgroups. Although the
Langlands classification usually applies to connected reductive groups, it can be extended
to some disconnected reductive groups \cite[Theorem 4.2 and Remark 4.2]{BaJa}. For our $G^+$ 
with the Levi subgroups $L^+ = \rZ_{\mc G^+}(\rZ (\mc L)^\circ) (F)$, the formulation is the 
almost same as for connected reductive groups, at least when $[G^+ : G] = [L^+ : L]$ and we
limit ourselves to $G^+$-representations from standard modules that are parabolically induced
from $L^+$-representations. Indeed, from \cite[Introduction]{BaJa} one
sees that in those cases the Langlands data are triples $(P^+,\tau,\nu)$ where $P^+ = L^+ U_P$ 
is a parabolic subgroup of $G^+$, $\tau \in \Irr (L^+)$ is tempered and $\nu \in \Hom (L^+,\R_{>0})$ 
is in strictly positive position with respect to $P^+$. As usual $I_{P^+}^{G^+} (\tau \otimes \nu)$
has a unique irreducible quotient $\pi (P^+,\tau,\nu)$, and that sets up a bijection 
between the $G^+$-conjugacy classes of such Langlands data and the associated subset of $\Irr (G^+)$.

There also exists a Langlands classification for (enhanced) $L$-parameters \cite{SiZi}, which is
similar. For every $\phi \in \Phi (G)$ there exists a parabolic subgroup $P^\vee = L^\vee U_P^\vee$
of $G^\vee$, a bounded $\phi_b \in \Phi (L)$ and $\hat \nu \in H^1 (\mb W_F/\mb I_F,Z(L^\vee))$,
strictly positive with respect to $P^\vee$, such that $\phi = \phi_v \hat \nu$ in $\Phi (G)$. 
Moreover, by \cite[Theorem 4.6]{SiZi} this yields a bijection between $\Phi (G)$ and the
$G^\vee$-conjugacy classes of such triples $(P^\vee,\phi_b,\hat \nu)$. Of course we can pass to
a bijection between $G^{\vee+}$-conjugacy classes in $\Phi (G)$ and of triples $(P^\vee,\phi_b,\hat \nu)$.
Further, by \cite[Proposition 7.3]{SiZi} 
\[
\mc S_\phi \cong \mc S_{\phi_b \hat \nu} = \mc S_{\phi_b},
\]
where $\mc S_{\phi_b \hat \nu}$ and $\mc S_{\phi_b}$ are computed in $L^{\vee+}$. By the same argument,
this extends to an isomorphism 
\[
\mc S_\phi^+ \cong \mc S^+_{\phi_b \hat \nu} \quad \text{when } [G^+ : G] = [L^+ : L] .
\]
Then the Langlands classification from \cite{SiZi} sends $(\phi,\epsilon) \in \Phi_\enh (G)$ to
$(P^\vee, \phi_b, \hat \nu, \epsilon)$ with $(\phi_b,\epsilon) \in \Phi_\enh (L^+)$.

In the exceptional cases where $L^+ = L$ but $G^+ \neq G$, this does not work so nicely. Then 
$I_P^G (\tau \otimes \nu)$ is indecomposable, but $I_P^{G^+}(\tau \otimes \nu)$ is a direct sum of
one or two indecomposable $G^+$-representations. For Langlands parameters, the analogous issue is that
$\mc S_{\phi_b}^+ = \mc S_{\phi_b}$ has index one or two in $\mc S_\phi^+$. To handle those cases in 
similar style, we resort to the method from \cite{ABPS0}.

\begin{thm}\label{thm:4.10}
With the above versions of the Langlands classification, one can canonically extend the parametrization 
of irreducible tempered $G^+$-representations from \cite{MoTa,MoRe} to a parametrization of $\Irr (G^+)$.
\end{thm}
\begin{proof}
\textbf{Case (i): $[G^+:G] = [L^+:L]$.}\\
The above Langlands classifications show us what to do: to $(\phi_b \hat \nu,\epsilon)$ we associate
the unique irreducible quotient of $I_{P^+}^{G^+}( \pi_{L^+} (\phi_b)_\epsilon \otimes \nu)$.

From now on we suppose that $L^+ = L$ but $G^+ \neq G$. Consider Langlands data $(P^\vee, \phi_b,
\nu, \epsilon)$ for $\Phi (G)$. Let $Y$ be the variety of $\chi \in X_\nr (L)$ such that 
$\mc S_{\phi_b \hat \chi}^+ \cong \mc S_{\phi_b \hat \nu}^+$. It is a coset of a complex algebraic 
subtorus of $X_\nr (L)$, with finitely many cosets of subtori of smaller dimension removed, see
\cite[\S 3]{ABPS0}. It intersects the maximal compact subgroup of $X_\nr (L)$, so we can pick a 
unitary $\chi \in Y$.

\noindent
\textbf{Case (ii): $\mc S_{\phi_b \hat \nu}^+ = \mc S_{\phi_b \hat \nu}$ 
for $\phi_b \hat \nu$ as element of $\Phi (G)$.}\\
By Corollary\autoref {cor:4.9} we may use Theorem\autoref {thm:3.4} for $(\phi_b \hat \chi, \epsilon)$.
Comparing Theorem\autoref {thm:3.4} for $G$ and $G^+$, we see that the condition of this case implies
that $\ind_G^{G^+}$ preserves the irreducibility of $\pi_G (\phi_b \hat \chi )_\epsilon$. By
Theorem\autoref {thm:3.6} and \eqref{eq:4.35} for $L$:
\[
\pi (\phi_b \hat \nu )_\epsilon = \ind_G^{G^+} \pi_G (\phi_b \hat \chi )_\epsilon =
\ind_G^{G^+} I_P^G (\pi_L (\phi_b \hat \chi)_\epsilon) = 
I_P^{G^+} (\pi_L (\phi_b )_\epsilon \otimes \chi) = I_P^{G^+} \pi (\phi_b \hat \chi, \epsilon) .
\]
Now $\pi (\phi_b \hat \chi,\epsilon) = \pi (\phi_b,\epsilon) \otimes \chi$ is irreducible for any 
$\chi \in Y$ (not necessarily unitary) and Theorem\autoref {thm:3.4} for $L$ shows that
$I_P^{G^+} (\pi (\phi_b)_\epsilon) \otimes \chi) = \pi_{st}(\phi_b \hat \chi,\epsilon)$ has a unique
irreducible quotient. That means that we are back in the setting of the Langlands classification.
As in case (i), it imposes that to $(\phi_b \hat \nu, \epsilon)$ one must associated the irreducible
quotient of $I_P^{G^+} (\pi (\phi_b)_\epsilon) \otimes \nu)$. 

\noindent
\textbf{Case (iii): $\mc S_{\phi_b \hat \nu}^+ \neq \mc S_{\phi_b \hat \nu}$ for 
$\phi_b \hat \nu$ as element of $\Phi (G)$.}\\
Let $\epsilon_G$ be the restriction of $\epsilon$ to $\mc S_{\phi_b \hat \nu}$ and write
$\ind_{\mc S_{\phi_b \hat \nu}}^{\mc S_{\phi_b \hat \nu}^+} \epsilon_G = \epsilon \oplus \epsilon'$.
Theorem\autoref {thm:3.4} for $G^+$ and $G$, in combination with the condition of the case, show that
$\Res_G^{G^+} \pi (\phi_b \hat \chi)_\epsilon$ is irreducible and equal to 
$\pi_G (\phi_b \hat \chi)_{\epsilon_G}$. By Theorem\autoref {thm:3.6}
\[
I_P^{G^+} (\pi_L (\phi_b)_{\epsilon_G} \otimes \chi) = \ind_G^{G^+} I_P^G \pi_L (\phi_b \hat \chi
)_{\epsilon_G} = \ind_G^{G^+} \pi_G (\phi_b \hat \chi)_{\epsilon_G} =
\pi (\phi_b \hat \chi )_\epsilon \oplus \pi (\phi_b \hat \chi )_{\epsilon'} .
\]
For arbitrary $\chi \in Y$, Theorems\autoref {thm:3.4} and\autoref {thm:3.6} entail that 
\[
I_P^{G^+} (\pi_L (\phi_b)_{\epsilon_G} \otimes \chi) =
\pi_{st} (\phi_b \hat \chi )_\epsilon \oplus \pi_{st} (\phi_b \hat \chi )_{\epsilon'} ,
\]
where both standard $G^+$-representations $\pi_{st} (\phi_b \hat \chi )_{\epsilon^{(')}}$
have a unique irreducible quotient. The Langlands classification decrees that to $(\phi_b \hat \nu,
\epsilon )$ we must associate one of the irreducible quotients of $I_P^{G^+} 
(\pi_L (\phi_b)_{\epsilon_G} \otimes \chi)$. But only $\pi_{st} (\phi_b \hat \chi )_\epsilon$ fits 
in a conti\-nuous family of $G^+$-representations $\chi \mapsto \pi_{st} (\phi_b \hat \chi )_\epsilon$
which for unitary $\chi \in Y$ recovers $\pi (\phi_b \hat \chi )_\epsilon$, so we must send 
$(\phi_b \hat \nu,\epsilon)$ to the irreducible quotient of that standard representation. 
\end{proof}

We next result collects the conclusions from Section~\ref{sec:Langlands}.

\begin{thm}\label{thm:4.12}
Let $\mf s^+$ be an inertial equivalence class for $G^+$ and fix a Whittaker datum for the 
quasi-split inner form of $G$. There exists an algebra isomorphism
$\mc H (\mf s^+)^{\op} \cong \mc H (\mf s^{+\vee}, q_F^{1/2})$, canonical up to conjugation by
elements of $\mc O (T_{\mf s^+})^\times$, such that the following holds. 

For each $(\phi,\epsilon) \in \Phi_\enh (G^+ )^{\mf s^{+\vee}}$, the $G^+$-representation 
$\pi (\phi,\epsilon)$ constructed via Hecke algebras in Theorem\autoref {thm:3.4} is isomorphic to the 
$G^+$-representation associated to $(\phi,\epsilon)$ by \cite{MoRe} and the Langlands classification. 
\end{thm}
\begin{proof}
As before, the Hecke algebra isomorphism and its canonicity property comes from Propositions\autoref{prop:3.2} and\autoref {prop:4.4}.

We claim that the extensions of \cite{MoRe} by means of the Langlands classification, as in Theorem
\autoref{thm:4.10}, always sends $(\phi_b \hat \nu ,\epsilon) \in \Phi_\enh (G^+)$ to the unique 
irreducible quotient of $\pi_{st}(\phi_b \hat \nu, \epsilon)$. Indeed, the proof of Theorem 
\autoref{thm:4.10} states that explicitly in the cases (ii) and (iii). In case (i) it follows because
Theorem\autoref {thm:3.6} shows that
\[
I_{P^+}^{G^+}( \pi_{L^+} (\phi_b)_\epsilon \otimes \nu) =
I_{P^+}^{G^+}( \pi_{L^+} (\phi_b \hat \nu,\epsilon)) = \pi_{st}(\phi_b \hat \nu, \epsilon) .
\]
On the other hand, by construction (see the lines before Theorem\autoref {thm:3.6}), $\pi (\phi_b
\hat \nu ,\epsilon)$ is also the irreducible quotient of $\pi_{st}(\phi_b \hat \nu, \epsilon)$. 
\end{proof}

\section{Unitary groups}
\label{sec:unitary}

In this section we discuss how the setup and the statements in Sections 
\ref{sec:GSpin}--\ref{sec:Langlands} can be adjusted, so that the arguments and the results hold for
unitary groups. Most of this can be found in \cite{Moe1} and \cite[\S C]{Hei4}. 
We prefer to use the convenient description of $L$-parameters for unitary groups from \cite{GGP}.

Let $E/F$ be a separable quadratic extension. Let $V$ be a finite dimensional $E$-vector space endowed
with an Hermitian form. Recall that the unitary group $\rU(V)$ is a reductive algebraic $F$-group, an
outer form of $\GL_{\dim V}$. The classification of pure inner twists reads:
\begin{itemize}
\item When $\dim V = 2n$, there is one quasi-split group $\rU_{2n}(E/F)$ and one pure inner twist
$\rU'_{2n}(E/F)$, which is not quasi-split.
\item When $\dim V = 2n+1$, there is a quasi-split group $\rU_{2n+1}(E/F)$, associated to a Hermitian
form with discriminant 1. There is an isomorphic but different form $\rU'_{2n+1}(E/F)$, which is
associated to an Hermitian form whose discriminant is nontrivial in $F^\times / N_{E/F}(E^\times)$.
\end{itemize} 
The complex dual group of $\rU_m (E/F)$ and $\rU'_m (E/F)$ is $\GL_m (\C)$. The group \\
$\mb W_F / \mb W_E = \mr{Gal}(E/F)$ acts on $\GL_m (\C)$ by the outer automorphism
\[
A \mapsto J_m A^{-T} J_m^{-1} ,
\] 
where $-T$ denotes inverse transpose and $J_m$ is the anti-diagonal $m \times m$-matrix whose with 
on the anti-diagonal alternating 1 and -1. We use a compressed form of the Langlands dual group:
\[
{}^L \rU_m (E/F) = {}^L \rU'_m (E/F) =\GL_m (\C) \rtimes \mb W_F / \mb W_E . \vspace{2mm}
\]

\noindent
\textbf{Modifications in Paragraph \ref{par:properties}.}\\
According to \cite[Theorem 8.1]{GGP}, any $L$-parameter $\phi$ for $\rU(V)$ is determined (up to
$\rU(V)^\vee$-conjugacy) by its restriction to $\mb W_E \times SL_2 (\C)$, which we denote $\Phi_\enh$.
This $\Phi_\enh$ is a conjugate-dual representation, which means that $\Phi_\enh^\vee$ is isomorphic
to $s \cdot \Phi_\enh$ for any $s \in \mb W_F \setminus \mb W_E$. Moreover $\Phi_\enh$ is conjugate-orthogonal 
(sign +1) if $\dim V$ is odd and conjugate-symplectic (sign -1) if $\dim V$ is even. 
That provides a bijection from $\Phi (U(V))$ to the isomorphism classes of conjugate-dual
representations of $\mb W_E$ with sign $(-1)^{\dim V - 1}$. For consistency we define
sgn$(\rU(V)^\vee) = (-1)^{\dim V - 1}$.

Conversely, let a conjugate-dual $m$-dimensional representation $\Phi_\enh$ of $\mb W_E \times \SL_2 (\C)$ 
with sign $(-1)^{m-1}$ be given. Then one can determine 
\begin{equation}
\phi \colon \mb W_F \times \SL_2 (\C) \to\GL_m (\C) \rtimes \mb W_F / \mb W_E
\end{equation}
up to conjugacy by requiring that $\phi (\mb W_F \setminus \mb W_E)$ consists of elements 
$s$ (in the non-identity component) such that $s \cdot \Phi_\enh$ is equivalent with $\Phi_\enh^\vee$. 
We abbreviate this operation to $\Phi_\enh \mapsto \mr{ind}_{\mb W_E}^{\mb W_F} \Phi_\enh$.

It is natural to relate the centralizer group of $\phi$ (computed in $\rU(V)^\vee$) to a suitable 
centralizer group of $\Phi_\enh$. To this end we recall from \cite{GGP} that $\phi$ determines an 
explicit bilinear form $B_\phi$ on $\C^m$, with respect to which $\Phi_\enh$ is conjugate-dual. 
By \cite[Theorem 8.1.iii]{GGP}
\begin{align*}
& \rZ_{\rU(V)^\vee} (\phi) = \rZ_{\mr{Aut} (B_\phi)} (\Phi_\enh) ,\\
& \mr{Aut}(B_\phi) = \big\{ g \in\GL_m (\C) : B_\phi (gv,gv') = 
B_\phi (v,v') \; \forall v,v' \in \C^m \big\} .
\end{align*}
From \cite[\S 4]{GGP} one sees that $\rZ_{\mr{Aut} (B_\phi)} (\Phi_\enh)$ behaves exactly like 
$\rZ_{{G^\vee}^+_\der}(\phi)$ in the case of general spin groups. More explicitly, 
$\rZ_{\rU(V)^\vee}(\phi)$ and $Z_{\mr{Aut} (B_\phi)}(\phi)$ are given by \eqref{eq:1.3} and 
\eqref{eq:1.6}, we only have to omit the $\rS$ (for $\det = 1$) from those formulas.\\

\noindent
\textbf{Modifications in Paragraph~\ref{par:HAL}.}\\
The standard Levi subgroups of $G_n = U(V)$ are of the form
\[
L = G_{n_-} \times\GL_{n_1}(E) \times \cdots \times\GL_{n_k}(E)
\]
with $G_{n_-} = U (V')$ of the same type as $G_n$ and $\dim V - \dim V' = 2 (n_1 + \cdots + n_k)$.
Similarly
\[
{}^L L = {}^L G_{n_-} \times \mr{ind}_{\mb W_E}^{\mb W_F} \big(\GL_{n_1}(\C) \times \cdots
\times\GL_{n_k}(\C) \big) .
\]
By  Shapiro's lemma, $\Phi (L)$ is naturally in bijection with
\[
\Phi (G_{n_-}) \times \Phi (\GL_{n_1}(E) \times \cdots \times\GL_{n_k}(E)) ,
\]
which by \cite[Theorem 8.1]{GGP} can be regarded as a set of conjugacy classes of homomorphisms 
with domain $\mb W_E \times \SL_2 (\C)$. Accordingly, the centralizer of 
$\phi \in \Phi (L)$ can be computed as the centralizer of $\Phi_\enh$ in
\[
L_E^\vee := \mr{Aut}(B_\phi) \times\GL_{n_1}(\C) \times \cdots \times\GL_{n_k}(\C) .
\]
We write
\[
\mc S_\phi = \mc S_{\Phi_\enh} = \pi_0 \big( Z_{L^\vee}(\phi) \big) = 
\pi_0 \big( Z_{L_E^\vee} (\Phi_\enh) \big) .
\]
The cuspidal support \cite{AMS1} of $(\phi,\epsilon) \in \Phi (G)$ can be computed via 
\[
Z_{G^\vee} (\phi (\mb W_F)) = Z_{\mr {Aut}(B_\phi)}(\Phi_\enh (\mb W_E)) .
\] 
This implies that 
\[
\Sc (\phi,\epsilon) = \Sc \big( \mr{ind}_{\mb W_E}^{\mb W_F} \Phi_\enh, \epsilon \big) = 
\ind_{\mb W_E}^{\mb W_F} \big( \Sc (\phi,\epsilon) \big) ,
\]
where $\mr{ind}_{\mb W_E}^{\mb W_F}$ does not change the enhancements.

As a consequence, everything in Paragraph~\ref{par:HAL} can be carried out for unitary groups,
with $\Phi_\enh$ and $L_E^\vee$ instead of $\phi$ and $L^\vee$. However, the results are not always
precisely as before. We have to distinguish two cases, depending on the ramification of $\rU(V)$,
that is, the ramification of $E/F$. 

Suppose first that $E/F$ is ramified. We take a Frobenius element of $\mb W_E$ also as
Frobenius element of $\mb W_F$, and we pick a representative for $\mb W_F / \mb W_E$ in $\mb I_F$.
Then $\mr{Res}_{\mb I_E}^{\mb W_E}$ and $\mr{Res}_{\mb I_F}^{\mb W_F}$ are compatible with
$\phi \mapsto \Phi_\enh$ and $\mr{ind}_{\mb W_E}^{\mb W_F}$. Hence the calculations in Paragraph~\ref{par:HAL} produce the correct results for $\rU(V)$. We only have to remember to omit the centre
$\C^\times$ of $\GSpin (V)^\vee$ and the S for $\det = 1$, like we needed to do for symplectic groups.

Next we suppose that $E/F$ is unramified. Then $\mb I_E = \mb I_F$ and as Frobenius element of 
$\mb W_F$ we take the square of a Frobenius element of $\mb W_E$. In contrast to the ramified case,
the impact on Paragraph~\ref{par:HAL} is substantial. 

For $\tau \in \Irr (\mb W_E )_\phi^\pm$,
there is still a unique (up to isomorphism) unramified twist $\tau' = \tau \otimes \chi$ which
is conjugate-dual and not isomorphic to $\tau$. However, in contrast to before $\tau'$ and $\tau$
always have different signs \cite[Proposition 4.10.b]{SolParam}. We order $\tau,\tau'$ so that
$\ell_\tau \geq \ell_{\tau'}$ and if $\ell_\tau = \ell_{\tau'} = 0$ then $a_\tau \geq a_{\tau'}$. 

The next change occurs in \eqref{eq:1.48}, there 
\begin{align*}
T_{\mf s^\vee} & 
= T \big/ \big( \prod\nolimits_j \rZ (\GL_{n_j}(\C))_{\mr{ind}_{\mb W_E}^{\mb W_F} \phi_j} \big) \\ 
& = \prod\nolimits_\tau \big( \C^\times / \rZ (\GL_{n_j}(\C))_{\mr{ind}_{\mb W_E}^{\mb W_F} \phi_j} 
\big)^{e_\tau} = \prod\nolimits_\tau T_{\mf s^\vee ,\tau} ,
\end{align*}
with the latter two products running over $\Irr' (\mb W_E )_\phi^\pm \cup \Irr (\mb W_E )_\phi^0$.
We note that
\[
\big| \rZ (\GL_{n_j}(\C))_{\mr{ind}_{\mb W_E}^{\mb W_F} \phi_j} \big| =
2 \: | \rZ (\GL_{d_\tau} (\C))_\tau | = 2 t_\tau .
\]
In particular 
\[
X^* (T_{\mf s^\vee,\tau}) = 2 t_\tau \big( X^* (T) \cap \Q X^* (T_{\mf s^\vee,\tau}) \big) .
\]
Further \eqref{eq:1.18} becomes
\begin{align*}
J & = Z_{G^\vee} (\phi (\mb I_F)) = \prod\nolimits_\tau G^\vee_{\phi (\mb I_E),\tau} \\
& = \prod_{\tau \in \Irr' (\mb W_E )_\phi^\pm}\GL_{2 e_\tau + \ell_\tau + \ell_{\tau'}} (\C)^{t_\tau}
\times \prod_{\tau \in \Irr (\mb W_E )_\phi^0}\GL_{e_\tau} (\C)^{2 t_\tau} .
\end{align*}
As a consequence \eqref{eq:1.49} has to be modified in the cases $\tau \in \Irr (\mb W_E )_\phi^+$,
now it reads
\[
R (G^\vee_{\phi (\mb I_E),\tau} T,T ) = \left\{ \begin{array}{ll}
C_{e_\tau} & \ell_\tau + \ell_{\tau'} = 0 \\
BC_{e_\tau} & \ell_\tau + \ell_{\tau'} > 0 
\end{array} \right. .
\]
In view of the new shape of $J$, its maximal torus given in \eqref{eq:1.47} becomes
\[
T_J = \prod_\tau T_{J,\tau}^{t_\tau} = \prod_{\tau \in \Irr' (\mb W_E )_\phi^\pm} 
\big( (\C^\times )^{2 e_\tau + \ell_\tau + \ell_{\tau'}} \big)^{t_\tau}
\times \prod_{\tau \in \Irr (\mb W_E )_\phi^0} \big( (\C^\times)^{e_\tau} \big)^{2 t_\tau} . 
\]
The computation of $m_\alpha$ for $\alpha \in R(J,T)_\red$ after \eqref{eq:1.48} also changes
for unramified unitary groups. For $\tau \in \Irr (\mb W_E )_\phi^0$, the root system
$R (G^\vee_{\phi (\mb I_E),\tau} T,T )$ has $2 t_\tau$ irreducible components, all of type
$A_{e_\tau - 1}$ and permuted cyclically by $\Fr_F$. Hence $m_\alpha$ equals $2 t_\tau m'_\alpha$,
and the same argument as before shows that $m'_\alpha = 1$.

When $\tau \in \Irr' (\mb W_E )_\phi^\pm$, the root system $R (G^\vee_{\phi (\mb I_E),\tau} T,T )$ 
has $t_\tau$ irreducible components. They are of type $A_{2e_\tau + \ell_\tau + \ell_{\tau'}}$ and
$\Fr_F$ permutes them cyclically, so $m_\alpha = t_\tau m'_\alpha$. Here the computation of $m'_\alpha$
proceeds analogously to in Paragraph~\ref{par:HAL} for the cases $\tau \in \Irr (\mb W_F)_\phi^{+-}$. 
We conclude that $m_\alpha = 2t_\tau$ unless $\ell_\tau + \ell_{\tau'} = 0$ and $\alpha \in C_{e_\tau}$
is long, then $m_\alpha = t_\tau$.

From this we obtain the root systems $R_{\mf s^\vee,\tau}$ whose union is $R_{\mf s^\vee}$. 
For $\tau \in \Irr (\mb W_E )_\phi^0$ we obtain $A_{e_\tau - 1} \subset X^* (T_{\mf s^\vee,\tau})$
as before. For $\tau \in \Irr' (\mb W_E )_\phi^\pm$ with $\ell_\tau + \ell_{\tau'} > 0$ we get
$2 t_\tau B_{e_\tau} \subset 2 t_\tau X^* (T)$, which can be identified with $B_{e_\tau}$ in
$X^* (T_{\mf s^\vee,\tau})$. For $\tau  \in \Irr' (\mb W_E )_\phi^\pm$ with $\ell_\tau + \ell_{\tau'} = 0$ 
we obtain $2 t_\tau D_{e_\tau} \cup t_\tau (C_{e_\tau} \setminus D_{e_\tau})$ in $2 t_\tau X^* (T)$,
which identifies with $B_{e_\tau}$ in $X^* (T_{\mf s^\vee,\tau})$.

The root datum for the affine Hecke algebra decomposes nicely:
\[
\mc R_{\mf s^\vee} = \bigoplus\nolimits_\tau \mc R_{\mf s^\vee,\tau} = 
\bigoplus\nolimits_\tau \big( X^* (T_{\mf s^\vee,\tau}),R_{\mf s^\vee,\tau}, X_* (T_{\mf s^\vee,\tau}),
R^\vee_{\mf s^\vee} \big) .
\]
The calculation of the parameter functions $\lambda,\lambda^*$ (following the method in 
\cite[\S 3.3]{AMS3}) leads to the following modified version of Table~\ref{tab:2}:
\begin{table}[h]
\caption{Data from $\mc R_{\mf s^\vee}$ for each $\tau$}\label{tab:3}
$\begin{array}{cccccccc}
a_\tau & a_{\tau'} & X^* (T_{\mf s^\vee,\tau}) & R_{\mf s^\vee,\tau} & \lambda(\alpha) & 
\lambda(\beta) & \lambda^* (\beta) \\
\hline
0 & -1 & \Z^{e_\tau} & B_{e_\tau} & 2 t_\tau & t_\tau  & t_\tau  \\
\geq 1 & \geq -1 & \Z^{e_\tau} & B_{e_\tau} & 2 t_\tau & t_\tau (a_\tau + a_{\tau'} + 2) & 
t_\tau ( a_\tau - a_{\tau'}) \\
\multicolumn{2}{c}{\Irr (\mb W_E )^0_\phi} & \Z^{e_\tau} & A_{e_\tau - 1} & 2 t_\tau & -- & -- 
\end{array}$
\end{table}\\
Here the first line is an instance of the second line, we mention it separately because it comes from
the exceptional case $\ell_\tau + \ell_{\tau'} = 0$ discussed above. We note that in all lines of
Table~\ref{tab:3} $W (R_{\mf s^\vee, \tau})$ is the full group $W_{\mf s^\vee}$, so 
$\Gamma_{\mf s^\vee}^{(+)}$ is trivial and can be omitted from the table. \\

\noindent
\textbf{Modifications in Section~\ref{sec:Moe}.}\\
Most of the necessary adjustments, as well as a proof of Theorem\autoref {thm:1.1}.c,d for unitary groups,
can be found in \cite{Moe1}. Let us spell out the significant changes.

The Jordan blocks of a discrete series representation $\pi$ of $G = U(V)$ are based on unitary 
supercuspidal representations $\rho$ of $\GL_m (E)$. Instead of \eqref{eq:1.44}, they have to be 
conjugate-dual: $\rho \cong \bar{\rho}^\vee$, where the bar indicates composing a representation 
with the natural action of Gal$(E/F)$ on $\rU (V)$. 

Although there exist outer automorphisms of unitary groups, we should not involve them like for
$\SO (V)$ and $\GSpin (V)$, because here $G^+ = G$. Rather, we should just replace $\mr{Out}(\mc G)$ 
by the trivial group everywhere. Then all results in Section~\ref{sec:Moe} hold for unitary groups
(except Theorem\autoref {thm:1.3} which is specific for general spin groups). \\

\noindent
\textbf{Modifications in Section~\ref{sec:Hecke}.}\\
No further adjustments are needed, everything works in the above setup.
The groups $\Gamma_{\mf s}$, $\Gamma_{\mf s}^+$, $\Gamma_{\mf s^\vee}$, $\Gamma_{\mf s^\vee}^+$ are
trivial, so all considerations about those are superfluous for unitary groups. Also, as $G^+ = G$
the material in Paragraph~\ref{par:G+} becomes trivial.\\

\noindent
\textbf{Modifications in Section~\ref{sec:Langlands}.}\\
There is only one small change, when $\rU(V)$ is unramified. In the proof of Proposition\autoref {prop:4.4} 
the case $\ell_{\tau'} = 0$ can be treated just as $\ell_{\tau'} > 0$, because by Table~\ref{tab:3}
the relevant Hecke algebra has a root datum of type $B_{e_\tau}$ with parameters such that
$\lambda (\beta) = \lambda^* (\beta) > 0$.\\

\smallskip
\textbf{Acknowledgements.}
We thank the referees for their helpful reports.

%\emph{For the purpose of open access, a CC BY public copyright license is applied to any
%Author Accepted Manuscript arising from this submission.}

\end{document}